%% file: rsk-thoma.tex
\documentclass[12pt]{amsart}

\usepackage{times}
\usepackage{ifpdf}

\usepackage{amsmath, amssymb, amsthm, bm, xspace, hyperref, enumitem, subfiles}

\usepackage[usenames,dvipsnames,svgnames,table]{xcolor}

\usepackage[config, labelfont={normalsize}]{caption,subfig}

\usepackage{tikz,pgfplots}
\usetikzlibrary{patterns}
\usetikzlibrary{matrix,arrows,decorations.pathmorphing}
\usetikzlibrary{shapes.geometric}
\usetikzlibrary{decorations.markings}
\usetikzlibrary{calc}
\usetikzlibrary{fadings}

\pgfdeclarepatternformonly{north west lines D}{\pgfqpoint{-1pt}{-1pt}}{\pgfqpoint{13pt}{13pt}}{\pgfqpoint{12pt}{12pt}}%
{
  \pgfsetlinewidth{0.4pt}
  \pgfpathmoveto{\pgfqpoint{0pt}{12.1pt}}
  \pgfpathlineto{\pgfqpoint{12.1pt}{0pt}}
  \pgfusepath{stroke}
}

\pgfdeclarepatternformonly{horizontal lines A}{\pgfpointorigin}{\pgfqpoint{100pt}{6pt}}{\pgfqpoint{100pt}{6pt}}%
{
  \pgfsetlinewidth{0.4pt}
  \pgfpathmoveto{\pgfqpoint{-1pt}{1pt}}
  \pgfpathlineto{\pgfqpoint{101pt}{1pt}}
  \pgfpathmoveto{\pgfqpoint{-1pt}{3pt}}
  \pgfpathlineto{\pgfqpoint{101pt}{3pt}}
  \pgfpathmoveto{\pgfqpoint{-1pt}{5pt}}
  \pgfpathlineto{\pgfqpoint{101pt}{5pt}}
  \pgfusepath{stroke}
}

\pgfdeclarepatternformonly{horizontal lines B}{\pgfpointorigin}{\pgfqpoint{100pt}{6pt}}{\pgfqpoint{100pt}{6pt}}%
{
  \pgfsetlinewidth{0.4pt}
  \pgfpathmoveto{\pgfqpoint{-1pt}{1pt}}
  \pgfpathlineto{\pgfqpoint{101pt}{1pt}}
  \pgfpathmoveto{\pgfqpoint{-1pt}{3pt}}
  \pgfpathlineto{\pgfqpoint{101pt}{3pt}}
  \pgfusepath{stroke}
}

\pgfdeclarepatternformonly{horizontal lines C}{\pgfpointorigin}{\pgfqpoint{100pt}{6pt}}{\pgfqpoint{100pt}{6pt}}%
{
  \pgfsetlinewidth{0.4pt}
  \pgfpathmoveto{\pgfqpoint{-1pt}{1pt}}
  \pgfpathlineto{\pgfqpoint{101pt}{1pt}}
  \pgfusepath{stroke}
}

\pgfdeclarepatternformonly{vertical lines A}{\pgfpointorigin}{\pgfqpoint{6pt}{100pt}}{\pgfqpoint{6pt}{100pt}}%
{
  \pgfsetlinewidth{0.4pt}
  \pgfpathmoveto{\pgfqpoint{1pt}{-1pt}}
  \pgfpathlineto{\pgfqpoint{1pt}{101pt}}
  \pgfpathmoveto{\pgfqpoint{3pt}{-1pt}}
  \pgfpathlineto{\pgfqpoint{3pt}{101pt}}
  \pgfpathmoveto{\pgfqpoint{5pt}{-1pt}}
  \pgfpathlineto{\pgfqpoint{5pt}{101pt}}
  \pgfusepath{stroke}
}

\pgfdeclarepatternformonly{vertical lines B}{\pgfpointorigin}{\pgfqpoint{6pt}{100pt}}{\pgfqpoint{6pt}{100pt}}%
{
  \pgfsetlinewidth{0.4pt}
  \pgfpathmoveto{\pgfqpoint{1pt}{-1pt}}
  \pgfpathlineto{\pgfqpoint{1pt}{101pt}}
  \pgfpathmoveto{\pgfqpoint{3pt}{-1pt}}
  \pgfpathlineto{\pgfqpoint{3pt}{101pt}}
  \pgfusepath{stroke}
}

\pgfdeclarepatternformonly{vertical lines C}{\pgfpointorigin}{\pgfqpoint{6pt}{100pt}}{\pgfqpoint{6pt}{100pt}}%
{
  \pgfsetlinewidth{0.4pt}
  \pgfpathmoveto{\pgfqpoint{1pt}{-1pt}}
  \pgfpathlineto{\pgfqpoint{1pt}{101pt}}
  \pgfusepath{stroke}
}

\usepackage{todonotes}

\usepackage[capitalise]{cleveref}
\Crefname{figure}{Figure}{Figures}

\crefname{empty}{}{}

\newtheorem{theorem}{Theorem}[section]
\newtheorem{proposition}[theorem]{Proposition}
\newtheorem{lemma}[theorem]{Lemma}

\newtheorem{fact}[theorem]{Fact}

\theoremstyle{remark}

\newtheorem{remark}[theorem]{Remark}

\numberwithin{equation}{section}

\newcommand{\E}{\mathbb{E}}
\newcommand{\jdt}{jeu de taquin\xspace}
\newcommand{\Jdt}{Jeu de taquin\xspace}

\newcommand{\alphabet}{\mathbb{A}}

\newcommand{\alphabetinsertion}{\mathbb{I}}
\newcommand{\alphabetjeu}{\mathbb{J}}
\newcommand{\alphabetmeasure}{\mathcal{M}}
\newcommand{\alphabetmeasureinsertion}{\alphabetmeasure^\alphabetinsertion}
\newcommand{\alphabetmeasurejeu}{\alphabetmeasure^\alphabetjeu}

\newcommand{\R}{\mathbb{R}}

\newcommand{\Z}{\mathbb{Z}}

\newcommand{\allpartitions}{\mathbb{Y}}
\newcommand{\partitions}[1]{\mathbb{Y}_{#1}}
\newcommand{\tableaux}{\mathbb{T}}
\newcommand{\Sym}[1]{\mathfrak{S}_{#1}}
\newcommand{\N}{\mathbb{N}}
\newcommand{\jpath}{\mathbf{p}}
\newcommand{\jpathlazy}{\mathbf{q}}
\newcommand{\character}{\chi}
\newcommand{\measure}{\mathfrak{M}}

\newcommand{\letter}{w}
\newcommand{\Letter}{W}
\newcommand{\Word}{\mathbf{W}}

\newcommand{\isomorphismItself}{\iota}
\newcommand{\isomorphism}[1]{\isomorphismItself(#1)}
\newcommand{\PI}{\bm{\pi}}

\DeclareMathOperator{\Id}{Id}

\DeclareMathOperator{\SC}{SC}
\newcommand{\SClaw}{\mathcal{L}_{\SC}}

\newcommand{\equalindist}{\stackrel{\mathcal{D}}{=}}            % notation for equality in distribution

\DeclareMathOperator{\RSK}{RSK}

\newcommand{\todoPiotr}[2][]{\todo[backgroundcolor=green!10,linecolor=black#1]{\textsc{Piotr} says: #2}}

\title[RSK, jeu de taquin and Vershik-Kerov measures]{Robinson-Schensted-Knuth algorithm,\\ jeu de taquin and  Kerov-Vershik measures \\ on infinite tableaux}

\author{Piotr \'Sniady}
\address{Zentrum Mathematik, M5,
Technische Universit\"at M\"unchen, \linebreak
Boltzmannstrasse 3,
85748 Garching, Germany \newline \indent
Institute of Mathematics, Polish Academy of Sciences, \linebreak
\mbox{ul.~\'Sniadec\-kich 8,} 00-956 Warszawa, Poland
 \newline
\indent 
Institute of Mathematics,
University of Wroclaw,  \linebreak \mbox{pl.\ Grunwaldzki~2/4,} 50-384
Wroclaw, Poland
} 
\email{Piotr.Sniady@tum.de, Piotr.Sniady@math.uni.wroc.pl}

\keywords{asymptotic representation theory of symmetric groups, jeu de taquin, Robinson-Schensted-Knuth algorithm, Young tableau, 
Vershik-Kerov measures, Thoma characters of the infinite symmetric group, dynamical system, isomorphism of measure preserving systems}

\subjclass[2010]{
60C05   %Combinatorial probability
(Primary)
05E10,  %Combinatorial aspects of representation theory
20C30,  %Representations of finite symmetric groups
20C32, 	%Representations of infinite symmetric groups
37A05   %Measure-preserving transformations
(Secondary)
}

\begin{document}

\begin{abstract}
We investigate \emph{Robinson-Schensted-Knuth algorithm} ($\RSK$) and Sch\"utzenberger's \emph{jeu de taquin} in the infinite setup. We show that the recording tableau in $\RSK$ defines an \emph{isomorphism} of the following two dynamical systems: (i)~a sequence of i.i.d.~random letters equipped with Bernoulli shift, and (ii)~a random infinite Young tableau (with the distribution given by \emph{Vershik-Kerov measure}, corresponding to some Thoma character of the infinite symmetric group) equipped with \jdt transformation. 
As a special case we recover the results on non-colliding random walks and multidimensional Pitman transform.
\end{abstract}

\maketitle

\section{Introduction}

We start with a rather informal introduction; the formal definitions and some missing notation are postponed until \cref{sec:preliminaries}.

\subsection{Characters of the infinite symmetric groups}
% The \emph{infinite symmetric group $\Sym{\infty}$} is defined as the group of \emph{finite permutations} $\pi:\N\rightarrow\N$ of the set of natural numbers, i.e.~permutations for which $\pi(i)=i$ holds true for almost all natural numbers $i\in\N$.
% For such an infinite group it is interesting to study its \emph{characters} (suitably normalized, positive definite functions on the group) and, in particular, \emph{factor characters} (indecomposable characters which cannot be written as a non-trivial linear combination of other characters). 
% 
The notion of \emph{irreducible representations} turns out to be not very suitable in the case of infinite groups and it is more convenient to replace it by the notion of \emph{indecomposable characters} (the name \emph{extremal characters} is also frequently used). 
The \emph{indecomposable characters} of the \emph{infinite symmetric group $\Sym{\infty}$}
were classified by Thoma \cite{Thoma1964};
he showed that there is a bijective correspondence between such characters and triples $(\alpha,\beta,\gamma)$ such that
\begin{align*}
\alpha&=(\alpha_1,\alpha_2,\dots)\text{ with }\alpha_1\geq \alpha_2\geq \cdots \geq 0,\\
\beta&=(\beta_1,\beta_2,\dots)\text{ with }\beta_1\geq \beta_2\geq \cdots \geq 0,
\end{align*}
are weakly decreasing sequences of non-negative numbers and $\gamma\geq 0$ is a non-negative number such that
\[ \alpha_1+\alpha_2+\cdots+\beta_1+\beta_2+\cdots+\gamma = 1.\]
The corresponding character will be denoted by $\character_{\alpha,\beta,\gamma}$.
The set of such triples $(\alpha,\beta,\gamma)$ is called \emph{Thoma simplex}.

The meaning of the parameters in Thoma's characterization remained rather mysterious until Vershik and Kerov \cite{VershikKerov1981} related them to asymptotics of some random \emph{infinite Young tableaux}. We shall review this relationship in the following.

\subsection{Infinite Young tableaux and Vershik-Kerov measures}
\label{subsec:infinite-young-tableaux}

\begin{figure}[t]
\begin{center}
\subfloat[]{
\subfile{programy-jeu-de-taquin/grafika-infinite-tableau2.tex}
\label{fig:inf-A}
}
\hfill
\subfloat[]{
\subfile{programy-jeu-de-taquin/grafika-infinite-tableau.tex}
\label{fig:inf-B}
}
\end{center}
\caption{Simulated infinite Young tableaux, sampled according to Vershik-Kerov measure $\measure_{\alpha,\beta,\gamma}$ for two choices of the parameters: \protect\subref{fig:inf-A} $\alpha=(0.1,0.1,0.1,0,0,\dots)$, $\beta=(0.5,0.2,0,0,\dots)$, $\gamma=0$;
 \protect\subref{fig:inf-B} $\alpha=(0,0,\dots)$, $\beta=(0.5,0,0,\dots)$, $\gamma=0.5$.
}
\label{fig:infinite-tableau}
\end{figure}

\begin{figure}[tb]
 
\begin{center}
\begin{tikzpicture}[scale=0.3]

\coordinate (d0) at (-6,0);
\draw (d0) node {$\bm{\emptyset}$};

% \draw[ultra thin,gray] (-3,-10) grid (10,10);

% % % % %  1 box
\coordinate (d1) at (0,0);
\draw[thick] (d1)  ++(-0.5,-0.5) rectangle +(1,1);
% \only<2>
{\draw[ultra thick] (d1)  ++(-0.5,-0.5) rectangle +(1,1);}

%  youngdiagram[{1}, {0, 0.5}], 

% % % % % 2 boxes
\coordinate (d2) at (6,-2.5);
\draw[thick] (d2) ++(-1,-0.5) 
rectangle +(1,1) 
{[current point is local] ++(1,0) rectangle +(1,1)};
% \only<2>
{\draw[ultra thick] (d2) ++(-1,-0.5) 
rectangle +(1,1) 
{[current point is local] ++(1,0) rectangle +(1,1)};
}
% youngdiagram[{2}, {4, -2.5}],

\coordinate (d11) at (6,3);
\draw[thick] (d11) ++(-0.5,-1) 
rectangle +(1,1) 
{[current point is local] ++(0,1) rectangle +(1,1)};
%  youngdiagram[{1, 1}, {4, 3}], youngdiagram[{3}, {8, -5.5}],

% % % % % % 3 boxes
\coordinate (d3) at (12,-5);
\draw[thick] (d3) ++(-1.5,-0.5) 
rectangle +(1,1) 
{[current point is local] ++(1,0) rectangle +(1,1)}
{[current point is local] ++(2,0) rectangle +(1,1)};
% youngdiagram[{3}, {8, -5.5}],

\coordinate (d21) at (12,0);
\draw[thick] (d21) ++(-1,-1) 
rectangle +(1,1) 
{[current point is local] ++(0,1) rectangle +(1,1)}
{[current point is local] ++(1,0) rectangle +(1,1)};
% \only<2>
{\draw[ultra thick] (d21) ++(-1,-1) 
rectangle +(1,1) 
{[current point is local] ++(0,1) rectangle +(1,1)}
{[current point is local] ++(1,0) rectangle +(1,1)};
}
%  youngdiagram[{2, 1}, {8, 0}], youngdiagram[{1, 1, 1}, {8, 6}],

\coordinate (d111) at (12,6);
\draw[thick] (d111) ++(-0.5,-1.5) 
rectangle +(1,1) 
{[current point is local] ++(0,1) rectangle +(1,1)}
{[current point is local] ++(0,2) rectangle +(1,1)};
% youngdiagram[{1, 1, 1}, {8, 6}],

% % % % % %  4 boxes
\coordinate (d4) at (20,-7);
\draw[thick] (d4) ++(-2,-0.5) 
rectangle +(1,1) 
{[current point is local] ++(1,0) rectangle +(1,1)}
{[current point is local] ++(2,0) rectangle +(1,1)}
{[current point is local] ++(3,0) rectangle +(1,1)};
%  youngdiagram[{4}, {14, -7}], 

\coordinate (d31) at (20,-4);
\draw[thick] (d31) ++(-1.5,-1) 
rectangle +(1,1) 
{[current point is local] ++(0,1) rectangle +(1,1)}
{[current point is local] ++(1,0) rectangle +(1,1)}
{[current point is local] ++(2,0) rectangle +(1,1)};
% youngdiagram[{3, 1}, {14, -4}],

\coordinate (d22) at (20,0);
\draw[thick] (d22) ++(-1,-1) 
rectangle +(1,1) 
{[current point is local] ++(0,1) rectangle +(1,1)}
{[current point is local] ++(1,0) rectangle +(1,1)}
{[current point is local] ++(1,1) rectangle +(1,1)};
%  youngdiagram[{2, 2}, {14, 0}], 

\coordinate (d211) at (20,4);
\draw[thick] (d211) ++(-1,-1.5) 
rectangle +(1,1) 
{[current point is local] ++(0,1) rectangle +(1,1)}
{[current point is local] ++(1,0) rectangle +(1,1)}
{[current point is local] ++(0,2) rectangle +(1,1)};
% \only<2>
{\draw[ultra thick] (d211) ++(-1,-1.5) 
rectangle +(1,1) 
{[current point is local] ++(0,1) rectangle +(1,1)}
{[current point is local] ++(1,0) rectangle +(1,1)}
{[current point is local] ++(0,2) rectangle +(1,1)};
}
% youngdiagram[{2, 1, 1}, {14, 4}],

\coordinate (d1111) at (20,8.5);
\draw[thick] (d1111) ++(-0.5,-2) 
rectangle +(1,1) 
{[current point is local] ++(0,1) rectangle +(1,1)}
{[current point is local] ++(0,2) rectangle +(1,1)}
{[current point is local] ++(0,3) rectangle +(1,1)};
%  youngdiagram[{1, 1, 1, 1}, {14, 8.5}]},

\draw[->,ultra thick] ($ (d0)!1.5 cm!(d1) $) -- ($ (d1)!1.5 cm!(d0) $);

\draw[->] ($ (d1)!1.5 cm!(d2) $) -- ($ (d2) !1.5 cm! (d1)$) ;
% \only<2>
{\draw[->,ultra thick] ($ (d1)!1.5 cm!(d2) $) -- ($ (d2) !1.5 cm! (d1)$) ;}
\draw[->] ($ (d1)!1.5 cm!(d11) $) -- ($ (d11) !1.5 cm! (d1)$) ;

\draw[->] ($ (d2)!1.5 cm!(d3) $) -- ($ (d3) !2 cm! (d2)$) ;
\draw[->] ($ (d2)!1.5 cm!(d21) $) -- ($ (d21) !1.5 cm! (d2)$) ;
% \only<2>
{\draw[->,ultra thick] ($ (d2)!1.5 cm!(d21) $) -- ($ (d21) !1.5 cm! (d2)$) ;}
\draw[->] ($ (d11)!1.5 cm!(d21) $) -- ($ (d21) !1.5 cm! (d11)$) ;
\draw[->] ($ (d11)!1.5 cm!(d111) $) -- ($ (d111) !1.5 cm! (d11)$) ;

\draw[->] ($ (d3)!2 cm!(d4) $) -- ($ (d4) !2.5 cm! (d3)$) ;
\draw[->] ($ (d3)!2 cm!(d31) $) -- ($ (d31) !2.5 cm! (d3)$) ;
\draw[->] ($ (d21)!1.5 cm!(d31) $) -- ($ (d31) !2.5 cm! (d21)$) ;
\draw[->] ($ (d21)!1.5 cm!(d22) $) -- ($ (d22) !1.5 cm! (d21)$) ;
\draw[->] ($ (d21)!1.5 cm!(d211) $) -- ($ (d211) !1.5 cm! (d21)$) ;
% \only<2>
{\draw[->,ultra thick] ($ (d21)!1.5 cm!(d211) $) -- ($ (d211) !1.5 cm! (d21)$) ;}
\draw[->] ($ (d111)!1.5 cm!(d211) $) -- ($ (d211) !1.5 cm! (d111)$) ;
\draw[->] ($ (d111)!1.5 cm!(d1111) $) -- ($ (d1111) !1.5 cm! (d111)$) ;

\node (infty) at (28,5) {$\cdots$};
\draw[->] ($ (d211) !1.5 cm! (infty) $) -- ($ (infty) ! 2cm ! (d211) $) ;
% \only<2>
{\draw[->,ultra thick] ($ (d211) !1.5 cm! (infty) $) -- ($ (infty) ! 2cm ! (d211) $) ;}

% , {1.2, 1.7}],

\end{tikzpicture}
\end{center}
\caption{The Young graph. The highlighted diagrams form a path corresponding to the infinite Young tableau from \protect\cref{fig:inf-A}.}
\label{fig:young-graph}
\end{figure}
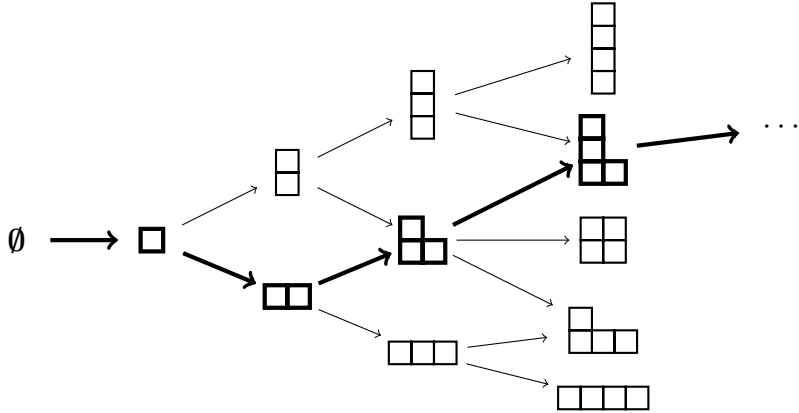

Vershik and Kerov \cite{VershikKerov1981} noticed that there is a natural bijective correspondence between the
indecomposable characters of the infinite symmetric group $\Sym{\infty}$ and \emph{indecomposable central measures} on the set $\tableaux$ of \emph{infinite Young tableaux};
thus Thoma's classification is equivalent to studying properties 
of some \emph{random infinite Young tableaux} (see \cref{fig:infinite-tableau}). 
These indecomposable central measures are in the focus of the current paper.
The measure on $\tableaux$ corresponding to the character $\character_{\alpha,\beta,\gamma}$ will be denoted by $\measure_{\alpha,\beta,\gamma}$; we will call it \emph{Vershik-Kerov measure}.

Any infinite Young tableau $t\in\tableaux$ can be alternatively viewed as an infinite path
$(\emptyset=\lambda^0\nearrow\lambda^1\nearrow\cdots)$ in Young graph, see \cref{fig:young-graph}. In the same paper \cite{VershikKerov1981}, Vershik and Kerov found a
beautiful interpretation of the parameters $\alpha$ and $\beta$ in Thoma simplex as \emph{asymptotic frequencies} of boxes appearing in appropriate rows and columns in such a sequence $(\lambda^0\nearrow\lambda^1\nearrow\cdots)$ of Young diagrams. We postpone the details of this result and we provide it as \cref{theo:VK-frequencies}.

\subsection{Generalized $\RSK$ algorithm}
Usually, a \emph{semistandard tableau} --- or, shortly, \emph{tableau} --- is defined as a filling of the boxes of a Young diagram (with the letters from some alphabet) in such a way that the rows and columns are, \emph{roughly speaking}, increasing.
This definition creates no difficulties as long as we consider tableaux in which the entries do not repeat. If the entries repeat, the traditional approach is to require that each row should be \emph{weakly increasing} and each column \emph{strongly increasing}; in other words a letter can appear several times in one row and can appear at most once in one column.

Kerov and Vershik \cite{KerovVershik1986} took a different approach: they declared that each letter of the alphabet can be either a \emph{row letter} (such a letter can appear several times in a row but can appear at most once in a column) or a \emph{column letter} (such a letter can appear several times in a column but can appear at most once in a row).
They also described how \emph{Robinson-Schensted-Knuth algorithm} ($\RSK$) can be adapted to this more general setup; we recall this construction in \cref{sec:generalized-schensted}. 
As we shall see below, this generalization was essential in order to give a new interpretation of the parameters of Thoma simplex.

This generalization of $\RSK$ appeared also in the work of Berele and Remmel
\cite{BereleRemmel1985} as well as Berele and Regev \cite{RereleRegev1987}, but only in a special case of finite alphabets in which 
any row letter is smaller than any column letter which is not sufficient for our purposes.

\subsection{$\RSK$ and Vershik-Kerov measures}
\label{subsec:rsk-and-central}

\subsubsection{$\RSK$ as a homomorphism}
In the usual setup, Robinson-Schensted-Knuth algorithm applied to a 
\emph{finite} sequence gives as an output a \emph{pair} of tableaux, namely the 
\emph{insertion tableau} and the \emph{recording tableau}.
Kerov and Vershik \cite{KerovVershik1986} applied Robinson-Schensted-Knuth algorithm to an \emph{infinite} sequence $(\letter_1,\letter_2,\dots) $
of letters from an arbitrary alphabet $\alphabet$ which consists of row letters and column letters. 
In this infinite setup the notion of the insertion tableau does not make sense and the outcome of
Robinson-Schensted-Knuth algorithm $\RSK(\letter_1,\letter_2,\dots)\in\tableaux$
is defined as just the \emph{recording tableau} (which is an \emph{infinite Young tableau}, see \cref{fig:infinite-tableau}). 

Kerov and Vershik \cite{KerovVershik1986} proved that if $(\Letter_1,\Letter_2,\dots)$ is a sequence of random, independent, identically distributed letters with the distribution $\alphabetmeasure$, then the distribution of the random infinite Young tableau $\RSK(\Letter_1,\Letter_2,\dots)$ coincides with the \emph{indecomposable central measure $\measure_{\alpha,\beta,\gamma}$ (Vershik-Kerov measure)} corresponding to some element $(\alpha,\beta,\gamma)$ of Thoma simplex given as follows:
$\alpha_1\geq \alpha_2\geq \cdots$ are the probabilities of the atoms of the measure $\alphabetmeasure$ on the row letters;
$\beta_1\geq \beta_2\geq \cdots$ are the probabilities of the atoms of the measure $\alphabetmeasure$ on the column letters; 
$\gamma$ is the total probability of the continuous part of $\alphabetmeasure$.
We present this result in full detail in \cref{theo:KerovVershik-RSK-homomorphism}.
This result gives another interpretation of the parameters in Thoma simplex as probabilities of atoms of the measure $\alphabetmeasure$ on the alphabet $\alphabet$.

Notice that the extension of $\RSK$ algorithm to \emph{column letters} was essential in order to recover all elements of Thoma's simplex. It is worth pointing out that the above result of Kerov and Vershik \cite{KerovVershik1986} --- contrary to the results presented in the current paper --- holds in general and does not require any additional assumptions on the alphabet $\alphabet$ and the probability distribution $\alphabetmeasure$ of the letters.

The above result of Kerov and Vershik provides a very concrete realization (or, viewed alternatively, an equivalent definition) of all indecomposable central measures on the set $\tableaux$ of infinite Young tableaux.
In other words: $\RSK$ provides a convenient \emph{homomorphism} between the following two probability spaces: from (i) the very simple \emph{product space} $ (\alphabet^\N, \mathcal{B},  \alphabetmeasure^\N)$ (i.e., i.i.d.~letters), to (ii) the probability space $(\tableaux,\mathcal{F},\measure_{\alpha,\beta,\gamma})$  of infinite Young tableaux equipped with some Vershik-Kerov measure. 
The original paper Kerov and Vershik \cite{KerovVershik1986} presents some applications of this homomorphism.

\subsubsection{$\RSK$ as an isomorphism}

It is a natural to ask if this homomorphism is, in fact, an \emph{isomorphism}. In the current paper we give a positive answer to this question under additional assumptions about the structure of the alphabet and the probability measure on it.
Namely, for a special choice of the alphabet $\alphabetjeu$ (\emph{the  \jdt alphabet}) which can be informally visualized as
\begin{equation}
\label{eq:jdt-alphabet}
\underbrace{1 < 2 < 3 < \cdots}_{\text{row letters}}< \cdots < 0.1 < \cdots< 0.9 < \cdots <\underbrace{\cdots < -3 < -2 < -1}_{\text{column letters}},
\end{equation}
and with a special choice of the probability distribution $\alphabetmeasurejeu_{\alpha,\beta,\gamma}$ on $\alphabetjeu$ (for details see \cref{subsec:alphabetjeu})
the following result holds true.

\begin{theorem}[$\RSK$ is an isomorphism of probability spaces]
\label[theorem]{theorem:isomorphism}
Let $(\alpha,\beta,\gamma)$ be an element of Thoma simplex. Then
$\RSK$ is an \emph{isomorphism} between the following two probability spaces:
\begin{itemize}
 \item $\big(\alphabetjeu^\N,\mathcal{B},(\alphabetmeasurejeu_{\alpha,\beta,\gamma})^\N\big)$, i.e., a sequence of i.i.d.~random letters of the \jdt alphabet with the distribution $\alphabetmeasurejeu_{\alpha,\beta,\gamma}$;
 \item $(\tableaux,\mathcal{F},\measure_{\alpha,\beta,\gamma})$, i.e., infinite Young tableaux with Vershik-Kerov measure $\measure_{\alpha,\beta,\gamma}$.
\end{itemize}

\end{theorem}

Here and through the whole paper the symbol $\mathcal{B}$ will refer to the product $\sigma$-algebra on appropriate product space. The $\sigma$-algebra $\mathcal{F}$ on $\tableaux$ will be defined in \cref{subsec:infinite-young-tableaux-revisited}.

So, it is natural to ask \emph{what is the inverse to this isomorphism?}
In order to answer this question we will have to study \emph{jeu de taquin} for infinite Young tableaux.

% In the following we will show that this isomorphism is, in fact, not only an isomorphism of \emph{measure spaces} but also an isomorphism of some natural \emph{dynamical systems} (\cref{theorem:isomorphism-dynamical}).
% For this purpose we will need to study \emph{jeu de taquin} for infinite Young tableaux.

\subsection{\Jdt}
\label{subsec:jdt}

\emph{\Jdt} (literally, \emph{teasing game}) was introduced by Sch\"utzenberger \cite{schutzenberger} for finite (semistandard) tableaux. It turned out to be a powerful tool of algebraic combinatorics, in particular for problems related to the representation theory of symmetric groups and Robinson-Schensted-Knuth algorithm.
In our previous paper \cite{RomikSniady2011} we investigated a generalization of \jdt to the setup of \emph{infinite Young tableaux}. We will recall it briefly.

\begin{figure}[tb]
\begin{center}
\subfloat[]{
\subfile{programy-jeu-de-taquin/grafika-jdt-boxes1.tex}
\label{fig-jdt-pathA}
}
\qquad
\subfloat[]{
\subfile{programy-jeu-de-taquin/grafika-jdt-boxes2.tex}
\label{fig-jdt-pathB}
} 
\end{center}

\caption{\protect\subref{fig-jdt-pathA} A part of an infinite Young tableau $t$. The highlighted boxes form the beginning of the \jdt path $\jpath(t)$. \protect\subref{fig-jdt-pathB} The outcome of sliding of the boxes along the highlighted \jdt path. The outcome of the \jdt transformation $J(t)$ is obtained by subtracting $1$ from every entry.}
\label{fig-jdt-path}
\end{figure}

\newcommand{\letterX}{r}
\newcommand{\letterY}{s}
\begin{figure}
\subfloat[]{
\begin{tikzpicture}[scale=1.2]
\clip (-0.3,-0.3) rectangle (2.3,2.3);
\draw[black!20] (-1,-1) grid (3,3);
 \draw[pattern color=red!50,pattern=north east lines] (0,1) +(0.1,0.1) rectangle +(0.9,0.9);
 \draw[pattern color=green!50,pattern=north west lines] (1,0) +(0.1,0.1) rectangle +(0.9,0.9);
 \draw[fill=black!5] (1,1) +(0.1,0.1) rectangle +(0.9,0.9);
 \draw[fill=black!5] (1,2) +(0.1,0.1) rectangle +(0.9,0.9);
 \draw[fill=black!5] (2,1) +(0.1,0.1) rectangle +(0.9,0.9);
 \draw[fill=black!5] (2,0) +(0.1,0.1) rectangle +(0.9,0.9);
 \draw[fill=black!5] (0,2) +(0.1,0.1) rectangle +(0.9,0.9);
 \draw[fill=black!5] (-1,0) +(0.1,0.1) rectangle +(0.9,0.9);
 \draw[fill=black!5] (0,-1) +(0.1,0.1) rectangle +(0.9,0.9);
 \draw[fill=black!5] (-1,1) +(0.1,0.1) rectangle +(0.9,0.9);
 \draw[fill=black!5] (1,-1) +(0.1,0.1) rectangle +(0.9,0.9);
 \draw[fill=black!5] (-1,2) +(0.1,0.1) rectangle +(0.9,0.9);
 \draw[fill=black!5] (2,-1) +(0.1,0.1) rectangle +(0.9,0.9);
 \draw[fill=black!5] (-1,-1) +(0.1,0.1) rectangle +(0.9,0.9);
 \draw[fill=black!5] (2,2) +(0.1,0.1) rectangle +(0.9,0.9);
 \draw (0.5,1.5) node [circle,inner sep=2pt,fill=white] {$\letterX$};
 \draw (1.5,0.5) node [circle,inner sep=2pt,fill=white] {$\letterY$};
\end{tikzpicture}
\label{subsig:jdtA}}
\hfill
\subfloat[]{
\begin{tikzpicture}[scale=1.2]
\clip (-0.3,-0.3) rectangle (2.3,2.3);
\draw[black!20] (-1,-1) grid (3,3);
 \draw[pattern color=red!50,pattern=north east lines] (0,0) +(0.1,0.1) rectangle +(0.9,0.9);
 \draw[pattern color=green!50,pattern=north west lines] (1,0) +(0.1,0.1) rectangle +(0.9,0.9);
 \draw[fill=black!5] (1,1) +(0.1,0.1) rectangle +(0.9,0.9);
 \draw[fill=black!5] (1,2) +(0.1,0.1) rectangle +(0.9,0.9);
 \draw[fill=black!5] (2,1) +(0.1,0.1) rectangle +(0.9,0.9);
 \draw[fill=black!5] (2,0) +(0.1,0.1) rectangle +(0.9,0.9);
 \draw[fill=black!5] (0,2) +(0.1,0.1) rectangle +(0.9,0.9);
 \draw[fill=black!5] (-1,0) +(0.1,0.1) rectangle +(0.9,0.9);
 \draw[fill=black!5] (0,-1) +(0.1,0.1) rectangle +(0.9,0.9);
 \draw[fill=black!5] (-1,1) +(0.1,0.1) rectangle +(0.9,0.9);
 \draw[fill=black!5] (1,-1) +(0.1,0.1) rectangle +(0.9,0.9);
 \draw[fill=black!5] (-1,2) +(0.1,0.1) rectangle +(0.9,0.9);
 \draw[fill=black!5] (2,-1) +(0.1,0.1) rectangle +(0.9,0.9);
 \draw[fill=black!5] (-1,-1) +(0.1,0.1) rectangle +(0.9,0.9);
 \draw[fill=black!5] (2,2) +(0.1,0.1) rectangle +(0.9,0.9);
 \draw (0.5,0.5) node [circle,inner sep=2pt,fill=white] {$\letterX$};
 \draw (1.5,0.5) node [circle,inner sep=2pt,fill=white] {$\letterY$};
 \draw[ultra thick,->] (0.5,1.5) -- (0.5,0.7);
\end{tikzpicture}
\label{subsig:jdtB}}
\hfill
\subfloat[]{
\begin{tikzpicture}[scale=1.2]
\clip (-0.3,-0.3) rectangle (2.3,2.3);
\draw[black!20] (-1,-1) grid (3,3);
 \draw[pattern color=red!50,pattern=north east lines] (0,1) +(0.1,0.1) rectangle +(0.9,0.9);
 \draw[pattern color=green!50,pattern=north west lines] (0,0) +(0.1,0.1) rectangle +(0.9,0.9);
 \draw[fill=black!5] (1,1) +(0.1,0.1) rectangle +(0.9,0.9);
 \draw[fill=black!5] (1,2) +(0.1,0.1) rectangle +(0.9,0.9);
 \draw[fill=black!5] (2,1) +(0.1,0.1) rectangle +(0.9,0.9);
 \draw[fill=black!5] (2,0) +(0.1,0.1) rectangle +(0.9,0.9);
 \draw[fill=black!5] (0,2) +(0.1,0.1) rectangle +(0.9,0.9);
 \draw[fill=black!5] (-1,0) +(0.1,0.1) rectangle +(0.9,0.9);
 \draw[fill=black!5] (0,-1) +(0.1,0.1) rectangle +(0.9,0.9);
 \draw[fill=black!5] (-1,1) +(0.1,0.1) rectangle +(0.9,0.9);
 \draw[fill=black!5] (1,-1) +(0.1,0.1) rectangle +(0.9,0.9);
 \draw[fill=black!5] (-1,2) +(0.1,0.1) rectangle +(0.9,0.9);
 \draw[fill=black!5] (2,-1) +(0.1,0.1) rectangle +(0.9,0.9);
 \draw[fill=black!5] (-1,-1) +(0.1,0.1) rectangle +(0.9,0.9);
 \draw[fill=black!5] (2,2) +(0.1,0.1) rectangle +(0.9,0.9);
 \draw (0.5,1.5) node[circle,inner sep=2pt,fill=white] {$\letterX$};
 \draw (0.5,0.5) node[circle,inner sep=2pt,fill=white] {$\letterY$};
 \draw[ultra thick,->] (1.5,0.5) -- (0.7,0.5);
\end{tikzpicture}
\label{subsig:jdtC}}
\caption{Elementary step of the \jdt transformation: 
\protect\subref{subsig:jdtA} the initial configuration of boxes,
\protect\subref{subsig:jdtB} the outcome of the slide in the case when $\letterX<\letterY$,
%(more generally, when $x<_r y$),
\protect\subref{subsig:jdtC} the outcome of the slide in the case when $\letterY<\letterX$.}
% (more generally, when $y<_c x$).}
\label{fig:jdt}
\end{figure}

Consider an \emph{infinite Young tableau} $t\in\tableaux$, see \cref{fig-jdt-pathA}.
We remove the bottom-left corner box (the box which contains the number $1$); in this way an empty space is created.
We start sliding the boxes according to the rules presented in \cref{fig:jdt}, i.e., we always slide one of the following two boxes: the one on the right or the one on the top of the empty space, always choosing the box which has smaller contents. 
As we continue sliding, the empty space keeps moving to the top or to the right, see \cref{fig-jdt-pathB}.

The outcome of \jdt is twofold. Firstly, it is the path of the empty space $\jpath(t)=\big( \jpath_1(t), \jpath_2(t), \dots\big)$, which will be called \emph{\jdt path}. (A careful reader might object that for some tableaux the \jdt path is a \emph{finite} sequence, see \cref{fig:finite-path}. We will show in \cref{theo:paths} that in the cases of our interest this is not the case.)

\begin{figure}[tb]
\centering
\subfile{programy-jeu-de-taquin/grafika-jdt-boxes-deadend1.tex}
\caption{Example of an infinite Young tableau $t$ for which the corresponding \jdt path $\jpath(t)=\big( \jpath_1(t),\dots,\jpath_\ell(t) \big)$ is finite. \cref{theo:paths} shows that this cannot happen for the random Young tableaux considered in the current paper.}
\label{fig:finite-path}
\end{figure}

Secondly, after performing all slides of \jdt, we obtain an object which
looks almost like an infinite Young tableau (see \cref{fig-jdt-pathB}) except that the numbering of boxes starts with $2$ instead of $1$. Let us subtract $1$ from every entry of this ``tableau''; the outcome is a true infinite Young tableau which we denote by $J(t)$. The map $t\mapsto J(t)$ will be called \emph{\jdt transformation}.

% Very informally speaking, \jdt answers the question \emph{how ``irreducible'' representations of the infinite symmetric group $\Sym{\{1,2,\dots\}}$, when restricted to the subgroup $\Sym{\{2,3,\dots\}}$ are decomposed into ``irreducible'' components.} 

These two outcomes of \jdt are in the focus of the current paper. 
In the following we will discuss them in more detail.

\subsection{The dynamical system of \jdt}

As we just mentioned,
one of the outcomes of \jdt applied to an infinite tableau $t\in\tableaux$ is another infinite tableau $J(t)\in\tableaux$.
This setup naturally raises questions about the iterations of the \jdt map
\[ t, \ J(t),\ J\big(J(t)\big), \dots \]
or, in other words, about the \emph{dynamical system of \jdt.} More precisely, we consider the set $\tableaux$ of infinite Young tableaux equipped with some Vershik-Kerov measure $\measure_{\alpha,\beta,\gamma}$, thus we consider the \emph{measure-preserving dynamical system} $(\tableaux,\mathcal{F},\measure_{\alpha,\beta,\gamma},J)$.
%  which will be called \emph{\jdt dynamical system}. 
Some basic properties of this dynamical system are summarized by the following theorem.

\begin{theorem}
\label[theorem]{theo:measure-preserving-and-ergodic}
\Jdt transformation $J:\tableaux\rightarrow\tableaux$ on the probability space $(\tableaux,\mathcal{F}, \measure_{\alpha,\beta,\gamma})$ of infinite Young tableaux equipped with an arbitrary Vershik-Kerov measure is 
\begin{itemize}
\item measure preserving,
\item ergodic (i.e., every measurable set $E\in\mathcal{F}$ which is $J$-invariant fulfills $\measure_{\alpha,\beta,\gamma}(E)\in\{0,1\}$; for an introduction to the ergodic theory see \cite{Silva2008}).
\end{itemize}
\end{theorem}

The following extension of \cref{theorem:isomorphism} holds true. 

\begin{theorem}[$\RSK$ is an isomorphism of dynamical systems]
\label[theorem]{theorem:isomorphism-dynamical}
Let $(\alpha,\beta,\gamma)$ be an element of Thoma simplex. Then
$\RSK$ is an \emph{isomorphism} of the following dynamical systems:
\begin{itemize}
 \item $\big(\alphabetjeu^\N,\mathcal{B},(\alphabetmeasurejeu_{\alpha,\beta,\gamma})^\N,S\big)$, i.e., a sequence of i.i.d.~random letters from the \jdt alphabet, equipped with Bernoulli shift $S:\alphabetjeu^\N \rightarrow \alphabetjeu^\N$, defined by
 \[ S(\letter_1,\letter_2,\dots):= (\letter_2,\letter_3,\dots);\]

 \item $(\tableaux,\mathcal{F},\measure_{\alpha,\beta,\gamma},J)$, i.e., infinite Young tableaux with Vershik-Kerov measure $\measure_{\alpha,\beta,\gamma}$ equipped with \jdt transformation.
\end{itemize}

\end{theorem}

In the following we will show explicitly the inverse map to this isomorphism. In order to do this we will have to investigate \jdt paths.

\subsection{Asymptotes of \jdt paths}

For a box $\Box=(x,y)$ of a tableau we denote by $x(\Box):=x$ the index of column of the box and by $y(\Box):=y$ the index of row of the box (the numbering of rows and columns starts with $1$).

As we already mentioned, one of the outcomes of \jdt applied to an infinite tableau $t\in\tableaux$ is the \jdt path $\jpath(t)$. The following theorem describes the asymptotic behavior of \jdt paths on random tableaux.

\begin{theorem}[Asymptotics of a \jdt path]
\label[theorem]{theo:paths}
Let $T$ be a random infinite Young tableau distributed according to some Vershik-Kerov measure $\measure_{\alpha,\beta,\gamma}$. 

Then, almost surely, \jdt path $\jpath(T)=\big( \jpath_1(T), \jpath_2(T),\dots \big)$ is an \emph{infinite} sequence (i.e., the situation from \cref{fig:finite-path} is not possible).

Furthermore, almost surely, exactly one of the following three events holds true.
\begin{enumerate}[label=\emph{(\Alph*)}]
 \item \label[empty]{enum:A}
\emph{The path stabilizes in some row $k$}; in other words $y\big(\jpath_i(T)\big)=k$ holds true for almost all $i$. This event happens with probability $\alpha_k$.
 \item \label[empty]{enum:B}
\emph{The path stabilizes in some column $k$}; in other words $x\big(\jpath_i(T)\big)=k$ holds true for almost all $i$. This event happens with probability $\beta_k$.
 \item \label[empty]{enum:C}
\emph{The path has some asymptotic slope}; 
in other words the limit 
\[ \lim_{i\to\infty} \frac{\jpath_i(T)}{\Vert \jpath_i(T) \Vert}\]
exists and thus is equal to $\big(\cos \Theta(T), \sin \Theta(T) \big)$ for some $0 < \Theta(T) < \frac{\pi}{2}$. 
This event happens with probability $\gamma$.
\end{enumerate}
\end{theorem}

\begin{figure}[tbp]
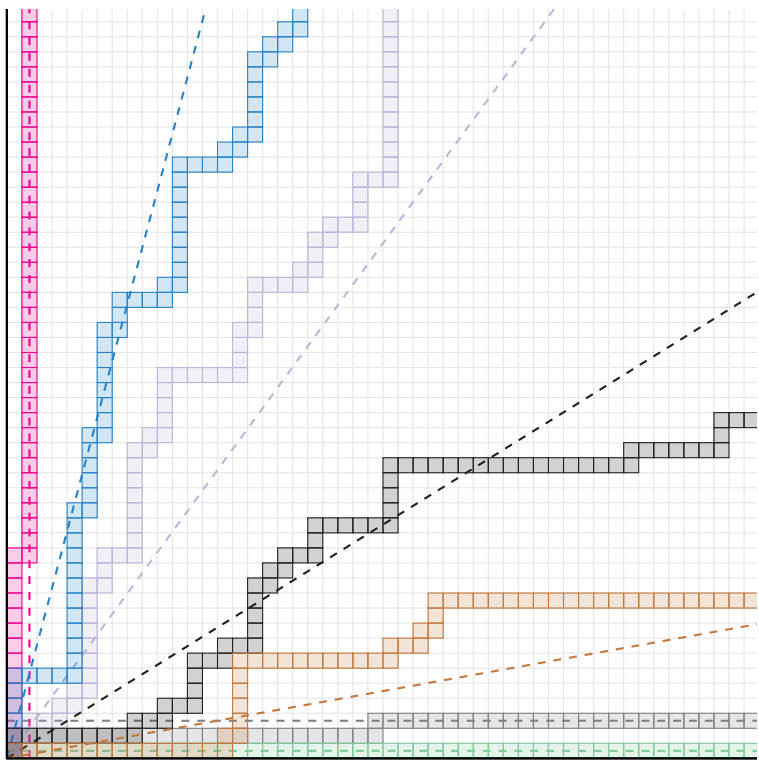

\centering
\subfile{programy-jeu-de-taquin/grafika-jdt.tex}
\caption{Simulated \jdt paths and their asymptotes (dashed lines).
Horizontal asymptotes correspond to case \ref{enum:A}, vertical asymptotes correspond to case \ref{enum:B}, sloped asymptotes correspond to case \ref{enum:C} of \cref{theo:paths}.}
\label{fig:simulated}
\end{figure}

\begin{figure}
\centering
\begin{tikzpicture}[scale=0.7]
\tiny
  \begin{scope}
    \clip (-1,-1) rectangle (9.7,9.7);

%     \draw[scale=20] (0,0) -- (0.213626375754752, 2.96182168028514);
%     \draw[scale=20] (0,0) -- (0.213626375754752, 2.96182168028514);
%     \draw[scale=20] (0,0) -- (0.446890612333807, 2.41433794338881);
%     \draw[scale=20] (0,0) -- (0.699492716817427, 1.97825875597953);
    \draw[scale=20,dashed,ultra thick] (0,0) -- (0.973639796553168, 1.60458457175326);
    \draw[scale=20,dashed,ultra thick] (0,0) -- (1.27323954473508, 1.27323954473508);
    \draw[scale=20,dashed,ultra thick] (0,0) -- (1.60458457175326, 0.973639796553168);
%     \draw[scale=20] (0,0) -- (1.97825875597953, 0.699492716817428);
%     \draw[scale=20] (0,0) -- (2.41433794338881, 0.446890612333807);
%     \draw[scale=20] (0,0) -- (2.96182168028514, 0.213626375754752);
%     \fill[white] (0,0) circle (7);

    \draw[dashed,ultra thick,red] (0,0.5) -- +(20,0); 
    \draw[dashed,ultra thick,red] (0,1.5) -- +(20,0); 
    \draw[dashed,ultra thick,red] (0,2.5) -- +(20,0); 
    \draw[dashed,ultra thick,blue] (0.5,0) -- +(0,20); 
    \draw[dashed,ultra thick,blue] (1.5,0) -- +(0,20); 
    \draw[dashed,ultra thick,blue] (2.5,0) -- +(0,20); 
    \fill[white] (0,0) circle (7);
    \draw[black!20] (0,0) grid (20,20);
    \draw[thick] (20,0) -- (0,0) -- (0,20);
    \end{scope}
    \draw (6.2,9.7) node[anchor=south] {$0.6$};
    \draw (9.7,9.7) node[anchor=south west] {$0.5$};
    \draw (9.7,6.2) node[anchor=west] {$0.4$};

    \draw (9.7,0.5) node[anchor=west] {$1$};
    \draw (9.7,1.5) node[anchor=west] {$2$};
    \draw (9.7,2.5) node[anchor=west] {$3$};
    \draw (0.5,9.7) node[anchor=south] {$-1$};
    \draw (1.5,9.7) node[anchor=south] {$-2$};
    \draw (2.5,9.7) node[anchor=south] {$-3$};

\end{tikzpicture}
\caption{Possible asymptotes for a \emph{\jdt path} and the corresponding values 
(elements of the alphabet $\alphabetjeu$) of the function $\Psi$. For details see \cref{theo:paths}.}
\label{fig:compactification}

\vspace{3ex}

\begin{tikzpicture}[scale=0.7]
\tiny
    \begin{scope}
    \clip (-1,-1) rectangle (9.7,9.7);

%     \draw[scale=20] (0,0) -- (0.213626375754752, 2.96182168028514);
%     \draw[scale=20] (0,0) -- (0.213626375754752, 2.96182168028514);
%     \draw[scale=20] (0,0) -- (0.446890612333807, 2.41433794338881);
%     \draw[scale=20] (0,0) -- (0.699492716817427, 1.97825875597953);
    \draw[scale=20,dashed,ultra thick] (0,0) -- (0.973639796553168, 1.60458457175326);
    \draw[scale=20,dashed,ultra thick] (0,0) -- (1.27323954473508, 1.27323954473508);
    \draw[scale=20,dashed,ultra thick] (0,0) -- (1.60458457175326, 0.973639796553168);
%     \draw[scale=20] (0,0) -- (1.97825875597953, 0.699492716817428);
%     \draw[scale=20] (0,0) -- (2.41433794338881, 0.446890612333807);
%     \draw[scale=20] (0,0) -- (2.96182168028514, 0.213626375754752);
%     \fill[white] (0,0) circle (7);

    \draw[dashed,ultra thick,blue] (0,0.5) -- +(20,0); 
    \draw[dashed,ultra thick,blue] (0,1.5) -- +(20,0); 
    \draw[dashed,ultra thick,blue] (0,2.5) -- +(20,0); 
    \draw[dashed,ultra thick,red] (0.5,0) -- +(0,20); 
    \draw[dashed,ultra thick,red] (1.5,0) -- +(0,20); 
    \draw[dashed,ultra thick,red] (2.5,0) -- +(0,20); 
    \fill[white] (0,0) circle (7);
    \draw[black!20] (0,0) grid (20,20);
    \draw[thick] (20,0) -- (0,0) -- (0,20);
    \end{scope}
    \draw (6.2,9.7) node[anchor=south] {$0.4$};
    \draw (9.7,9.7) node[anchor=south west] {$0.5$};
    \draw (9.7,6.2) node[anchor=west] {$0.6$};

    \draw (9.7,0.5) node[anchor=west] {$-1$};
    \draw (9.7,1.5) node[anchor=west] {$-2$};
    \draw (9.7,2.5) node[anchor=west] {$-3$};
    \draw (0.5,9.7) node[anchor=south] {$1$};
    \draw (1.5,9.7) node[anchor=south] {$2$};
    \draw (2.5,9.7) node[anchor=south] {$3$};

\end{tikzpicture}
\caption{Possible asymptotes for \emph{Schensted insertion} and the corresponding values 
(elements of the alphabet $\alphabetinsertion$) of the letter $\letter$. For details see \cref{theorem:insertion-determinism}.}
\label{fig:compactification-twisted}

\end{figure}

This theorem is illustrated in \cref{fig:simulated} where some sample \jdt paths are shown together with their asymptotes (dashed lines).
The set of all possible asymptotes is visualized in \cref{fig:compactification}.

% As we shall see later, the asymptote of a \jdt path $\jpath(T)$ carries some important information about the infinite Young tableau $T$.

The probability distribution of slopes $\Theta$ of \jdt paths in case \ref{enum:C} 
is universal (in the sense that it does not depend on $\alpha$, $\beta$, $\gamma$)
and is known explicitly; 
we postpone presentation of its details until
\cref{theorem:explicit-distribution-Theta}.

\subsection{The inverse of $\RSK$}

Recall that $\alphabetjeu$ is the \jdt alphabet shown in \cref{eq:jdt-alphabet}; the details of its definition are postponed to
\cref{subsec:alphabetjeu}.
We shall define now a function $\Psi:\tableaux\rightarrow\alphabetjeu$.
Let $t$ be an infinite Young tableau. We will use notations of \cref{theo:paths}. 
\[ \Psi(t):= 
\begin{cases} 
    k  & \text{if case \ref{enum:A} holds}, \\
    -k & \text{if case \ref{enum:B} holds}, \\          
F_\Theta\big( \Theta(t) \big) & \text{if case \ref{enum:C} holds,}
\end{cases}\]
where $F_\Theta$ is the cumulative distribution function of the distribution of $\Theta$
(see \cref{theorem:explicit-distribution-Theta}).

% 
% 
% In case \ref{enum:A} we define $\Psi(t)=k$.
% In case \ref{enum:B} we define $\Psi(t)=-k$.
% In case \ref{enum:C} we define $\Psi(t)=F_\Theta\big( \Theta(t) \big)$.
This function $\Psi$ is visualized in \cref{fig:compactification}.
\cref{theo:paths} shows that (with respect to the probability measure $\measure_{\alpha,\beta,\gamma}$) this function is well-defined almost everywhere.

\begin{theorem}
\label[theorem]{theorem:inverse-rsk}
The inverse of\/ $\RSK$ map from \cref{theorem:isomorphism} and \cref{theorem:isomorphism-dynamical}
is given (almost surely) by asymptotic slopes of \jdt in consecutive iterations of \jdt transformation $J$. More explicitly,
\begin{equation}
\label{eq:rsk-inverse}
 \RSK^{-1}(t) := \Big( \Psi(t), \Psi\big( J(t) \big), \Psi\big( J(J(t)) \big) ,\dots \Big) \in \alphabetjeu^\N. 
\end{equation}
\end{theorem}

It should be stressed that the above theorem holds \emph{only in the almost sure sense}, i.e., it states that the equalities
\begin{align*}
\RSK \circ \RSK^{-1} & = \Id ,\\
\RSK^{-1} \circ \RSK & = \Id ,
\end{align*}
hold true except for measure-zero sets.

\subsection{Special case: Plancherel measure}
One of Thoma characters, the one corresponding to $\alpha=\beta=(0,0,\dots)$, $\gamma=1$ plays a special role. The corresponding indecomposable central measure is the celebrated \emph{Plancherel measure} on the set $\tableaux$ of infinite Young tableaux.
The jeu de taquin alphabet $\alphabetjeu$ in this case can be
identified simply with the unit interval $(0,1)$ equipped with the Lebesgue measure
and one does not have to consider the subtleties related to row letters and column letters. This case was considered in our previous paper \cite{RomikSniady2011};
in particular \cref{theorem:isomorphism,theo:measure-preserving-and-ergodic,theorem:isomorphism-dynamical,theo:paths,theorem:inverse-rsk} were all proved there in this special case.
The proofs for the general case presented in the current paper will heavily use the results from that paper (see \cref{theorem:insertion-determinism-Plancherel}).

\subsection{Special case: non-colliding random walks and Pitman transform}
We consider the special case when $\alpha=(\alpha_1,\dots,\alpha_\ell,0,0,\dots)$ has only finitely many non-zero entries, and $\beta=(0,0,\dots)$, $\gamma=0$ are zero. 
In particular, this means that as the \jdt alphabet we can take $\alphabetjeu=[\ell]=\{1,\dots,\ell\}$, thus we recover the usual version of $\RSK$ without column letters.

In this case, a random infinite word $(\Letter_1,\Letter_2,\dots)$ of i.i.d.~letters with distribution $\alphabetmeasurejeu_{\alpha,\beta,\gamma}$ can be identified with a random walk $X$ in $\Z^\ell_+$. The recording tableau $\RSK(\Letter_1,\Letter_2,\dots)$ has boxes only in the first $\ell$ rows, thus the corresponding path  $(\lambda^0\nearrow \lambda^1 \nearrow\cdots):=\RSK(\Letter_1,\Letter_2,\dots)$ in the Young graph can be also viewed as a random walk $\Lambda$ in $\Z^\ell_+$.

This setup has been studied by O'Connell and Yor \cite{OConnellYor2002}
who introduced a certain path-transformation $G^{(\ell)}$, called \emph{generalized Pitman transform}, with the property that the transformed walk $G^{(\ell)}(X)$ has the same law as the original walk $X=(X_1,\dots,X_\ell)$ conditioned never to exit the Weyl chamber $\{x: x_1 \geq \cdots \geq x_\ell\}$; such a walk can be alternatively viewed as a collection of $\ell$ random walks $X_1,\dots,X_\ell$ which are conditioned to be \emph{non-colliding}, i.e.~$X_1\geq \cdots \geq X_\ell$.
This path-transformation has been further studied by O'Connell \cite{OConnell2003} 
who has shown that Pitman transform is nothing else but $\RSK$ transform in disguise,
i.e., $\Lambda=G^{(\ell)}(X)$.
He also proved that the inverse of the map $G^{(\ell)}=\RSK$ exists and he found it explicitly. Clearly, his result is a special case of \cref{theorem:isomorphism},
however it is not immediate that his formula \cite[Corollary 3.2]{OConnell2003} for $\RSK^{-1}$ is equivalent to the one
given in the current paper (\cref{theorem:inverse-rsk}).

\subsection{Outline of the paper}

The main results of the paper (which were presented in this Introduction)
will be proved in \cref{sec:proof-of-main-results}.
All proofs will base on key \Cref{theo:determinism-of-jdt} which gives a detailed information about the \jdt path for some special random infinite Young tableau. It will be convenient to prove this result in an equivalent form as \cref{theorem:insertion-determinism}; essentially most of the current paper is just a preparation for the proof of this \cref{theorem:insertion-determinism}.
We review it briefly.

\Cref{sec:preliminaries} contains some missing notation from this Introduction and presents some wider context.

In \Cref{sec:alphabets} we present how some classical combinatorial notions can be adapted to the more general setup of alphabets containing row letters and column letters.

\Cref{sec:jdt} concerns some basic properties of \jdt.

\Cref{sec:growth} concerns typical shape of some random Young diagrams and the asymptotic determinism of Schensted insertion in the special case related to the Plancherel measure.

In \cref{sec:asymptotic-determinism-schensted} we show the key technical result, \cref{theorem:insertion-determinism} which concerns asymptotic determinism of Schensted insertion in the general case.
 
Finally, \cref{sec:proof-of-main-results} contains the proofs of the main results.

\section{Preliminaries: \\ 
% symmetric groups, 
Young diagrams, Young tableaux}
\label{sec:preliminaries}

% \subsection{The symmetric groups}
% 
% We use the convention $\N=\{1,2,\dots\}$ for the set of the natural numbers.
% 
% We view the \emph{(finite) symmetric group} $\Sym{n}$ as the group of permutations of the set $\{1,\dots,n\}$ or, alternatively, as the group of permutations $\pi:\N\rightarrow\N$ with the property that $\pi(i)=i$ for any $i\notin\{1,\dots,n\}$.
% 
% The \emph{infinite symmetric group} $\Sym{\infty}$ is defined as the group of permutations
% $\pi:\N\rightarrow\N$ for which the \emph{support}, i.e., the set $\{i\in\N:\pi(i)\neq i\}$, is finite.
% 
% With these notations there is a natural inclusion of the groups:
% \[ \Sym{0} \subset \Sym{1} \subset \Sym{2} \subset \cdots \subset \Sym{\infty}\] 
% 
% \subsection{Characters}
% If $G$ is a discrete group, a \emph{character} of $G$ is defined as any function $\chi:G\rightarrow\C$ which is:
% (i) \emph{normalized}, i.e., $\chi(e)=1$ (where $e$ denotes the group identity), (ii) \emph{positive definite}, (iii) is a \emph{class function}, i.e., it is constant on each conjugacy class.
% A character is called \emph{indecomposable} if it cannot be written as a non-trivial linear combination of other characters.
% 
% In the case when the group $G$ is finite, its indecomposable characters coincide with (suitably normalized) \emph{irreducible characters} 
% \[\chi_\rho(g):=\frac{\Tr \rho(g)}{\Tr \rho(e)},\] 
% where $\rho$ is an irreducible representation; in the general case they are in a bijective correspondence with \emph{finite factor representations} of the group (in the sense of Murray and von Neumann).

\subsection{Young diagrams, Young graph}
The set of Young diagrams with $n$ boxes will be denoted by $\partitions{n}$; the set of all Young diagrams will be denoted by $\allpartitions$.

The set $\allpartitions$ of Young diagrams carries in a natural way the structure of a directed graph, which will be called \emph{Young graph}, see \cref{fig:young-graph}. Namely, for a pair of Young diagrams we write $\lambda \nearrow \mu$ if the diagram $\mu$ is obtained from $\lambda$ by adding exactly one box. The \emph{empty Young diagram} with no boxes will be denoted by $\emptyset$.

From the perspective of the asymptotic representation theory, it is very interesting to investigate the boundary of this graph. This motivates investigation of \emph{infinite paths in this graph} which, as we shall see, correspond to \emph{infinite Young tableaux}.

\subsection{Infinite Young tableaux}
\label{subsec:infinite-young-tableaux-revisited}
We use the notation $\N=\{1,2,3,\dots\}$ for the set of the natural numbers. We use so defined natural numbers to index rows and columns of Young diagrams and tableaux; in particular the first row (column) corresponds to the number $1$, etc.

An \emph{infinite Young tableau $t$} is a function 
$\mathbb{N}^2\ni (x,y) \mapsto t_{x,y}\in\{1,2,3,\dots,\infty\}$. We interpret it as a filling of the boxes of the first quadrant of the plane; the boxes filled with the symbol $\infty$ can be interpreted as empty boxes (see \cref{fig:inf-A}). We require that each finite entry (an element of the set $\{1,2,\dots\}$) appears in \emph{exactly} one box and that each row and each column is weakly increasing (from left to right and from bottom to top).
This definition differs slightly from the one from our previous paper \cite{RomikSniady2011}, where no empty boxes were allowed.

An infinite Young tableau can be viewed alternatively, as follows.
There is a bijective correspondence between infinite Young tableaux and \emph{infinite paths in the Young graph}
\begin{equation}
\label{eq:infinite-path}
 \emptyset = \lambda^0 \nearrow \lambda^1 \nearrow \cdots. 
\end{equation}
This correspondence is defined as follows: for an infinite tableau $t$ we define the Young diagram $\lambda^i$ as the collection of boxes with entries $\leq i$.

The \emph{set of infinite Young tableaux} will be denoted by $\tableaux$.
It is equipped with its natural measurable structure, namely, the minimal $\sigma$-algebra $\mathcal{F}$ of subsets of $\tableaux$
such that all the coordinate functions $t\mapsto t_{x,y}$ are measurable.

\section{Alphabets with row letters and column letters}
\label{sec:alphabets}

\subsection{Alphabets with row letters and column letters}
\label{subsec:generalized-alphabets}

Let $\alphabet=\alphabet_{c} \sqcup \alphabet_{r}$ be an \emph{alphabet} (i.e.,~a linearly ordered set).
The elements of $\alphabet_{r}$ will be called \emph{row letters} while the elements of $\alphabet_{c}$ will be called \emph{column letters} (in the original paper \cite[Section 1]{KerovVershik1986} these were called, respectively, \emph{positive} and \emph{negative}, which is not very convenient for our purposes).
We define the relationships $<_r$ and $<_c$ by
\begin{align*} 
a <_r b  \quad & \iff\quad  (a<b) \vee \big[ (a=b) \wedge a \in \alphabet_{r}  \big], \\
a <_c b  \quad & \iff\quad  (a<b) \vee \big[ (a=b) \wedge a \in \alphabet_{c}  \big] 
\end{align*}
for any $a,b\in \alphabet$. 
Notice that for any $a,b\in\alphabet$ \emph{exactly one} of the following statements is true:
$a <_r b$ or $b <_c a$.

We may also consider $\alphabet=\alphabet_{c} \sqcup \alphabet_0 \sqcup \alphabet_{r}$;
the elements of $\alphabet_0$ will be called \emph{neutral letters}.
In this case the relationships $a <_r a$ and $a<_c a$ are not well-defined for $a\in \alphabet_0$. This will not create any problems as long as any element of $\alphabet_0$ appears in the words and tableaux which we consider at most once. Alternatively, any element of $\alphabet_0$ can be regarded either as an element of $\alphabet_r$ or $\alphabet_c$.

\subsection{The \jdt alphabet}
\label{subsec:alphabetjeu}

\begin{figure}[tbp]
\centering
\subfile{programy-jeu-de-taquin/grafika-insertion-tableau2a.tex}
\caption{Example of a tableau in the \jdt alphabet $\alphabetjeu$, see \cref{subsec:alphabetjeu}. Row letters are marked by horizontal lines, column letters are marked by vertical lines.
Highlighted boxes form the bumping route when letter $2$ is inserted into tableau.}
\label{fig:tableauJEU-a}

\vspace{5ex}

\centering
\subfile{programy-jeu-de-taquin/grafika-insertion-tableau2b.tex}
\caption{The outcome of insertion of letter $2$ into the tableau from  \cref{fig:tableauJEU-a}. Highlighted boxes form the bumping route.}
\label{fig:tableauJEU-b}
\end{figure}

% 
% \begin{figure}[tbp]
% \subfile{programy-jeu-de-taquin/grafika-insertion-tableau.tex}
% \caption{Example of a tableau in the insertion alphabet $\alphabetinsertion$, see \cref{subsec:insertionalphabet}. }
% \label{fig:tableauINSERTION}
% \end{figure}

For our purposes, the most important example of an alphabet is $\alphabetjeu=\alphabetjeu_r\sqcup \alphabetjeu_0 \sqcup \alphabetjeu_c$ with $\alphabetjeu_{r}=\{1,2,3,\dots\}$, $\alphabetjeu_0=(0,1)\subset\R$ and $\alphabetjeu_{c}=\{\dots,-3,-2,-1\}$
with the linear order defined as follows: on each of the sets $\alphabetjeu_r$, $\alphabetjeu_0$, $\alphabetjeu_{c}$ we consider the natural order; we declare any element of $\alphabetjeu_{r}$ smaller than any element of $\alphabetjeu_0$, which is smaller than any element of $\alphabetjeu_{c}$.  This alphabet will be called \emph{the \jdt alphabet}; it can be visualized informally as \cref{eq:jdt-alphabet}.

If $(\alpha,\beta,\gamma)$ belongs to Thoma simplex, we define the following probability measure $\alphabetmeasurejeu_{\alpha,\beta,\gamma}$ on $\alphabetjeu$: 
\begin{itemize}
 \item 
for $i\in\{1,2,3,\dots\}=\alphabetjeu_{r}$ we set $\alphabetmeasurejeu_{\alpha,\beta,\gamma}(i)=\alpha_i$; 

\item
for $-i\in\{-1,-2,-3,\dots\}= \alphabetjeu_{c}$ we set $\alphabetmeasurejeu_{\alpha,\beta,\gamma}(-i)=\beta_i$; 

\item on $\alphabetjeu_0=(0,1)$ we take as $\alphabetmeasurejeu_{\alpha,\beta,\gamma}$ the absolutely continuous measure on the unit interval $(0,1)$ with constant density $\gamma$.
\end{itemize}

This alphabet and probability measure were used in \cref{theorem:isomorphism}
and \cref{theorem:isomorphism-dynamical}.

\subsection{Tableaux}

A \emph{(semistandard) tableau} in our new set up is defined as a filling of the entries of a Young diagram with the property that each row is \emph{$<_r$-increasing} (from left to right) and each column is \emph{$<_c$-increasing} (from bottom to top), see \cref{fig:tableauJEU-a}. This definition is equivalent to the one of Kerov and Vershik \cite[Section 2]{KerovVershik1986}.

\subsection{Robinson-Schensted-Knuth algorithm}
\label{sec:generalized-schensted}

We assume that the reader is familiar with the details of Robinson-Schensted-Knuth algorithm, which are described in several well-known sources such as \cite{Fulton1997,KnuthVol3,StanleyVol2,Sagan2001}. We provide only a brief overview below.

The \emph{(row) insertion procedure} applied to a tableau $t$ and a letter $\letter\in\alphabet$
produces a new tableau denoted $t\leftarrow \letter$. 
The new
tableau is computed by performing a succession of bumping steps
whereby $\letter$ is inserted (by a procedure which we call \emph{elementary insertion})
into the first row of the diagram, bumping an
existing entry from the first row into the second row, which results in
an entry of the second row being bumped to the third row, and so on,
until finally the entry being bumped settles down in an unoccupied
position outside the diagram.

The \emph{elementary insertion} has to be adjusted to our new setup: we insert the new letter into the row as much to the right as possible, so that the row remains \emph{$<_r$-increasing} and no gaps are created,
see \cref{fig:tableauJEU-a,fig:tableauJEU-b}.
This definition is equivalent to the one from the work of Kerov and Vershik \cite[Section 2]{KerovVershik1986}. 

The \emph{insertion tableau} $P(\letter_1,\dots,\letter_n)$ associated to a finite word is defined as the outcome of iterative insertion of the letters into the empty tableau:
\begin{equation} 
\label{eq:definition-insertion}
P(\letter_1,\dots,\letter_n):= 
\Big( \big((\emptyset \leftarrow \letter_1) \leftarrow \letter_2 \big) \leftarrow \cdots \Big) \leftarrow \letter_n.
\end{equation}

The \emph{$\RSK$ shape} of a finite word $(\letter_1,\dots,\letter_n)$ is defined as the Young diagram, equal to the shape of $P(\letter_1,\dots,\letter_n)$.

The \emph{recording tableau} $Q(\letter_1,\letter_2,\dots)$ associated to the (finite, respectively, infinite) word $(\letter_1,\letter_2,\dots)$ is defined as the (finite, respectively, infinite) Young tableau which corresponds to the (finite, respectively, infinite) path $\lambda^0\nearrow\lambda^1\nearrow\cdots$ in Young graph defined as follows: $\lambda^k$ is the $\RSK$ shape of the prefix $(\letter_1,\dots,\letter_k)$.

If $w_1,w_2,\dots$ is an infinite word, we define the outcome of 
\emph{Robinson-Schensted-Knuth algorithm} as the corresponding recording tableau.
\[ \RSK(w_1,w_2,\dots):=Q(w_1,w_2,\dots)\in\tableaux. \]

\subsection{Robinson-Schensted-Knuth algorithm as a homomorphism of probability spaces}

We present now the precise form of the result of Kerov and Vershik which we discussed in \cref{subsec:rsk-and-central}. We will use this result several times: roughly speaking, whenever a random infinite Young tableau distributed according to some indecomposable central measure (Vershik-Kerov measure) has to be used, we will use a concrete realization of such a random tableau on the probability space of a sequence of i.i.d.~random letters.

Note that the result below applies, in particular, to the special cases when (a) the alphabet $\alphabet=\alphabetjeu$ is the \jdt alphabet equipped with the probability measure $\alphabetmeasurejeu_{\alpha,\beta,\gamma}$ 
or, (b) when 
the alphabet $\alphabet=\alphabetinsertion$ is the insertion alphabet equipped with the probability measure $\alphabetmeasureinsertion_{\alpha,\beta,\gamma}$ (the definition of this alphabet is postponed until \cref{subsec:insertionalphabet}). 
In fact, these are the only two cases which will be used in the current paper, so the reader can focus her attention on them.

\begin{fact}[$\RSK$ is a \emph{homomorphism} of probability spaces, Kerov and Vershik {\cite[Theorem 2]{KerovVershik1986}}]
\label[fact]{theo:KerovVershik-RSK-homomorphism}
Let alphabet $\alphabet=\alphabet_r\sqcup\alphabet_0\sqcup\alphabet_c$ with a probability measure $\alphabetmeasure$ be given. 
Let $\alpha_1\geq \alpha_2\geq \cdots$ be the probabilities (listed in the weakly decreasing order) of the atoms of the measure $\alphabetmeasure$ restricted to $\alphabet_r$ and let
$\beta_1\geq \beta_2\geq \cdots$ be the probabilities (listed in the weakly decreasing order) of the atoms of the measure $\alphabetmeasure$ restricted to $\alphabet_c$. 
Let $\gamma$ be the total probability of the continuous part of $\alphabetmeasure$.
We assume that the probability measure $\alphabetmeasure$ restricted to $\alphabet_0$ has no atoms.

Let $\Letter_1,\Letter_2,\dots $ be a sequence of random, independent, identically distributed letters from $\alphabet$ with distribution $\alphabetmeasure$.
Then the  distribution of the recording tableau $Q(\Letter_1,\Letter_2,\dots)\in\tableaux$ coincides with  
Vershik-Kerov measure $\measure_{\alpha,\beta,\gamma}$.

In other words, $\RSK$ is a \emph{homomorphism} between the following two probability spaces:
\begin{itemize}
 \item $\big(\alphabet^\infty,\mathcal{B},\alphabetmeasure^\infty\big)$, i.e., sequences of i.i.d.~random letters;
 \item $(\tableaux,\mathcal{F},\measure_{\alpha,\beta,\gamma})$, i.e., random infinite Young tableaux with Vershik-Kerov measure $\measure_{\alpha,\beta,\gamma}$.
\end{itemize}
\end{fact}

In order to recover this formulation from the original work of Kerov and Vershik, 
one should simply declare that any element of $\alphabet_0$ is either a row or a column letter.
Since, almost surely, any neutral letter appears in the sequence  $W_1,W_2,\dots$ at most once, this does not create any difficulties.

\subsection{Greene's theorem}

\begin{fact}[Greene's theorem]
\label[fact]{theorem:greene}
Let $\bm{w}$ be a finite word in some alphabet $\alphabet=\alphabet_r \sqcup \alphabet_c$. 
Let $\lambda=(\lambda_1,\lambda_2,\dots)$ be the $\RSK$ shape associated to $\bm{w}$. 

Then
for each $k\geq 1$, the sum of the lengths of the first $k$ rows,
$\lambda_1+\cdots+\lambda_k$,  is equal to the length of the longest subsequence of $\bm{w}$ which can be decomposed into $k$ disjoint $<_r$-increasing subsequences.

Also, the sum of the lengths of the first $k$ columns,
$\lambda'_1+\cdots+\lambda'_k$,  is equal to the length of the longest subsequence of $\bm{w}$ which can be decomposed into $k$ disjoint $<_c$-decreasing subsequences.

\end{fact}

For the proof of this result for alphabets containing row letters and column letters 
we refer to the work of Kerov and Vershik \cite[Proposition~1]{KerovVershik1986}.

\subsection{Standardization of a sequence}
\label{subsec:rectification}

In the current paper we will use generalizations of several classical results concerning $\RSK$ in the setup of alphabets involving row letters and column letters. In the following we present a simple technical tool which will be used in order to show that a given result in the generalized setup is, in fact, equivalent to its classical version.

Let $\bm{\letter}=(\letter_1,\dots,\letter_n)$ with $\letter_1,\dots,\letter_n\in\alphabet$. We assume that each neutral letter appears at most once in $\bm{\letter}$. Let $\PI=(\pi_1,\dots,\pi_n)$ be a tuple of some abstract elements which are all different. We define a linear order on $\{\pi_1,\dots,\pi_n\}$ by setting for all $1\leq i<j\leq n$:
\begin{equation}
\label{eq:rectification-sequence}
  \pi_i < \pi_j\quad  \iff w_i <_r w_j; 
\end{equation}
in other words it is a lexicographic order in which we first compare  $\letter_i$ with $\letter_j$ with respect to the usual order $<$; if they are equal then we compare the indices $i$ and $j$ in the usual order (for $\letter_i\in\alphabet_r$) or in the opposite order (for $\letter_i\in\alphabet_c$).

The tuple $\PI(\letter_1,\dots,\letter_n):=(\pi_1,\dots,\pi_n)$, called \emph{standardization} of $\bm{\letter}$, is uniquely determined 
up to an order-preserving isomorphism; it can be identified with a permutation.
The following \Cref{lem:rectification-of-a-sequence} shows that with respect to $\RSK$, the original tuple and its standardization have similar properties;
the advantage of the tuple $\PI$ is that its entries are not repeated, thus we avoid the difficulties related to column letters and row letters and we can apply some classical results directly.

\begin{lemma}
\label[lemma]{lem:rectification-of-a-sequence}
The recording tableaux corresponding to the words $\bm{\letter}$ and its standardization $\PI(\bm{\letter})$ are equal.
\end{lemma}
\begin{proof}
In order to show that the recording tableaux are equal, it is enough to show that for each $1\leq k\leq n$, $\RSK$ shapes associated to the prefixes  
$(\letter_1,\dots,\letter_k)$ and $(\pi_1,\dots,\pi_k)$ are equal.

Since there is a bijective correspondence between $<_r$-increasing subsequences of $\letter$ and $<$-increasing subsequences of $\pi$, i.e.,
for any $i_1< \cdots < i_\ell$ 
\[ w_{i_1} <_r \cdots <_r w_{i_\ell} \quad \iff \quad \pi_{i_1}<\cdots <\pi_{i_\ell},\]
Greene's theorem (\cref{theorem:greene}) finishes the proof.
\end{proof}

\section{Elementary properties of \jdt}
\label{sec:jdt}

\subsection{Lazy version of \jdt}
It will be convenient to work with a modified version of the \jdt path in which time is reparametrized. We call this the \emph{natural parametrization} of the \jdt path. To define it, for a given tableau $t\in\tableaux$ let $\jpathlazy_n(t) =\jpath_{K(n)}$ where $K(n)$ is the maximal number $k$ such that $t_{\jpath_k}\leq n$, i.e., the tableau entry in position $\jpath_k$ is smaller or equal than $n$. The reparametrized sequence $(\jpathlazy_n)_{n\ge1}$ is simply a slowed-down or ``lazy'' version of the \jdt path: as $n$ increases it either jumps to its right or up if in the growth process \eqref{eq:infinite-path} a box was added in one of those two positions, and stays put at other times.

\subsection{Finite version of \jdt}
For a finite Young tableau $t$ with $n\geq 1$ boxes, just like for the infinite case considered in \cref{subsec:jdt}, we remove the corner box, we perform the sequence of slidings
(which is now a \emph{finite} sequence), and we subtract $1$ from every entry of the resulting ``tableau''. The resulting Young tableau with $n-1$ boxes will be denoted by $j(t)$.

\subsection{Sch\"utzenberger's \jdt}
We will use the special name \emph{Sch\"u\-tzen\-berger's \jdt}
(which maps the set of \emph{skew tableaux} to the set of tableaux; this map associates to a skew tableau its rectification, see \cite[Section 1.2]{Fulton1997} and \cite[Section 3.7]{Sagan2001}) in order to distinguish
it from \emph{\jdt transformation} considered in the current paper (which is a map $J$, respectively $j$, on the set of infinite, respectively finite, Young tableaux).
In particular, the finite \emph{\jdt transformation} $j$ can be described equivalently as the composition of (i) removal of the corner box, (ii) \emph{Sch\"utzenberger's \jdt}, (iii) subtracting $1$ from each entry.

% We assume that the reader is accustomed with Sch\"utzenberger's \jdt; the details can be found in the textbooks \cite{Fulton1997,Sagan2001}.

\subsection{Duality between \jdt and one-directional shift}

In the setup when the alphabet $\alphabet$ consists only of row letters, this result has been proved by Sch\"utzenberger \cite{Schutzenberger1963}; we will use its generalized version for alphabets consisting of row and column letters.

\begin{lemma}[Duality between \jdt and one-directional shift]
\label[lemma]{lem:factor-map}
Let the alphabet $\alphabet=\alphabet_r\sqcup\alphabet_0\sqcup\alphabet_c$ be given and
let $\letter_1,\dots,\letter_n\in\alphabet$. 
We assume that each neutral letter appears at most once in this tuple. 
% Let $Q_n=Q(\letter_1,\letter_2,\dots,\letter_n)$ and $\widetilde{Q}_{n-1}=Q(\letter_2,\letter_3,\dots,\letter_n)$ be the recording tableaux correspoding, respectively, to the original sequence and to the original sequence with the first element removed. 

Then
\[ Q(\letter_2,\letter_3,\dots,\letter_n)  = j\big( Q(\letter_1,\letter_2,\dots,\letter_n) \big),\]
where $j$ is the finite version of the \jdt map.
\end{lemma}
\begin{proof}
In \cref{subsec:rectification} we defined the standardization $(\pi_1,\dots,\pi_n)=\PI(\letter_1,\dots,\letter_n)$. One can easily show that $(\pi_2,\dots,\pi_n)=\PI(\letter_2,\dots,\letter_n)$ (more precisely, we can define $\PI(\letter_2,\dots,\letter_n):=(\pi_2,\dots,\pi_n)$ and check that it fulfills the requirement \eqref{eq:rectification-sequence} from the definition; notice that $\PI(\letter_2,\dots,\letter_n)$ is defined only up to an order-preserving isomorphism). \Cref{lem:rectification-of-a-sequence} shows that the corresponding recording tableaux are equal:
\begin{align*} 
Q(w_1,\dots,w_n) &= Q(\pi_1,\dots,\pi_n), \\
Q(w_2,\dots,w_n) &= Q(\pi_2,\dots,\pi_n). 
\end{align*}

Thus it is enough to show the lemma for the tuple $(\letter'_1,\dots,\letter_n'):=(\pi_1,\dots,\pi_n)$.
Since $\pi_1,\dots,\pi_n$ are distinct, this is the setup considered by Sch\"utzenberger, see \cite[Proposition 3.9.3]{Sagan2001}.
% our previous work \cite[Lemma 2.3]{RomikSniady2011}. 
\end{proof}

\section{Growth of random Young diagrams}
\label{sec:growth}

\subsection{Lengths of rows and columns of random Young diagrams}

The following is the classical result of Vershik and Kerov 
(which we discussed already in \cref{subsec:infinite-young-tableaux-revisited}) about the asymptotic growth of a random Young diagram distributed according to some indecomposable central measure.

\begin{fact}[Vershik and Kerov {\cite[Corollary 5]{VershikKerov1981}}]
\label[fact]{theo:VK-frequencies}
Let $(\alpha,\beta,\gamma)$ be an element of Thoma simplex and let
$(\Lambda^0\nearrow \Lambda^1 \nearrow \cdots)\in\tableaux$ be a random infinite tableau with the distribution given by Vershik-Kerov measure $\measure_{\alpha,\beta,\gamma}$.

Then, almost surely, for each $i\in\{1,2,\dots\}$
\begin{align*}
 \lim_{n\to\infty} \frac{\Lambda^n_i}{n} & = \alpha_i, \\ 
 \lim_{n\to\infty} \frac{(\Lambda^n)'_i}{n} & = \beta_i, \\
\end{align*}
where $\Lambda^n_i$ (respectively,  $(\Lambda^n)'_i$) denotes the number of boxes in $i$-th row (respectively, $i$-th column) of Young diagram $\Lambda^n$.
\end{fact}

% \begin{proof}
% Since this result does not correspond literally to the results presented in \cite{VershikKerov1981}, we present how to recover this result from the original paper of Vershik and Kerov.
% 
% \cite[Theorem 1]{VershikKerov1981} states existence of an infinite tableau with a certain property. In fact, the proof of this result easily implies the following stronger statement: this property holds almost surely for a random infinite tableau with the distribution $\measure_{\alpha,\beta,\gamma}$. 
%  
% \end{proof}

We will also need the following more refined information about the growth of the number of rows and the number of columns
in the case when some parameters in Thoma simplex are zero.
\begin{lemma}
We keep notations from \cref{theo:VK-frequencies}.
\begin{itemize}
 \item 
Assume that $\beta=(0,0,\dots)$ and $\gamma=0$. 
Then for each $\epsilon>0$ there exists a constant $d>0$ such that
\[  P\left( \frac{(\Lambda^n)'_1}{\sqrt{n}} > \epsilon \right) =
O\left( e^{-d \sqrt{n}} \right). \]

\item 
Assume that $\alpha=(0,0,\dots)$ and $\gamma=0$. 
Then for each $\epsilon>0$ there exists a constant $d>0$ such that
\[  P\left( \frac{\Lambda^n_1}{\sqrt{n}} > \epsilon \right) =
O\left( e^{-d \sqrt{n}} \right). \]
\end{itemize}
\end{lemma}
\begin{proof}
Without loss of generality we may assume that the tableau $(\Lambda^0\nearrow \Lambda^1 \nearrow \cdots)=Q(\Letter_1,\Letter_2,\dots)$ is the recording tableau of an i.i.d.~sequence of random letters with a suitable probability distribution, as prescribed by \cref{theo:KerovVershik-RSK-homomorphism}. Thus the claim is equivalent to \cref{lemma:row-letters-no-quadratic-height} below.
\end{proof}

\begin{lemma} 
\label[lemma]{lemma:row-letters-no-quadratic-height}
\
\begin{itemize}
 \item 

Let $\alphabet=\alphabet_r$ be an alphabet which consist only of row letters, equipped with a probability measure $\alphabetmeasure$ which does not have any continuous part.
Let
$(\Letter_1,\Letter_2,\dots)$ be a sequence of independent, identically distributed elements of $\alphabet$ with distribution $\alphabetmeasure$.

Then, for each $\epsilon>0$ there exists some $d>0$ such that
\[ P\big( \text{$\RSK$ shape of $(\Letter_1,\dots,\Letter_n)$ has at least $\epsilon \sqrt{n}$ rows} \big)  = O\left( e^{-d\sqrt{n}}\right).\]

\item
Let $\alphabet=\alphabet_c$ be an alphabet which consist only of column letters, equipped with a probability measure $\alphabetmeasure$ which does not have any continuous part. Let
$(\Letter_1,\Letter_2,\dots)$ be a sequence of independent, identically distributed elements of $\alphabet$ with distribution $\alphabetmeasure$.

Then, for each $\epsilon>0$ there exists some $d>0$ such that
\[ P\big( \text{$\RSK$ shape of $(\Letter_1,\dots,\Letter_n)$ has at least $\epsilon \sqrt{n}$ columns} \big)  = O\left( e^{-d\sqrt{n}}\right).\]
\end{itemize}
\end{lemma}
\begin{proof}
We will show the first part of the lemma.
Let $\alpha_1\geq \alpha_2 \geq \cdots \geq 0$ be the probabilities of the atoms of the probability measure $\alphabetmeasure$; clearly
$$ \alpha_1+\alpha_2+\cdots = 1 .$$
Let $D>0$ be a positive constant, we will fix its value at the end of the proof. Let $m$ be big enough so that
\[\alpha_1+\cdots+\alpha_m > 1- D.\]
Let $x_1,\dots,x_m\in\alphabet$ be the atoms of the measure $\alphabetmeasure$ with the biggest weights. 

Note that the case when $(\alpha_1,\alpha_2,\dots)$ contains only a finite number of non-zero entries will require later on some special attention; in this case we set $m$ to be the number of such non-zero entries; thus
\begin{equation}
\label{eq:sum-is-one}
 \alpha_1+\cdots+\alpha_m=1. 
\end{equation}

We denote by $\Word'=(\Letter'_1,\dots,\Letter'_{\ell(n)})$ the tuple $(\Letter_1,\dots,\Letter_n)$ with all entries which belong to $\{x_1,\dots,x_m\}$ removed. 
% Since the alphabet consists only of row letters,
By Greene's theorem (\cref{theorem:greene}),
the number of rows of the $\RSK$ shape of $(\Letter_1,\dots,\Letter_n)$ is equal to the length of the longest $<_c$-decreasing subsequence of $(\Letter_1,\dots,\Letter_n)$. In our case, there are no column letters, so such a sequence is strictly $<$-decreasing, hence its length is bounded from above by
\[ m + \big(\text{length of the longest strictly decreasing subsequence of $\Word'$} \big).\]
Thus it remains to show that (with high probability) the second summand grows sufficiently slowly with $n$.

In the case \eqref{eq:sum-is-one} when $(\alpha_1,\alpha_2,\dots)$ contains only finitely many non-zero entries, the tuple $\Word'$ is almost surely empty and the statement of the lemma follows trivially. Thus it remains to show the lemma in the remaining case
\[ \alpha_1+\cdots+\alpha_m<1. \]
We denote by $g>0$ any constant such that
\[ \alpha_1+\cdots+\alpha_m+g <1.\]

The distribution of the random length $\ell(n)$ of the word $\Word'$ is given by a binomial distribution 
with success probability $p:=1-(\alpha_1+\cdots+\alpha_m)$ with $g<p<D$. Thus by elementary large deviations theory there exists some constant $h>0$ such that
\begin{equation}
\label{eq:law-of-large-numbers-unusual-letters}
P\left(  \frac{\ell(n)}{n} \notin (g,D) \right) = O\left( e^{-h n}\right).
\end{equation}
% we provide an elementary proof below.
% The distribution of the random length $\ell(n)$ of the word $\Word'$ is given by a binomial distribution 
% with success probability $p:=1-(\alpha_1+\cdots+\alpha_m)$.
% Markov inequality applied to the random variable $e^{r\ell(n)}$ for $r>0$ shows that
% \begin{multline*} P\big( \ell(n) > D n  \big) \leq \frac{\E e^{r \ell(n)}}{e^{r D n}}=\left( \frac{ (1-p) e^0 + p e^{r}}{e^{r D}} \right)^n 
% =\\ \left(1+(p-D)r + O(r^2)  \right)^n; \end{multline*}
% since $p<D$, we can select $r>0$ small enough that 
% $\frac{ (1-p) e^0 + p e^{r}}{e^{r D}}<1$.
% An analogous bound can be obtained for the probability 
% $P\big( \ell(n) < g n  \big)$.
% 
In the following we condition over $\ell=\ell(n)$ and assume that 
\begin{equation}
\label{eq:assumption}
 g n < \ell(n) < D n.  
\end{equation}

We consider the set $Z(\Word')$ of all permutations $\pi=(\pi_1,\dots,\pi_\ell)$ with the property that for any $1\leq i,j\leq \ell$
\[ (\Letter'_i \neq \Letter'_j) \quad \implies\quad \big[ (\pi_i<\pi_j) \iff (\Letter'_i < \Letter'_j) \big],\]
in other words, except for repeating letters, the order of the entries of $\pi$ should coincide with the order of entries of $\Word'$. Any such a  permutation has the property that
\begin{multline*}  \big(\text{length of the longest strictly decreasing subsequence of $\Word'$} \big) \leq \\  \big(\text{length of the longest decreasing subsequence of $\pi$} \big).
\end{multline*}

Let $\pi$ be a random element of the (random) set $Z(\Word')$ (we sample with the uniform probability). We claim that $\pi$ is uniformly distributed on the symmetric group. Indeed, the natural action of the symmetric group $\Sym{\ell}$ on the set of words of length $\ell$ (by permutation of the letters) is such that each $\sigma\in\Sym{\ell}$ maps the set $Z(\Word')$
to the set $Z\big(\sigma(\Word')\big)$. Since the words $\Word'$ and $\sigma(\Word')$ have the same probability, it follows that the probability distribution of
the random permutation $\pi$ coincides with the distribution of $\sigma \pi$.
This invariance uniquely characterizes the uniform distribution,
so the claim that $\pi$ is uniformly distributed follows immediately.
Therefore it remains to find a suitable bound for the length of the longest decreasing subsequence of a random permutation $\pi$, distributed uniformly on the symmetric group. 
This is the classical Ulam-Hammersley problem for which lot of results are available, see \cite{Romik2013}. We provide an elementary estimate below.

\newcommand{\lengthdecreasing}{r}

By Markov's inequality, the probability that $\pi$ contains a decreasing sequence of length at least $\lengthdecreasing:= \lceil 3\sqrt{\ell} \rceil \leq 3\sqrt{Dn}+1$ is at most the expected number of such subsequences which is
\begin{equation} 
\label{eq:Markov}
\binom{\ell}{\lengthdecreasing} \frac{1}{\lengthdecreasing!} < \left( \frac{e^2 \ell}{\lengthdecreasing^2} \right)^\lengthdecreasing
\leq \left( \frac{e^2}{3^2} \right)^{3 \sqrt{\ell}} 
\leq  e^{-d \sqrt{n}},
\end{equation}
for some constant $d>0$,
where we used Stirling's approximation $\lengthdecreasing!> \lengthdecreasing^\lengthdecreasing e^{-\lengthdecreasing}$ and the assumption \eqref{eq:assumption}.

This shows that the unconditional probability of the event
\begin{multline*} 
\Big( \text{the number of rows of the $\RSK$ shape of $(\Letter_1,\dots,\Letter_n)$}
\Big) \geq \\ m+3\sqrt{Dn}+1 
 \end{multline*}
is bounded from above by the sum of the right-hand sides of \eqref{eq:law-of-large-numbers-unusual-letters} and \eqref{eq:Markov}. Thus, by choosing $D>0$ in such a way that $3\sqrt{D}<\epsilon$ we finish the proof of the first part of the Lemma.

The second part of the Lemma is completely analogous. Alternatively, one can apply the symmetry argument, as follows. We define an alphabet $\alphabet'=\alphabet'_r$ which consists only of row letters, and which, as a set, is equal to $\alphabet$. The linear order on $\alphabet'$ is defined as the opposite of the linear order on $\alphabet$. Greene's theorem (\cref{theorem:greene}) shows that the number of columns of $\RSK$ shape of $(\Letter_1,\dots,\Letter_n)$, regarded as a word in $\alphabet$, is equal to the number of rows of the $\RSK$ shape of $(\Letter_1,\dots,\Letter_n)$, this time regarded as a word in $\alphabet'$. Thus the first part of the Lemma implies immediately the second part.
\end{proof}

\subsection{Plancherel measure and Vershik-Kerov-Logan-Shepp limit shape}

The \emph{Plancherel measure} on the set $\partitions{n}$ of Young diagrams with $n$ boxes is
the probability measure given by
\[ P(\lambda)=\frac{(\operatorname{dim} \lambda)^2}{n!}, \]
where $\operatorname{dim} \lambda$ is the dimension of the irreducible representation of the symmetric group $\Sym{n}$ corresponding to $\lambda$ or, in other words, the number of Young tableaux with shape $\lambda$. 
Equivalently, Plancherel measure is the distribution of $\RSK$ shape associated to a random permutation in $\Sym{n}$ with the uniform distribution.

Asymptotically, the shape of a random Plancherel-distributed Young diagram
converges to a well-known limit shape discovered in the celebrated works of Logan-Shepp \cite{loganshepp} and Vershik-Kerov \cite{vershikkerov1, vershikkerov2}.
Below we present this shape in a parametrization which is not the simplest one, but the most convenient for our purposes. The reason for this choice of parametrization will become obvious in \cref{theorem:insertion-determinism-Plancherel}.

For $-2\leq u \leq 2$ and $0\leq w\leq 1$ we define:
\begin{align*}
\Omega(u) &= \frac{2}{\pi} \left( u \sin^{-1}\left(\frac{u}{2}\right) + \sqrt{4-u^2} \right),\\ 
F(u) &= \frac12 + \frac{1}{\pi} \left( \frac{u \sqrt{4-u^2}}{4} + \sin^{-1}\left( \frac{u}{2}\right) \right),\\
U(w) &=  F^{-1}(w),  \\
V(w) &=  \Omega\big( U(w) \big), \\
X(w)&= \frac{V(w)+U(w)}{2},\\
Y(w)&= \frac{V(w)-U(w)}{2},
\end{align*}
where $F^{-1}$ denotes the compositional inverse.
See \Cref{fig:VKLGcurve} for an illustration.

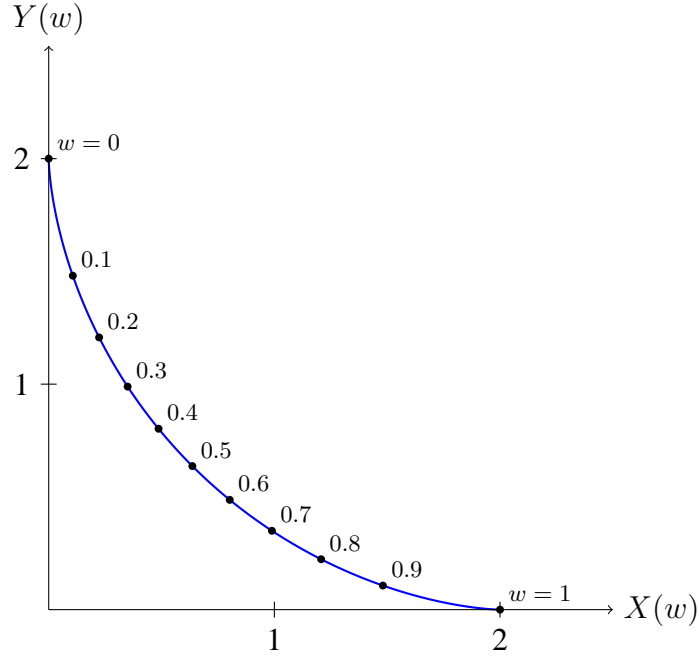
\begin{figure}[tb]
\centering
\begin{tikzpicture}[scale=3]
\scriptsize
\draw[blue,thick] plot[smooth] file {programy-jeu-de-taquin/VK.txt};
\fill (0                , 2               )   circle (0.5pt) node[anchor=south west] {$w=0$};
\fill (0.106813187877376, 1.48091084014257)   circle (0.5pt) node[anchor=south west] {$0.1$};
\fill (0.223445306166904, 1.20716897169440)   circle (0.5pt) node[anchor=south west] {$0.2$};
\fill (0.349746358408714, 0.989129377989765)  circle (0.5pt) node[anchor=south west] {$0.3$};
\fill (0.486819898276584, 0.802292285876628)  circle (0.5pt) node[anchor=south west] {$0.4$};
\fill (0.636619772367539, 0.636619772367539)  circle (0.5pt) node[anchor=south west] {$0.5$};
\fill (0.802292285876628, 0.486819898276584)  circle (0.5pt) node[anchor=south west] {$0.6$};
\fill (0.989129377989765, 0.349746358408714)  circle (0.5pt) node[anchor=south west] {$0.7$};
\fill (1.20716897169440, 0.223445306166904)   circle (0.5pt) node[anchor=south west] {$0.8$};
\fill (1.48091084014257, 0.106813187877376)   circle (0.5pt) node[anchor=south west] {$0.9$};
\fill (2                , 0               )   circle (0.5pt) node[anchor=south west] {$w=1$};
\normalsize
\draw[->] (0,0) -- (2.5,0) node[anchor=west]  {$X(w)$};
\draw[->] (0,0) -- (0,2.5) node[anchor=south] {$Y(w)$};
    	\foreach \x in {1,2}
     		\draw (\x,1pt) -- (\x,-1pt)
			node[anchor=north] {\x};
    	\foreach \y in {1,2}
     		\draw (1pt,\y) -- (-1pt,\y) 
     			node[anchor=east] {\y}; 
\end{tikzpicture}
\caption{Vershik-Kerov-Logan-Shepp limit shape and its parametrization $\big(X(w),Y(w)\big)$.} 
\label{fig:VKLGcurve}
\end{figure}

\begin{fact}[Typical shape of random, Plancherel distributed Young diagrams]
\label[fact]{theorem:plancherel-asymptotic-shape}
For each $n\geq 1$ let $\Lambda^n=(\Lambda^n_1,\Lambda^n_2,\dots)$ be a random Young diagram with $n$ boxes, distributed according to Plancherel measure.
Let $(y_n)$ be a sequence of positive integers with the property that
\[ \mathbf{y} := \lim_{n\to\infty} \frac{y_n}{\sqrt{n}}  >0. \]

Then the lengths of the rows of these Young diagrams behave asymptotically as follows:
\[ \frac{ \Lambda^n_{y_n} }{\sqrt{n}} \xrightarrow[n\to\infty]{P} X\big(Y^{-1}(\mathbf{y})\big),\] 
where $Y^{-1}$ denotes the compositional inverse.
Furthermore, the rate of convergence is given as follows:  
for each $\epsilon>0$ there exists some $d>0$ with the property that
\[ P\left( \left| \frac{ \Lambda^n_{y_n} }{\sqrt{n}} - X\big(Y^{-1}(\mathbf{y})\big) \right| > \epsilon \right) = O\left( e^{-d \sqrt{n}} \right).
\]
\end{fact}
\begin{proof}
Essentially, this result is a rather straightforward reformulation of the results of 
Vershik and Kerov. We provide the details below.

We consider the rotated (so called, Russian) coordinate system 
\[ u = x-y, \quad v = x+y \]
on the plane. \Cref{fig:french} shows how a Young diagram in the Russian coordinate system can be identified with its \emph{profile} which is just a function on the real line $\R$.

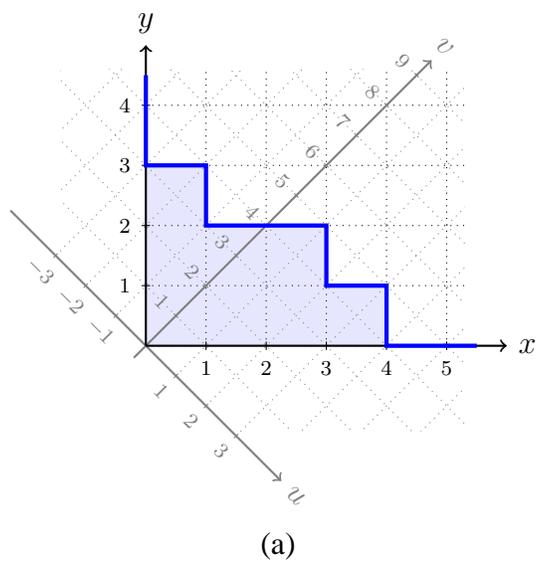
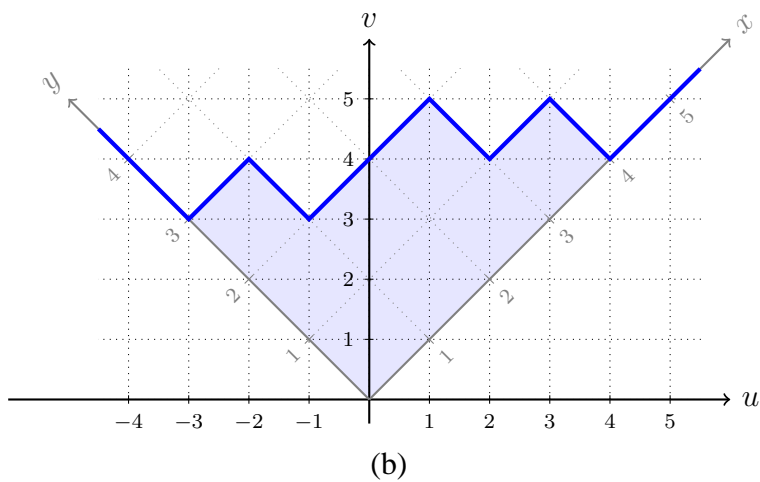
\begin{figure}[tbp]

\centering
\subfloat[]{
\begin{tikzpicture}[scale=0.8]

\begin{scope}[scale=1/sqrt(2),rotate=-45,draw=gray]

      \begin{scope}[draw=black,rotate=45,scale=sqrt(2)]
          \fill[fill=blue!10] (4,0) -- (4,1) -- (3,1) -- (3,2) -- (1,2) -- (1,3) -- (0,3) -- (0,0) -- cycle ;
      \end{scope}  

      \begin{scope}
          \clip[rotate=45] (-2,-2) rectangle (7.5,6.5);
          \draw[thin, dotted, draw=gray] (-10,0) grid (10,10);
          \begin{scope}[rotate=45,draw=black,scale=sqrt(2)]
              \draw[thin, dotted] (0,0) grid (15,15);
          \end{scope}
      \end{scope}

      \draw[->,thick] (-4.5,0) -- (4.5,0) node[anchor=west,rotate=-45]{\textcolor{gray}{$u$}};
      \foreach \z in { -3, -2, -1, 1, 2, 3}
            { \draw (\z, -2pt) node[anchor=north,rotate=-45] {\textcolor{gray}{\tiny{$\z$}}} -- (\z, 2pt); }

      \draw[->,thick] (0,-0.4) -- (0,9.5) node[anchor=south,rotate=-45]{\textcolor{gray}{$v$}};

      \foreach \t in {1, 2, 3, 4, 5, 6, 7, 8, 9}
            { \draw (-2pt,\t) node[anchor=east,rotate=-45] {\textcolor{gray}{\tiny{$\t$}}} -- (2pt,\t); }

      \begin{scope}[draw=black,rotate=45,scale=sqrt(2)]

          \draw[->,thick] (0,0) -- (6,0) node[anchor=west]{{{$x$}}};
          \foreach \x in {1, 2, 3, 4, 5}
              { \draw (\x, -2pt) node[anchor=north] {{\tiny{$\x$}}} -- (\x, 2pt); }

          \draw[->,thick] (0,0) -- (0,5) node[anchor=south] {{{$y$}}};
          \foreach \y in {1, 2, 3, 4}
              { \draw (-2pt,\y) node[anchor=east] {{\tiny{$\y$}}} -- (2pt,\y); }

          \draw[ultra thick,draw=blue] (5.5,0) -- (4,0) -- (4,1) -- (3,1) -- (3,2) -- (1,2) -- (1,3) -- (0,3) -- (0,4.5) ;

      \end{scope}
 
\end{scope}

\end{tikzpicture}
\label{subfigure:french}
}

\vspace{5ex}

\subfloat[]{
\begin{tikzpicture}[scale=0.8]

\begin{scope}[scale=1]
 
        \begin{scope}[draw=gray,rotate=45,scale=sqrt(2)]
          \fill[fill=blue!10] (4,0) -- (4,1) -- (3,1) -- (3,2) -- (1,2) -- (1,3) -- (0,3) -- (0,0) -- cycle ;
        \end{scope}

       \begin{scope}
          \clip (-4.5,0) rectangle (5.5,5.5);
          \draw[thin, dotted] (-6,0) grid (6,6);
          \begin{scope}[rotate=45,draw=gray,scale=sqrt(2)]
              \clip (0,0) rectangle (4.5,5.5);
              \draw[thin, dotted] (0,0) grid (6,6);
          \end{scope}
      \end{scope}

      \draw[->,thick] (-6,0) -- (6,0) node[anchor=west]{$u$};
      \foreach \z in {-4, -3, -2, -1, 1, 2, 3, 4, 5}
            { \draw (\z, -2pt) node[anchor=north] {\tiny{$\z$}} -- (\z, 2pt); }

      \draw[->,thick] (0,-0.4) -- (0,6) node[anchor=south]{$v$};
      \foreach \t in {1, 2, 3, 4, 5}
            { \draw (-2pt,\t) node[anchor=east] {\tiny{$\t$}} -- (2pt,\t); }

  \begin{scope}[draw=gray,rotate=45,scale=sqrt(2)]

          \draw[->,thick] (0,0) -- (6,0) node[anchor=west,rotate=45] {\textcolor{gray}{{$x$}}};
          \foreach \x in {1, 2, 3, 4, 5}
              { \draw (\x, -2pt) node[anchor=north,rotate=45] {\textcolor{gray}{\tiny{$\x$}}} -- (\x, 2pt); }

          \draw[->,thick] (0,0) -- (0,5) node[anchor=south,rotate=45] {\textcolor{gray}{{$y$}}};
          \foreach \y in {1, 2, 3, 4}
              { \draw (-2pt,\y) node[anchor=east,rotate=45] {\textcolor{gray}{\tiny{$\y$}}} -- (2pt,\y); }

          \draw[ultra thick,draw=blue] (5.5,0) -- (4,0) -- (4,1) -- (3,1) -- (3,2) -- (1,2) -- (1,3) -- (0,3) -- (0,4.5) ;

  \end{scope}

\end{scope}

\end{tikzpicture}
\label{subfigure:russian}
}

\caption{A Young diagram $\lambda=(4,3,1)$ shown in \protect\subref{subfigure:french} the French and \protect\subref{subfigure:russian} the Russian convention. The solid line represents the \emph{profile} of the Young diagram. The coordinates system $(u,v)$ corresponding to the Russian convention and the coordinate system $(x,y)$ corresponding to the French convention are shown.}

\label{fig:french}
\end{figure}

A slight variation of the results of Vershik and Kerov \cite{vershikkerov2} (it follows from the numerical estimates in Section~3 of that paper by modifying some parameters in an obvious way; see also \cite[Chapter~1]{Romik2013}) states that for each $\varepsilon>0$ there exists some $d=d(\varepsilon)>0$ with the property that the rescaled (by factor $\frac{1}{\sqrt{n}}$) profile of a Plancherel-random Young diagram with $n$ boxes is (with probability at least $1-O\left( e^{-d\sqrt{n}}\right)$) contained in an $\varepsilon$-neighborhood of the graph of the function $v=\Omega(u)$, see
\Cref{figure:epsilon-neighborhood}. 

The diagonal solid line on \Cref{figure:epsilon-neighborhood} shows the intersection of this neighborhood with the line $y=\frac{y_n}{\sqrt{n}}$;
we are interested in the $x$-coordinates of the points from this intersection since they correspond to (scaled by a factor $\frac{1}{\sqrt{n}}$) possible values of $\Lambda^n_{y_n}$. In the following we will show that as $\varepsilon\to 0$, the length of this intersection converges to zero uniformly over $\frac{y_n}{\sqrt{n}}> C$ for arbitrary $C>0$. 
This would imply that for each $\epsilon>0$ it is possible to choose $\varepsilon>0$ small enough that
% \[ \left| \frac{\Lambda_{y_n}^n}{\sqrt{n}} - X\left( Y^{-1} \left( \frac{y_n}{\sqrt{n}} \right) \right) \right| \xrightarrow[n\to\infty]{P} 0 \]
\[ P\left( \left| \frac{ \Lambda^n_{y_n} }{\sqrt{n}} - X\left(Y^{-1}\left( \frac{y_n}{\sqrt{n}} \right)\right) \right| > \epsilon \right) = O\left( e^{-d(\varepsilon) \sqrt{n}} \right)
\]
as the common point of the curve $\Omega$ and the line 
$y=\frac{y_n}{\sqrt{n}}$ belongs to the above intersection as well.
The continuity of the function $X\big( Y^{-1}( \cdot ) \big)$ would finish the proof.

\begin{figure}[tb]
\centering
\begin{tikzpicture}[scale=3]
\clip (-2.2,-0.1) rectangle (2.2,2);
\begin{scope}[rotate=45]
\begin{scope}
\clip (-0.5,-0.5) rectangle (2.3,2.3);
% \draw[blue,fill=blue!5,thick] 

\draw[blue!50,fill=blue!10] 
plot[smooth] file {programy-jeu-de-taquin/VK-A.txt}
plot[smooth] file {programy-jeu-de-taquin/VK-B.txt};

\draw[blue,thick] plot[smooth] file {programy-jeu-de-taquin/VK-C.txt};

\draw[blue,dashed,thick,<->] 
(0.106813187877376, 1.48091084014257) +(0.2,0.2) -- +(-0.2,-0.2);
\draw[blue,dashed,thick,<->] 
(0.486819898276584, 0.802292285876628) +(0.2,0.2) -- +(-0.2,-0.2);

\draw[red,dashed] (0,0.38)   -- (4,0.38);
\draw (0,0.38) +(1pt,0) -- +(-1pt,0) node[anchor=north east] {$\frac{y_n}{\sqrt{n}}$};
 
\begin{scope}
\clip 
plot[smooth] file {programy-jeu-de-taquin/VK-A.txt}
plot[smooth] file {programy-jeu-de-taquin/VK-B.txt};
\draw[very thick,red] (0,0.38) -- (2,0.38);
\end{scope}

\end{scope}

\draw[->] (0,0) -- (2.5,0) node[anchor=south west]  {$X(w)$};
\draw[->] (0,0) -- (0,2.5) node[anchor=south east] {$Y(w)$};
    	\foreach \x in {1,2}
     		\draw (\x,1pt) -- (\x,-1pt)
			node[anchor=north west] {\x};
    	\foreach \y in {1,2}
     		\draw (1pt,\y) -- (-1pt,\y) 
     			node[anchor=north east] {\y}; 
\end{scope}
\end{tikzpicture}
\caption{Logan-Shepp-Vershik-Kerov curve and its $\varepsilon$-neighborhood.} 
\label{figure:epsilon-neighborhood}
\end{figure}
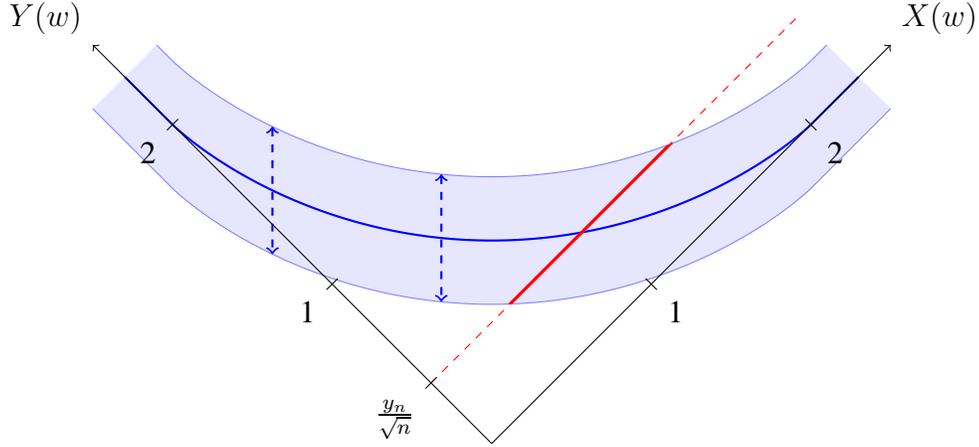

It remains to show that as $\varepsilon\to 0$, the length of the intersection converges to zero uniformly over $\frac{y_n}{\sqrt{n}}> C$ for arbitrary $C>0$. 
Let $(u_1,v_1)$, $(u_2,v_2)$ be the coordinates (in the Russian coordinate system) of some points on this intersection. This implies that their $y$-coordinates are equal:
\[ 2y=v_1-u_1= v_2-u_2.\]
On the other hand,
\[ \left| \Omega(u_i) - v_i \right| < \varepsilon \]
for each $i\in\{1,2\}$. Thus
\begin{equation}
\label{eq:omega-lipschitz-1}
  \Omega(u_2)- \Omega(u_1)  > u_2-u_1- 2\varepsilon. 
\end{equation}

Suppose that $u_i>2-\delta$ for some $\delta>0$. Note that $y$-coordinate of $(u_i,v_i)$ fulfills
\[ C< y \leq \frac{\Omega(u_i)-u_i+\varepsilon}{2};\]
as $\varepsilon\to 0$ and $\delta\to 0$, the right-hand side converges to zero, which leads to a contradiction. This shows that there exists $\delta>0$ such that $u_i<2-\delta$ for all $\varepsilon>0$ which are sufficiently small.

A direct calculation of the derivative shows that there exists $c>0$ such that 
$\Omega'(u) < 1-c$ for any $u\in (-\infty,2-\delta)$. Thus,
for any $u_1\leq u_2$ such that $u_1,u_2\in(-\infty, 2-\delta)$
\begin{equation}
\label{eq:omega-lipschitz-2}
\Omega(u_2)-\Omega(u_1) \leq (u_2 - u_1)(1-c).
\end{equation}
\Cref{eq:omega-lipschitz-1,eq:omega-lipschitz-2} show that if $\varepsilon\to 0$ then
$u_2-u_1 \to 0$ as well. This implies that the difference of the $x$-coordinates
$\frac{\Omega(u_i)+u_i}{2}$ converges to zero as well.
This concludes the proof that the length of the intersection converges to zero.
\end{proof}

\subsection{Asymptotic determinism of Schensted insertion for Plancherel measure}

In order to show \cref{theorem:insertion-determinism} we will need the following special case of it for $\alpha=\beta=(0,0,\dots)$ and $\gamma=1$ which has been proved in our previous work.
It explains our parametrization of Vershik-Kerov-Logan-Shepp curve:
$\big( X(w), Y(w) \big)$ is just the (rescaled) typical position of the newly created box by Schensted insertion, when $w\in(0,1)$ is inserted.

\begin{fact}[{\cite[Theorem 5.1]{RomikSniady2011}}]
\label[fact]{theorem:insertion-determinism-Plancherel}
Let $\Letter_1,\Letter_2,\dots$ be the sequence of random, i.i.d.~letters from the interval $(0,1)$, taken with the uniform distribution.
Let $\letter\in (0,1)$ be deterministic. 
Let $\Box_n$ denote the location of the last box added to the recording tableau by $\RSK$ algorithm applied to the sequence
\[ (\Letter_1,\dots,\Letter_{n-1},\letter).\]

Then 
\[ \frac{\Box_n}{\sqrt{n}} \xrightarrow[n\to\infty]{P} \big( X(w), Y(w) \big). 
\]
The rate of convergence is given, for each $\epsilon>0$, by
\[  P\left\{ \left\| \frac{\Box_n}{\sqrt{n}} - \big( X(w), Y(w) \big) \right\| > \epsilon \right\} = O\left( n^{-\frac{1}{4}} \right).\]
\end{fact}

\section{Asymptotic determinism of Schensted insertion}
\label{sec:asymptotic-determinism-schensted}

\subsection{The insertion alphabet}
\label{subsec:insertionalphabet}

The second most important example of an alphabet is $\alphabetinsertion=\alphabetinsertion_r\sqcup \alphabetinsertion_0 \sqcup \alphabetinsertion_c$ with $\alphabetinsertion_{r}=\{\dots,-3,-2,-1\}$, $\alphabetinsertion_0=(0,1)\subset\R$ and $\alphabetinsertion_{c}=\{1,2,3,\dots\}$
with the linear order defined as follows: on each of the sets $\alphabetinsertion_r$, $\alphabetinsertion_0$, $\alphabetinsertion_{c}$ we consider the natural order; we declare any element of $\alphabetinsertion_{c}$ smaller than any element of $\alphabetinsertion_0$, which is smaller than any element of $\alphabetinsertion_{r}$.  This linear order can be visualized as follows:
\[ \underbrace{1 < 2 < 3 < \cdots}_{\text{column letters}}< \underbrace{\cdots < 0.1 < \cdots< 0.9 < \cdots} <\underbrace{\cdots < -3 < -2 < -1}_{\text{row letters}}, \]
compare with \eqref{eq:jdt-alphabet}. 
This alphabet will be called \emph{the insertion alphabet}.

If $(\alpha,\beta,\gamma)$ belongs to Thoma simplex, we define the following probability measure $\alphabetmeasureinsertion_{\alpha,\beta,\gamma}$ on $\alphabetinsertion$:
\begin{itemize}
\item for $-i\in\{-1,-2,-3,\dots\}=\alphabetinsertion_{r}$ we set $\alphabetmeasureinsertion_{\alpha,\beta,\gamma}(-i)=\alpha_i$; 
\item for $i\in\{1,2,3,\dots\}= \alphabetinsertion_{c}$ we set $\alphabetmeasureinsertion_{\alpha,\beta,\gamma}(i)=\beta_i$; 
\item on $\alphabetinsertion_0=(0,1)$ we take as $\alphabetmeasureinsertion_{\alpha,\beta,\gamma}$ the absolutely continuous measure on the unit interval $(0,1)$ with constant density $\gamma$.
\end{itemize}

This alphabet and the measure are the ones used in \cref{theorem:insertion-determinism}.

\subsection{The opposite alphabets}
The alphabets $\alphabetjeu$ and $\alphabetinsertion$, regarded as ordered sets, are equal. However, since their decompositions $\alphabet=\alphabet_r \sqcup \alphabet_0 \sqcup \alphabet_c$ into row letters and column letters are different, this equality turns out to be not very important.

It is much more convenient to consider the bijection $\isomorphismItself:\alphabetjeu \rightarrow \alphabetinsertion$ defined by
\begin{align*} 
\alphabetjeu_r = \{1,2,\dots\} \ni k & \mapstochar\xrightarrow{\isomorphismItself} -k \in\{-1,-2,-3,\dots\} = \alphabetinsertion_r, \\
\alphabetjeu_c = \{-1,-2,\dots\} \ni -k & \mapstochar\xrightarrow{\isomorphismItself} k \in\{1,2,3,\dots\} = \alphabetinsertion_c, \\
\alphabetjeu_0 = (0,1) \ni x & \mapstochar\xrightarrow{\isomorphismItself} 1-x \in (0,1) = \alphabetinsertion_0.
\end{align*}
The map $\isomorphismItself$ is an anti-isomorphism of ordered sets which preserves the decomposition of the alphabets $\alphabet=\alphabet_r \sqcup \alphabet_0 \sqcup \alphabet_c$ into row letters and column letters.
Furthermore, the pushforward of $\alphabetmeasurejeu_{\alpha,\beta,\gamma}$ is equal to $\alphabetmeasureinsertion_{\alpha,\beta,\gamma}$.

\subsection{Duality between \jdt and Schensted insertion}

The following lemma shows that Schensted insertion and \jdt are closely related to each other.
\begin{lemma}
\label[lemma]{lem:duality}
Let $\letter_1,\dots,\letter_n\in \alphabetjeu$. 
We assume that each neutral letter appears at most once in this tuple.
Let $Q:=Q(\letter_1,\dots,\letter_n)$ be the corresponding recording tableau and let $\jpathlazy_n$ be the box where the finite version of \jdt leaves tableau $Q$.

We consider the tuple $\isomorphism{\letter_n},\dots,\isomorphism{\letter_1}\in\alphabetinsertion$ and the corresponding recording tableau
$Q':=Q\big(\isomorphism{\letter_n},\dots,\isomorphism{\letter_1}\big)$. Let $\Box_n$ be the box with the label $n$ in $Q'$ (i.e., it is the box added in the last Schensted insertion step).

Then $\jpathlazy_n=\Box_n$.
\end{lemma}
\begin{proof}
Let $\lambda_n$ be the $\RSK$ shape associated to $(\letter_1,\letter_2,\dots,\letter_n)$.
By Greene's theorem (\cref{theorem:greene}) it follows that it is also the $\RSK$ shape associated to $\big(\isomorphism{\letter_n},\dots,\isomorphism{\letter_2},\isomorphism{\letter_1}\big)$.

Let $\lambda_{n-1}$ be the $\RSK$ shape associated to the postfix $(\letter_2,\dots,\letter_n)$.
By the same argument it follows that it is also the $\RSK$ shape associated to $\big(\isomorphism{\letter_n},\dots,\isomorphism{\letter_2}\big)$.

By definition, $\lambda_n$ is the shape of $Q$; 
\cref{lem:factor-map} shows that $\lambda_{n-1}$ is the shape of $j(Q)$ thus
\begin{align*} 
\{ \jpathlazy_n \} = & \lambda_n \setminus \lambda_{n-1}. \\
\intertext{On the other hand,} 
 \{ \Box_n \} = & \lambda_n \setminus \lambda_{n-1} 
\end{align*}
which finishes the proof.
\end{proof}

\begin{remark}
\cref{lem:duality} holds true in bigger generality with the alphabets $\alphabetjeu$ and $\alphabetinsertion$ replaced by arbitrary alphabets $\alphabet$, $\mathbb{B}$ with the property that there exists anti-isomorphism of ordered sets $\isomorphismItself:\alphabet\to\mathbb{B}$ which preserves the decompositions $\alphabet=\alphabet_r\sqcup \alphabet_0 \sqcup \alphabet_c$.
\end{remark}

\begin{lemma}
\label[lemma]{proba-is-decreasing}
Let $(\alpha,\beta,\gamma)$ be an element of Thoma simplex.
Let $\Letter_1,\Letter_2,\dots$ be the sequence of random, i.i.d.~letters from the insertion alphabet $\alphabetinsertion$, with probability distribution $\alphabetmeasureinsertion_{\alpha,\beta,\gamma}$. Let $\letter\in\alphabetinsertion$ be a deterministic letter. 
Let $\Box_n$ denote the location of the last box added to the recording tableau by $\RSK$ algorithm applied to the sequence
\[ (\Letter_1,\dots,\Letter_{n-1},\letter).\]

Then for arbitrary $k\in\{1,2,\dots\}$
\begin{equation}
\label{eq:proba-first-rows}
 \Big( P\big( \text{$\Box_n$ is in one of the first $k$ rows}\big) \Big)_{n}  
\end{equation}
is a weakly decreasing sequence. 
\end{lemma}
\begin{proof}
We apply \cref{lem:duality}; it implies that the sequence \eqref{eq:proba-first-rows} coincides with
\begin{equation}
\label{eq:proba-first-rows-2}
  \Big( P\big( \text{$\jpathlazy_n(Q_n)$ is in one of the first $k$ rows}\big) \Big)_{n},  
\end{equation}
where $Q_n$ is defined as the recording tableau associated to the sequence $\big(\isomorphism{w},\isomorphism{W_{n-1}},\dots,\isomorphism{W_2},\isomorphism{W_1}\big)$. It does not change the sequence \eqref{eq:proba-first-rows-2} if we change the definition of $Q_n$ to be the recording tableau associated to the sequence $\big(\isomorphism{w},\isomorphism{W_1},\isomorphism{W_2},\dots,\isomorphism{W_{n-1}}\big)$. In particular, the sequence \eqref{eq:proba-first-rows-2} coincides with the sequence
\begin{equation}
\label{eq:proba-first-rows-3}
  \Big( P\big( \text{$\jpathlazy_n(Q)$ is in one of the first $k$ rows}\big) \Big)_{n},  
\end{equation}
where $Q:=\RSK\big(\isomorphism{w},\isomorphism{W_1},\isomorphism{W_2},\dots\big)\in\tableaux$.
Clearly, for any tableau $Q$, the sequence $y\big(\jpathlazy_n(Q)\big)$ 
of $y$-coordinates is weakly increasing which immediately implies that \eqref{eq:proba-first-rows-3} and thus \eqref{eq:proba-first-rows} are weakly decreasing.
\end{proof}

\subsection{Asymptotic determinism of Schensted insertion}

The proofs of our results will be based on the following technical result, which might be interesting on its own.

% In \cref{subsec:insertionalphabet} we presented a special choice of the alphabet $\alphabetinsertion$ (\emph{the insertion alphabet}) with a special choice of the probability distribution $\alphabetmeasureinsertion_{\alpha,\beta,\gamma}$ on $\alphabetinsertion$ for which the following result holds true.

\begin{theorem}[Asymptotic determinism of Schensted insertion]
\label[theorem]{theorem:insertion-determinism}
Let $(\alpha,\beta,\gamma)$ be an element of Thoma simplex.
Let $\Letter_1,\Letter_2,\dots$ be the sequence of random, i.i.d.~letters from the insertion alphabet $\alphabetinsertion$, with probability distribution $\alphabetmeasureinsertion_{\alpha,\beta,\gamma}$. Let $\letter\in\alphabetinsertion$ be a deterministic letter. 
Let $\Box_n$ denote the location of the last box added to the recording tableau by $\RSK$ algorithm applied to the sequence
\[ (\Letter_1,\dots,\Letter_{n-1},\letter).\]

\begin{enumerate}[label=\emph{(\Alph*)}]
\item \label[empty]{item:insertion-Row}
In the case when $\letter=-k\in \{-1,-2,-3,\dots\}=\alphabetinsertion_r$ we assume that $\alpha_k>0$.
Then
\[ \lim_{n\to\infty} P\big( \text{$\Box_n$ is in row $k$} \big) =1. \]

 \item \label[empty]{item:insertion-Column}
In the case when $\letter=k\in \{1,2,3,\dots\}=\alphabetinsertion_c$ we assume that $\beta_k>0$.
Then
\[ \lim_{n\to\infty} P\big( \text{$\Box_n$ is in column $k$} \big) =1. \]

\item \label[empty]{item:insertion-Diagonal}
In the case when $\letter\in (0,1)= \alphabetinsertion_0$ we assume that $\gamma>0$. Then
$$  \frac{\Box_n}{\sqrt{\gamma n}} \xrightarrow[n\to\infty]{P} \big( X(w), Y(w) \big), $$
where $\big( X(w), Y(w) \big)$ is the parametrization of Vershik-Kerov-Logan-Shepp curve considered in \Cref{fig:VKLGcurve}.
The rate of convergence is given, for any $\epsilon>0$, by
\begin{equation}
\label{eq:rate-of-convergence}
  P\left\{ \left\| \frac{\Box_n}{\sqrt{\gamma n}} - 
\big( X(w), Y(w) \big) \right\| > \epsilon    \right\} = O\left( n^{-\frac{1}{4}} \right). 
\end{equation}
% $\theta=\frac{\pi}{2}-(F_\Theta)^{-1}(\letter) $ and $F_\Theta$ is the cumulative distribution function of the distribution considered in \eqref{eq:theta-dist}.
\end{enumerate}

\bigskip

In each of the above three cases,
\begin{equation}
\label{eq:diverges}
%  \lim_{n\to\infty} P\Big(\|\Box_n\| > C\Big) =1 
 \|\Box_n\|  \xrightarrow[n\to\infty]{P} \infty.
\end{equation}
% holds true for arbitrary constant $C$.

Informally speaking, $\Box_n$ converges in probability (as $n\to\infty$) to the appropriate asymptote depicted in \cref{fig:compactification-twisted}. 
\end{theorem}
\begin{proof}
We will consider each of the three cases separately.

\bigskip

\emph{The case \cref{item:insertion-Row}.} 

The elements which are bumped from consecutive rows in a given Schensted insertion step form an \mbox{$<_c$-increasing} sequence. The insertion alphabet $\alphabetinsertion$ has the property that any \mbox{$<_c$-increasing} sequence of its elements which starts with $-k$ is of length (at most) $k$. This shows that $\Box_n$ belongs to one of the first $k$ rows. 
Thus it remains to show that  
\[ \lim_{n\to\infty} P\big( \text{$\Box_n$ is in one of the first $k-1$ rows} \big) =0. \]

Let $\Lambda^n=(\Lambda^n_1,\Lambda^n_2,\dots)$ be the $\RSK$ shape associated to the sequence $(W_1,\dots,W_n)$ of i.i.d.~random letters from $\alphabetinsertion$ distributed according to the probability measure
$\alphabetmeasureinsertion_{\alpha,\beta,\gamma}$.
We use the notational shorthand
\[ [\text{\emph{condition}}] =\begin{cases} 1 & \text{if \emph{condition} is true}, \\ 0 & \text{otherwise}. \end{cases}\]
We clearly have
\begin{multline*} \Lambda^n_1 + \cdots +\Lambda^n_{k-1} = \\
\sum_{1\leq m\leq n} \Big[\text{the box created in the insertion 
$P(W_1,\dots,W_{m-1})\leftarrow W_m$} \\
\text{is in one of the first $k-1$ rows}\Big]
\end{multline*}
thus, by considering the events $W_m\in\{-1,-2,\dots,-(k-1)\}$ and $W_m=-k$, we obtain
\begin{multline*} \E \left( \Lambda^n_1 + \cdots +\Lambda^n_{k-1} \right) \geq \\
\sum_{1\leq m \leq n}\Big[ \alpha_1+\cdots+\alpha_{k-1}+\alpha_k \
 P\big( \text{$\Box_m$ is in one of the first $k-1$ rows}\big) \Big].
\end{multline*} 
This, together with \cref{proba-is-decreasing} implies that
\begin{multline}
\label{proba-in-first-rows} 
\liminf_{n\to\infty} \E \frac{ \Lambda^n_1 + \cdots +\Lambda^n_{k-1}}{n} \geq \\
 \alpha_1+\cdots+\alpha_{k-1}+\alpha_k \lim_{n\to\infty}
 P\big( \text{$\Box_n$ is in one of the first $k-1$ rows}\big).
\end{multline}

On the other hand, \cref{theo:KerovVershik-RSK-homomorphism} implies that $(\Lambda^0\nearrow\Lambda^1\nearrow\cdots)$ is a random infinite Young tableau with the distribution given by Vershik-Kerov measure $\measure_{\alpha,\beta,\gamma}$ thus  \cref{theo:VK-frequencies} can be applied. By Lebesgue's dominated convergence theorem
\begin{equation}
\label{VK-frequencies}
 \lim_{n\to\infty} \E \frac{\Lambda^n_1 + \cdots +\Lambda^n_{k-1}}{n} = 
\alpha_1+\cdots+\alpha_{k-1}. 
\end{equation}

\cref{proba-in-first-rows,VK-frequencies} finish the proof.

In order to show \eqref{eq:diverges} in this case it enough to use that $\Lambda_1^n,\dots,\Lambda_k^n$ all tend almost surely to infinity.

\bigskip

\emph{The case \cref{item:insertion-Column}.} 

The insertion alphabet $\alphabetinsertion$ has the property that any $<_r$-increasing sequence of its elements which ends with $k$ is of length (at most) $k$.
This implies that when $k$ is inserted by Schensted insertion to an arbitrary tableau, it is inserted into one of the first $k$ columns, thus $\Box_n$ belongs to one of the first $k$ columns as well.
Thus it remains to show that  
\[ \lim_{n\to\infty} P\big( \text{$\Box_n$ is in one of the first $k-1$ columns} \big) =0. \]
The remaining part of the proof is completely analogous to the case \cref{item:insertion-Row} considered above; one should simply replace the notion of \emph{rows} by \emph{columns}, the lengths of rows $\Lambda^n_1,\Lambda^n_2,\dots$ should be replaced by the lengths of columns 
$(\Lambda^n)'_1,(\Lambda^n)'_2,\dots$, and one should consider the events $W_m\in\{1,2,\dots,k-1\}$ and $W_m=k$.

\bigskip

\emph{The case \cref{item:insertion-Diagonal}.} 

Our goal is to find the location $(x_1,y_1)$ of the box containing $n$ in the recording tableau corresponding to $\Word:=(\Letter_1,\dots,\Letter_{n-1},\letter)$.
Let $\PI=(\pi_1,\dots,\pi_n)$ be the permutation given by standardization 
(see \cref{subsec:rectification}) of the sequence $\Word$.
By \cref{lem:rectification-of-a-sequence}, the recording tableaux corresponding to $\Word$ and $\PI$ are equal. The latter recording tableau is equal to the insertion tableau  $P(\PI^{-1})$. In the remaining part of the proof we will be studying this insertion tableau.
We use the shorthand notation $(\pi^{-1}_1,\dots,\pi^{-1}_n):=\PI^{-1}$.

Let $r$ (respectively, $c$) denote the number of row letters (respectively, column letters) in $\Word$. 
We define
\begin{align*}
\PI^{-1}_c&:= (\pi^{-1}_1,\dots,\pi^{-1}_c), \\
\PI^{-1}_0&:= (\pi^{-1}_{c+1},\dots,\pi^{-1}_{n-r}), \\
\PI^{-1}_r&:= (\pi^{-1}_{n+1-r},\dots,\pi^{-1}_{n})
\end{align*}
so that $\PI^{-1}$ is a concatenation of the words $\PI^{-1}_c$, $\PI^{-1}_0$, $\PI^{-1}_r$.
In this way the insertion tableau $P(\PI^{-1})$ can be obtained by stacking 
the insertion tableaux $P(\PI^{-1}_c)$, $P(\PI^{-1}_0)$, $P(\PI^{-1}_r)$ as shown in \cref{fig:stacking} \label{page:stacking}
and by performing Sch\"utzenberger's \jdt. 

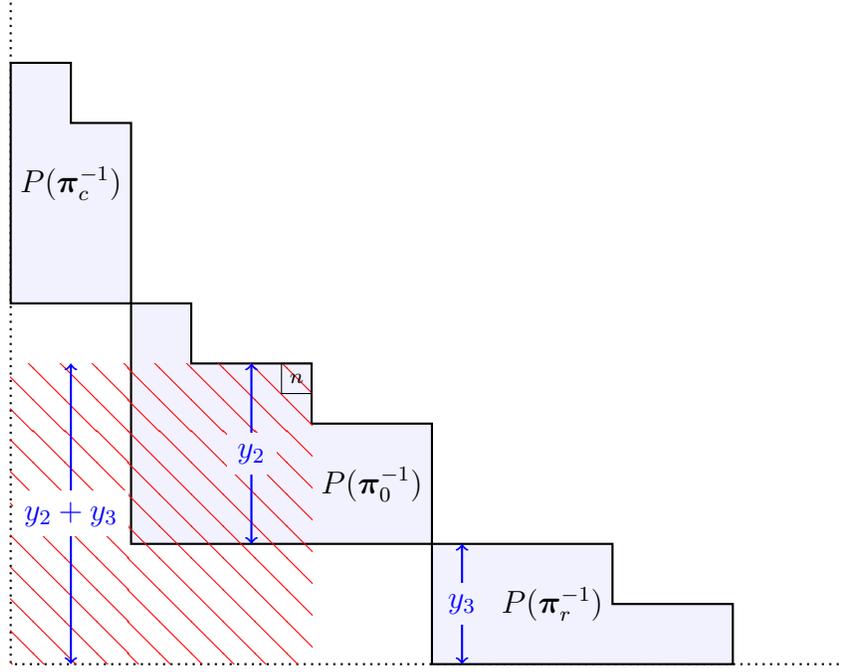
\begin{figure}
\centering
\begin{tikzpicture}[scale=0.8]
\draw[thick,dotted] (0,11) -- (0,0) -- (14,0);
\draw[fill=blue!5,thick] (0,6) +(0,0) -- +(0,4) -- +(1,4) -- +(1,3) -- +(2,3) -- +(2,0) -- cycle
+(1,2) node {$P(\PI^{-1}_c)$}; 
\draw[fill=blue!5,thick] (2,2) +(0,0) -- +(5,0) -- +(5,2) -- +(3,2) -- +(3,3) -- +(1,3) --
+(1,4) -- +(0,4) -- cycle;
\draw[fill=blue!5,thick] (7,0) +(0,0) -- +(5,0) -- +(5,1) -- +(3,1) -- +(3,2) -- +(0,2) -- cycle
+(2,1) node {$P(\PI^{-1}_r)$};
\fill[pattern=north west lines D, pattern color=red] (0,0) rectangle (5,5);  

\draw[<->,blue,thick] (7.5,0) -- (7.5,2) node[midway,fill=blue!5] {$y_3$};
\draw[<->,blue,thick] (1,0) -- (1,5) node[midway,fill=white]{$y_2+y_3$};
\draw[<->,blue,thick] (4,2) -- (4,5) node[midway,fill=blue!5]    {$y_2$};

\draw (4,1.5) +(2,1.5) node {$P(\PI^{-1}_0)$};
\draw (2,2) +(3,3) rectangle +(2.5,2.5) +(2.75,2.75) node {\tiny $n$};
\end{tikzpicture}
\caption{The way of stacking the insertion tableaux corresponding to $\PI^{-1}_c$, $\PI^{-1}_0$, $\PI^{-1}_r$.
%  used on page \pageref{page:stacking}.
As Sch\"utzenberger's \jdt involves only slides of the boxes down and to the left, the hatched area shows possible locations of the box $n$ in $P(\PI^{-1})$.}
\label{fig:stacking}
\end{figure}

The entries of $\PI^{-1}_c$ (respectively, $\PI^{-1}_r$)  are the locations in the word $\Word$ of the column letters (respectively, row letters); in particular $n$ is one of the entries of $\PI^{-1}_0$.

Let $(x_2,y_2)$ be the location of the box containing $n$ in the insertion tableau
$P(\PI^{-1}_0)$; let $y_3$ be the number of rows of the insertion tableau  $P(\PI^{-1}_r)$ and let $x_4$ be the number of columns of the insertion tableau $P(\PI^{-1}_c)$.
Thus 
\begin{align*}
 x_1 \leq & x_2+x_4, \\
 y_1 \leq & y_2+y_3 
\end{align*}
(the proof of the second inequality is illustrated in \cref{fig:stacking}; the
proof of the first inequality is analogous).

Word $\PI^{-1}$ is created from the word  $\PI^{-1}_0$ by (i) adding a postfix $\PI^{-1}_r$ and then (ii) adding a prefix $\PI^{-1}_c$; it follows that the
the insertion tableau $P(\PI^{-1})$ can be created from  $P(\PI^{-1}_0)$ by (i) a sequence of \emph{row insertions} of the letters forming $\PI_r^{-1}$ (this part of the claim follows from the definition \eqref{eq:definition-insertion} of the insertion tableau), followed by (ii) a sequence of \emph{column insertions} of the letters (in the reverse order) forming $\PI^{-1}_c$ (for this part of the claim and for the definition of the column insertion see \cite[Section A.2]{Fulton1997}). It follows that
the $\RSK$ shape corresponding to $\PI^{-1}$ contains the $\RSK$ shape
corresponding to $\PI^{-1}_0$. 

Let $(\lambda_1,\lambda_2,\dots)$ denote the $\RSK$ shape
corresponding to $\PI^{-1}_0$. As $(x_1,y_1)$ is one of the inner corners of 
$P(\PI^{-1})$,
\begin{align*} 
x_1 \geq & \lambda_{y_1} \geq \lambda_{y_2+y_3},\\
y_1 \geq & \lambda'_{x_1} \geq \lambda'_{x_2+x_4}.
\end{align*}

For a moment let us condition over the value of $c$.
The $\RSK$ shape corresponding to $\PI^{-1}_c$ depends only on the relative order of its entries $(\pi^{-1}_1,\dots,\pi^{-1}_c)$ which are the positions of the column letters in the tuple $\Word$. This order would not change if we remove from $\Word$ all letters which are not column letters.
It follows that the number of columns of $\RSK$ shape corresponding to $\PI^{-1}_c$ has the same distribution as the number of columns of $\RSK$ shape corresponding to a sequence of length $c$ of i.i.d.~letters from $\alphabetinsertion_c$ such that the probability of the letter $k$ is equal to $\frac{\beta_k}{\beta_1+\beta_2+\cdots}$.
The number of columns can only increase if we increase the length of the sequence to $n$; thus, by \cref{lemma:row-letters-no-quadratic-height}, we have unconditional convergence
\begin{align}
\label{eq:convergence-x4}
 \frac{x_4}{\sqrt{n}} & \xrightarrow[n\to\infty]{P} 0. \\
\intertext{An analogous reasoning shows that}
\label{eq:convergence-y3}
 \frac{y_3}{\sqrt{n}} & \xrightarrow[n\to\infty]{P} 0. 
\end{align}

We denote by $n':=n-r-c$ the length of the sequence $\PI^{-1}_0$, which is the number of the elements of the tuple $(W_1,\dots,W_{n-1},w)$ which belong to $\alphabetinsertion_0$. By the law of large numbers, 
\begin{equation}  
\label{eq:fraction-converges}
\frac{n'}{n} \xrightarrow[n\to\infty]{P} \gamma,
\end{equation}
% \[ \lim_{n\to\infty} \frac{n'}{n}=\gamma> 0\]
% holds almost surely, 
in particular
\[  n' \xrightarrow[n\to\infty]{P} \infty.\]
% \[ \lim_{n\to\infty} n' = \infty\]
% holds almost surely.
Thus by \cref{theorem:insertion-determinism-Plancherel} 
\begin{align}
\label{eq:convegenceAAA1}
 \frac{x_2}{\sqrt{n}} & \xrightarrow[n\to\infty]{P} \sqrt{\gamma}\ X(w),\\
\label{eq:convegenceAAA2}
 \frac{y_2}{\sqrt{n}} & \xrightarrow[n\to\infty]{P} \sqrt{\gamma}\ Y(w).
\end{align}

Thus we have shown that
\[  \frac{\lambda_{y_2+y_3}}{\sqrt{n}} \leq \frac{x_1}{\sqrt{n}} \leq \frac{x_2+x_4}{\sqrt{n}}; \]
the right-hand side converges in probability to $\sqrt{\gamma}\ X(w)$;
by \cref{theorem:plancherel-asymptotic-shape} the left-hand side also converges in probability to the same limit. Thus we have shown that 
\begin{align*}
\frac{x_1}{\sqrt{n}} & \xrightarrow[n\to\infty]{P} \sqrt{\gamma}\ X(w),\\
\intertext{as required. Proof of the other limit}
\frac{y_1}{\sqrt{n}} & \xrightarrow[n\to\infty]{P} \sqrt{\gamma}\ Y(w)
\end{align*}
follows in an analogous way.

In order to show \eqref{eq:rate-of-convergence} it is enough to revisit the above proof and check the rates of convergence in Eqs.~\eqref{eq:convergence-x4}--\eqref{eq:convegenceAAA2}.

As $X(w),Y(w)>0$ for $w\in(0,1)$, \Cref{eq:diverges} follows immediately.
\end{proof}

\section{Proofs of the main results}
\label{sec:proof-of-main-results}

\subsection{Asymptotic determinism of \jdt}

The following result is the final tool necessary in order to show the main results of the paper.
\begin{theorem}[Asymptotic determinism of \jdt]
\label[theorem]{theo:determinism-of-jdt}  
Let $(\alpha,\beta,\gamma)$ be an element of Thoma's simplex.
Let $\Letter_1,\Letter_2,\dots$ be a sequence of i.i.d.~random letters in $\alphabetjeu$ with the distribution $\alphabetmeasurejeu_{\alpha,\beta,\gamma}$. Let $w\in\alphabetjeu$ be fixed. Let  $\jpathlazy_n$ be the natural parametrization of the \jdt path associated with the random infinite Young tableau
\[ \RSK(w,\Letter_1,\Letter_2,\dots).\]

\begin{enumerate}[label=\emph{(\Alph*)}]
\item \label[empty]{item:jdt-Row}
In the case when $w=k\in\{1,2,3,\dots\}=\alphabetjeu_r$ we assume that $\alpha_k>0$.

Then, almost surely, \jdt trajectory stabilizes in $k$-th row:
\[ \lim_{n\to\infty} y(\jpathlazy_n) = k.\]

\item \label[empty]{item:jdt-Column}
In the case when $w=-k\in\{-1,-2,-3,\dots\}=\alphabetjeu_c$ we assume that $\beta_k>0$.

Then, almost surely, \jdt trajectory stabilizes in $k$-th column:
\[ \lim_{n\to\infty} x(\jpathlazy_n) = k.\]

\item \label[empty]{item:jdt-Diagonal}
In the case when $w\in (0,1)=\alphabetjeu_0$ we assume that $\gamma>0$.

Then, almost surely, 
\begin{equation}
\label{eq:theta-theta}
  \lim_{n\to\infty} \frac{\jpathlazy_n}{\sqrt{n}}  =
\sqrt{\gamma}\cdot \big( X(1-w), Y(1-w) \big). 
\end{equation}

\end{enumerate}

In all three above cases,
\[ \lim_{n\to\infty} \left\| \jpathlazy_n \right\| = \infty \]
holds almost surely.
\end{theorem}

Note that while in \Cref{theorem:insertion-determinism} the convergence holds only in the sense of convergence in probability, in the above theorem the convergence is in the almost sure sense.

\begin{proof}
The proof which we provide below is analogous to the proof of \cite[Theorem 5.2]{RomikSniady2011}. 

Let $\Box_n$ be the box with the label $n$ in 
\[ \RSK\big( \isomorphism{\Letter_{n-1}}, \isomorphism{\Letter_{n-2}}, \dots,
\isomorphism{\Letter_1}, \isomorphism{w}\big).\]
By \Cref{lem:duality}, 
\[\jpathlazy_n=\Box_n.\]
\Cref{theorem:insertion-determinism} can be applied in order to study the asymptotic behavior of the right-hand side; we will discuss the three cases separately.

\bigskip

\emph{The case \cref{item:jdt-Row}.} 
\Cref{theorem:insertion-determinism} shows that $y(\jpathlazy_n)$ converges to $k$ in probability. Since $y(\jpathlazy_n)$ is a weakly increasing sequence, the limit 
$\lim_{n\to\infty} y(\jpathlazy_n)\in\{1,2,\dots,\infty\}$ exists almost surely; this implies that 
$\lim_{n\to\infty} y(\jpathlazy_n)=k$ holds almost surely, as required.

\bigskip

\emph{The case \cref{item:jdt-Column}.} 
This case is analogous to the case \cref{item:jdt-Row} considered above.

\bigskip

\emph{The case \cref{item:jdt-Diagonal}.} 
Setting $n_m=m^8$, from \Cref{theorem:insertion-determinism} and 
Borel-Cantelli lemma we show an almost sure convergence
\[ \lim_{m\to\infty} \frac{\jpathlazy_{n_m}}{\sqrt{n_m}} = \sqrt{\gamma}\cdot \big( X(1-w), Y(1-w) \big) \]
along the subsequence $n=n_m$. 
Finally, note that $n_{m+1}/n_m \to 1$ as $m\to\infty$. It is easy to see that this, together with the fact that the path $(\jpathlazy_{n})_n$ advances monotonically in both the $x$ and $y$ directions, guarantees (deterministically) that convergence along the subsequence implies convergence for the entire sequence.

\bigskip

In all three above cases, the sequence $\left\| \jpathlazy_n \right\|$ is weakly increasing and \Cref{theorem:insertion-determinism} guarantees that
$\left\| \jpathlazy_n \right\| \xrightarrow[n\to\infty]{P} \infty$;  this implies that $\lim_{n\to\infty} \left\| \jpathlazy_n \right\| =\infty$ holds almost surely.
\end{proof}

\subsection{Proof of \Cref{theo:paths}}

\begin{proof}[Proof of \Cref{theo:paths}]
Again, without loss of generality, we can take 
\[T:=\RSK(\Letter_1,\Letter_2,\dots),\]
where $\Letter_1,\Letter_2,\dots$ is a sequence of i.i.d.~random letters from
$\alphabetjeu$ with the probability distribution $\alphabetmeasurejeu_{\alpha,\beta,\gamma}$.  We apply \Cref{theo:determinism-of-jdt}; 
the asymptotic behavior of \jdt path depends only on the value of $\Letter_1$, the first letter. 
Note that \jdt path is parametrized in a different way in \Cref{theo:paths} and in \cref{theo:determinism-of-jdt}; this difference, however, creates no difficulties.
\end{proof}

\subsection{Probability distribution of \jdt asymptotic angles}

\begin{proposition}
\label[proposition]{theorem:explicit-distribution-Theta}
We keep notations from \cref{theo:paths}.
If $\gamma>0$, the distribution of the asymptotic angle $\Theta$ 
(conditioned under event that the case \ref{enum:C} holds true) is an absolutely continuous random variable on $(0,\pi/2)$ whose distribution has the following explicit description:
\begin{equation} 
\label{eq:theta-dist} 
\Theta \equalindist \Pi(W),
\end{equation}
where $W$ is a random variable distributed according to the semicircle distribution $\SClaw$ on $[-2,2]$, i.e., having density given by
\begin{equation} \label{eq:semicircle}
\SClaw(dw) = \frac{1}{2 \pi} \sqrt{4-w^2}\,dw \qquad (|w|\le 2),
\end{equation}
and $\Pi(\cdot)$ is the function 
\[ \Pi(w) = \frac{\pi}{4}-\cot^{-1}\left[\frac{2}{\pi} \left( \sin^{-1}\left(\frac{w}{2}\right)+\frac{\sqrt{4-w^2}}{w} \right) \right]
\quad (-2 \le w \le 2).\]
\end{proposition}
\begin{proof}
Since \Cref{theo:paths} depends only on the distribution of the random infinite tableau $T$, without loss of generality we can assume, by \cref{theo:KerovVershik-RSK-homomorphism}, that $T=\RSK(W_1,W_2,\dots)$, where $W_1,W_2,\ldots\in\alphabetjeu$ is a sequence of i.i.d.~random letters with the distribution $\alphabetmeasurejeu_{\alpha,\beta,\gamma}$.

By \Cref{theo:determinism-of-jdt}  it follows that (as long as $\gamma>0$)
the conditional distribution of $\Theta$ does not depend on the element $(\alpha,\beta,\gamma)$ of Thoma's simplex.
Again, the difference of parametrizations of \jdt paths creates no difficulties.
% 
% note that \jdt path is parametrized in a different way in \Cref{theo:paths} and in \cref{theo:determinism-of-jdt}; this difference, however, creates no difficulties.
In particular, this conditional distribution coincides with the unconditional distribution of $\Theta$ in the case $\alpha=\beta=(0,0,\dots)$, $\gamma=1$ which corresponds to Plancherel measure. The result in this special case has been proved in our previous paper \cite[Theorem 1.1]{RomikSniady2011}.
\end{proof}

\subsection{Proof of \cref{theorem:isomorphism-dynamical,theo:measure-preserving-and-ergodic,theorem:inverse-rsk}. }

\begin{proof}[Proof of \cref{theorem:isomorphism-dynamical,theo:measure-preserving-and-ergodic,theorem:inverse-rsk}]

\Cref{theorem:isomorphism-dynamical,theo:measure-preserving-and-ergodic,theorem:inverse-rsk}
contain several claims:
\begin{itemize}
\item
\emph{$\RSK$ is a homomorphism of probability spaces (this is a part of \Cref{theorem:isomorphism-dynamical}).} 

This has been shown by Kerov and Vershik, see \cref{theo:KerovVershik-RSK-homomorphism}.

\item
\emph{$\RSK$ is a factor map of dynamical systems, i.e.
\begin{equation}
\label{eq:factor-map}
J\circ \RSK= \RSK \circ S 
\end{equation}
(this is a part of \Cref{theorem:isomorphism-dynamical}).}

This follows from \cref{lem:factor-map}; for details see \cite[Section 2.4]{RomikSniady2011}.

\item 
\emph{\Jdt transformation $J$ is measure-preserving (this is a part of \Cref{theo:measure-preserving-and-ergodic}).} 

We need to show that if $T$ is a random infinite Young tableau with the distribution $\measure_{\alpha,\beta,\gamma}$ then $J(T)$ has the same distribution. Without loss of generality we may assume that
\begin{align*} T&= Q(\Letter_1,\Letter_2,\dots) \\
\intertext{is as prescribed by \cref{theo:KerovVershik-RSK-homomorphism}.
Equation \eqref{eq:factor-map} implies}
J(T)&= Q(\Letter_2,\Letter_3,\dots). 
\end{align*}
Another invocation of \cref{theo:KerovVershik-RSK-homomorphism} shows that the distribution of $J(T)$ is as required.

\item 
\emph{Map $\RSK^{-1}$ defined by \eqref{eq:rsk-inverse} is well defined almost everywhere (this is a part of \cref{theorem:inverse-rsk}).}

This follows from the facts that $J$ is measure-preserving and $\Psi$ is well-defined almost everywhere.

\item 
\emph{$\RSK^{-1} \circ \RSK = \Id$ almost everywhere, where $\RSK^{-1}$ is defined by \eqref{eq:rsk-inverse} (this is a part of \Cref{theorem:inverse-rsk}).}

Let $\Word=(\Letter_1,\Letter_2,\dots)\in\alphabetjeu^\N$ 
be an i.i.d.~sequence of letters with distribution $\alphabetmeasurejeu_{\alpha,\beta,\gamma}$ and let
$(U_1,U_2,\dots)=\RSK^{-1} \circ \RSK( \Word)$.
For any $k\geq 1$
\begin{equation}
\label{eq:why-inverse-ok}
 U_k = \Psi\big[ J^{k-1}\big( \RSK(\Letter_1,\Letter_2,\dots) \big) \big]  =
\Psi\big[  \RSK(\Letter_k,\Letter_{k+1},\dots) \big],  
\end{equation}
where the last equality follows from \eqref{eq:factor-map}.
We apply \cref{theo:determinism-of-jdt} in order to show that $U_k=\Letter_k$ almost surely; in the case when $\Letter_k\in\alphabetjeu_r$ or $\Letter_k\in\alphabetjeu_c$ this is straightforward, below we present a more detailed analysis of the case when $\Letter_k\in(0,1)=\alphabetjeu_0$. 

Concerning the right-hand side of \eqref{eq:why-inverse-ok}, the value of $\Theta$ (and thus the value of $\Psi$ as well) corresponding to \eqref{eq:theta-theta} depends only on $W_k$ and not on the element $(\alpha,\beta,\gamma)$ of Thoma simplex, as long as $\gamma>0$; in particular we can take $\alpha=\beta=(0,0,\dots)$, $\gamma=1$ which corresponds to Plancherel measure. The result in this case has been proved  in our previous work \cite[Eq.~(47)]{RomikSniady2011}.
% \textbf{WARNING!!!}
% % % WARNING. This reference may change!!!

% \todoPiotr{Maybe it would make sense if the first formula on page 47 in \cite{RomikSniady2011} was numbered, in this way we would be able to refer to it!}

\item 
\emph{$\RSK \circ \RSK^{-1} = \Id$ almost everywhere, where $\RSK^{-1}$ is defined by \eqref{eq:rsk-inverse} (this is a part of \Cref{theorem:inverse-rsk}).}

Let $T$ be a random infinite Young tableau with the distribution $\measure_{\alpha,\beta,\gamma}$. Without loss of generality we may assume that \linebreak
$T=\RSK(\Letter_1,\Letter_2,\dots)$, where $\Letter_1,\Letter_2,\dots$ is an i.i.d.~sequence of random letters with the distribution $\alphabetmeasurejeu_{\alpha,\beta,\gamma}$ ({\cref{theo:KerovVershik-RSK-homomorphism}}).
Then 
\begin{multline*}
 \RSK \circ \RSK^{-1} (T)= \RSK \circ \underbrace{\RSK^{-1} \circ \RSK}_{=\Id, \text{ as shown above}}(\Letter_1,\Letter_2,\dots)= \\ \RSK(\Letter_1,\Letter_2,\cdots) = T
\end{multline*}
holds true almost surely, as required.

\item
\emph{\Jdt transformation $J$ is ergodic (this is a part of \Cref{theo:measure-preserving-and-ergodic}).}

By \cref{theorem:isomorphism-dynamical}, $J$ is isomorphic to a Bernoulli shift $S$ which is clearly ergodic, see \cite{Silva2008}.
\end{itemize}
\end{proof}

\section*{Acknowledgments}

In the initial phase of research, Piotr \'Sniady was a holder of a fellowship of \emph{Alexander von Humboldt-Stiftung}.
Piotr \'Sniady's research has been supported by \emph{Deutsche Forschungsgemeinschaft} under grant SN 101/1-1.
I~thank the referees for their constructive criticism which helped improve the paper.

\appendix

\section{The ``counterexample'' of Fulman}

The work \cite{KerovVershik1986} of Kerov and Vershik has been criticized by Fulman \cite{Fulman2002}.
Since the current paper heavily uses the results of Kerov and Vershik, 
we feel obliged to respond to this criticism.

Fulman writes (all quotations are from \cite[p.~186--187]{Fulman2002}):
\begin{quote}
% \emph{However the case $\gamma\neq 0$ [\dots], is treated incorrectly in \cite{KerovVershik1986} [\dots].}
% \cite[p.~167]{Fulman2002}
% 
\emph{The paper \cite{KerovVershik1986} states a version of
Theorem 12 in which there is also a parameter $\gamma$ (their Proposition 3), but it is incorrect for
$\gamma\neq 0$ 
as the following counterexample shows. Setting all parameters other than $\alpha$ and
$\gamma = 1-\alpha$ equal to $0$, it follows from the definitions that the extended Schur function $\tilde{s}_2$
is equal to $\frac{\alpha^2+1}{2}$.}
\end{quote}
Indeed, this is the correct value of the extended Schur function.

\begin{quote}
\emph{But if Proposition 3 of \cite{KerovVershik1986} were correct, it would also equal $\alpha^2+(1-\alpha)\alpha=\alpha$ since the two words giving a Young tableau with $1$ row of length $2$ are $11$ and $01$.}
\end{quote}
With our notations, \cite[Proposition 3]{KerovVershik1986} states that the extended Schur function $\tilde{s}_2$ is equal to the probability that a random filling
\[
\begin{tikzpicture}
\draw (0,0) grid (2,1); 
\draw[very thick] (0,0) rectangle (2,1);
\draw (0.5,0.5) node {$u$};       
\draw (1.5,0.5) node {$v$};
\end{tikzpicture}
\]
of a Young diagram $(2)$ with letters $u,v\in\alphabet$ gives a (semistandard) tableau.
The probability distribution of the letters is assumed to have a unique atom (of weight $\alpha$) on some row letter $L\in\alphabet_r$. 
There are the following disjoint possibilities:
\begin{enumerate}[label=\emph{(\alph*)}]
\item \label[empty]{enum:counterexample-a} $u=v=L$;
\item \label[empty]{enum:counterexample-b} $u=L$, $v\neq L$, $u<v$;
\item \label[empty]{enum:counterexample-c} $u\neq L$, $v=L$, $u<v$;
\item \label[empty]{enum:counterexample-d} $u,v\neq L$, $u<v$.
\end{enumerate}
The event \ref{enum:counterexample-a} occurs with probability $\alpha^2$.
The union of the events \ref{enum:counterexample-b} and \ref{enum:counterexample-c} occurs with probability $\alpha (1-\alpha)$. The event \ref{enum:counterexample-d} occurs with probability $\frac{1}{2} (1-\alpha)^2$. The sum of these probabilities gives the correct value of the extended Schur function $\tilde{s}_2$.

It seems that in the calculation of Fulman the case \ref{enum:counterexample-d} is missing.
His explanation: \emph{``since the two words giving a Young tableau with $1$ row of length $2$ are $11$ and $01$''} probably stems from a collision in the notation used by Kerov and Vershik and the one used by Fulman.

\begin{quote}
\emph{In fact as
the $2$ in the denominator of $\frac{\alpha^2+1}{2}$
shows, one can't interpret the extended Schur functions with $\gamma 
\neq 0$ in terms of $\RSK$ and words on a finite number of symbols. }
\end{quote}

Indeed, in order to have $\gamma>0$ one should use an infinite alphabet 
and the non-atomic part of the probability distribution should be non-zero, just as
claimed by Kerov and Vershik.

\bibliographystyle{alpha}
\bibliography{biblio}

\end{document}

%% file: programy-jeu-de-taquin/grafika-infinite-tableau2.tex
\begin{tikzpicture}[scale=0.8]
\tikzfading[name=fade right,
      left color=transparent!100,
      right color=transparent!0]
\tikzfading[name=fade up,
      bottom color=transparent!100,
      top color=transparent!0]
\clip (-0.2,-0.2) rectangle (5.9,5.9);
\begin{scope}[]
\clip[](0,0) -- (60,0) -- (60,1) -- (41,1) -- (41,2) -- (32,2) -- (32,3) -- (2,3) -- (2,103) -- (1,103) -- (1,270) -- (0,270) -- cycle;
\draw[very thin,black] (0,0) grid (6,6);
\end{scope}
\draw[very thick](0,0) -- (60,0) -- (60,1) -- (41,1) -- (41,2) -- (32,2) -- (32,3) -- (2,3) -- (2,103) -- (1,103) -- (1,270) -- (0,270) -- cycle;
\draw (0.5,0.5) node {1};
\draw (1.5,0.5) node {2};
\draw (2.5,0.5) node {7};
\draw (3.5,0.5) node {24};
\draw (4.5,0.5) node {29};
\draw (5.5,0.5) node {40};
\draw (6.5,0.5) node {45};
\draw (0.5,1.5) node {3};
\draw (1.5,1.5) node {8};
\draw (2.5,1.5) node {13};
\draw (3.5,1.5) node {33};
\draw (4.5,1.5) node {52};
\draw (5.5,1.5) node {71};
\draw (6.5,1.5) node {79};
\draw (0.5,2.5) node {4};
\draw (1.5,2.5) node {9};
\draw (2.5,2.5) node {20};
\draw (3.5,2.5) node {62};
\draw (4.5,2.5) node {78};
\draw (5.5,2.5) node {99};
\draw (6.5,2.5) node {119};
\draw (0.5,3.5) node {5};
\draw (1.5,3.5) node {12};
\draw (0.5,4.5) node {6};
\draw (1.5,4.5) node {21};
\draw (0.5,5.5) node {10};
\draw (1.5,5.5) node {22};
\draw (0.5,6.5) node {11};
\draw (1.5,6.5) node {36};
\ifpdf
      \fill [white,path fading=fade right] (4,-1) rectangle (6.1,6.1);
      \fill [white,path fading=fade up] (-4,5) rectangle (6.1,6.1);
\fi
 
\end{tikzpicture}

%% file: programy-jeu-de-taquin/grafika-infinite-tableau.tex
\begin{tikzpicture}[scale=0.8]
\tikzfading[name=fade right,
      left color=transparent!100,
      right color=transparent!0]
\tikzfading[name=fade up,
      bottom color=transparent!100,
      top color=transparent!0]
\clip (-0.2,-0.2) rectangle (5.9,5.9);
\begin{scope}[]
\clip[](0,0) -- (30,0) -- (30,1) -- (25,1) -- (25,2) -- (19,2) -- (19,3) -- (18,3) -- (18,4) -- (17,4) -- (17,5) -- (16,5) -- (16,6) -- (15,6) -- (15,7) -- (14,7) -- (14,8) -- (11,8) -- (11,10) -- (10,10) -- (10,12) -- (9,12) -- (9,13) -- (8,13) -- (8,15) -- (7,15) -- (7,16) -- (6,16) -- (6,19) -- (5,19) -- (5,20) -- (4,20) -- (4,22) -- (3,22) -- (3,26) -- (2,26) -- (2,31) -- (1,31) -- (1,250) -- (0,250) -- cycle;
\draw[very thin,black] (0,0) grid (6,6);
\end{scope}
\draw[very thick](0,0) -- (30,0) -- (30,1) -- (25,1) -- (25,2) -- (19,2) -- (19,3) -- (18,3) -- (18,4) -- (17,4) -- (17,5) -- (16,5) -- (16,6) -- (15,6) -- (15,7) -- (14,7) -- (14,8) -- (11,8) -- (11,10) -- (10,10) -- (10,12) -- (9,12) -- (9,13) -- (8,13) -- (8,15) -- (7,15) -- (7,16) -- (6,16) -- (6,19) -- (5,19) -- (5,20) -- (4,20) -- (4,22) -- (3,22) -- (3,26) -- (2,26) -- (2,31) -- (1,31) -- (1,250) -- (0,250) -- cycle;
\draw (0.5,0.5) node {1};
\draw (1.5,0.5) node {2};
\draw (2.5,0.5) node {5};
\draw (3.5,0.5) node {10};
\draw (4.5,0.5) node {32};
\draw (5.5,0.5) node {34};
\draw (6.5,0.5) node {36};
\draw (0.5,1.5) node {3};
\draw (1.5,1.5) node {7};
\draw (2.5,1.5) node {13};
\draw (3.5,1.5) node {22};
\draw (4.5,1.5) node {38};
\draw (5.5,1.5) node {39};
\draw (6.5,1.5) node {41};
\draw (0.5,2.5) node {4};
\draw (1.5,2.5) node {12};
\draw (2.5,2.5) node {15};
\draw (3.5,2.5) node {37};
\draw (4.5,2.5) node {47};
\draw (5.5,2.5) node {52};
\draw (6.5,2.5) node {60};
\draw (0.5,3.5) node {6};
\draw (1.5,3.5) node {14};
\draw (2.5,3.5) node {20};
\draw (3.5,3.5) node {54};
\draw (4.5,3.5) node {56};
\draw (5.5,3.5) node {62};
\draw (6.5,3.5) node {77};
\draw (0.5,4.5) node {8};
\draw (1.5,4.5) node {21};
\draw (2.5,4.5) node {25};
\draw (3.5,4.5) node {79};
\draw (4.5,4.5) node {85};
\draw (5.5,4.5) node {93};
\draw (6.5,4.5) node {110};
\draw (0.5,5.5) node {9};
\draw (1.5,5.5) node {26};
\draw (2.5,5.5) node {33};
\draw (3.5,5.5) node {87};
\draw (4.5,5.5) node {89};
\draw (5.5,5.5) node {133};
\draw (6.5,5.5) node {138};
\draw (0.5,6.5) node {11};
\draw (1.5,6.5) node {27};
\draw (2.5,6.5) node {53};
\draw (3.5,6.5) node {97};
\draw (4.5,6.5) node {100};
\draw (5.5,6.5) node {142};
\draw (6.5,6.5) node {155};
\ifpdf
      \fill [white,path fading=fade right] (4,-1) rectangle (6.1,6.1);
      \fill [white,path fading=fade up] (-4,5) rectangle (6.1,6.1);
\fi
 
\end{tikzpicture}

%% file: programy-jeu-de-taquin/grafika-jdt-boxes1.tex
\begin{tikzpicture}[scale=0.7]
\small
 \clip (-0.2,-0.2) rectangle (6.9,6.9);
\fill[blue!20](5,7) -- (5,6) -- (5,5) -- (5,4) -- (5,3) -- (4,3) -- (4,2) -- (4,1) -- (3,1) -- (3,0) -- (2,0) -- (1,0) -- (0,0) -- (0,1) -- (1,1) -- (2,1) -- (2,2) -- (3,2) -- (3,3) -- (3,4) -- (4,4) -- (4,5) -- (4,6) -- (4,7) -- (4,8) -- cycle;
\begin{scope}[]
\clip[](0,0) -- (49,0) -- (49,1) -- (33,1) -- (33,2) -- (25,2) -- (25,3) -- (18,3) -- (18,4) -- (5,4) -- (5,23) -- (4,23) -- (4,30) -- (3,30) -- (3,32) -- (2,32) -- (2,38) -- (1,38) -- (1,73) -- (0,73) -- cycle;
\draw[very thin,black] (0,0) grid (7,7);
\end{scope}
\draw[very thick](0,0) -- (49,0) -- (49,1) -- (33,1) -- (33,2) -- (25,2) -- (25,3) -- (18,3) -- (18,4) -- (5,4) -- (5,23) -- (4,23) -- (4,30) -- (3,30) -- (3,32) -- (2,32) -- (2,38) -- (1,38) -- (1,73) -- (0,73) -- cycle;
% \draw[draw=blue,ultra thick](5,7) -- (5,6) -- (5,5) -- (5,4) -- (5,3) -- (4,3) -- (4,2) -- (4,1) -- (3,1) -- (3,0) -- (2,0) -- (1,0) -- (0,0) -- (0,1) -- (1,1) -- (2,1) -- (2,2) -- (3,2) -- (3,3) -- (3,4) -- (4,4) -- (4,5) -- (4,6) -- (4,7) -- (4,8) -- cycle;
\draw (0.5,0.5) node {1};
\draw (1.5,0.5) node {2};
\draw (2.5,0.5) node {3};
\draw (3.5,0.5) node {8};
\draw (4.5,0.5) node {16};
\draw (5.5,0.5) node {21};
\draw (6.5,0.5) node {28};
\draw (7.5,0.5) node {31};
\draw (0.5,1.5) node {4};
\draw (1.5,1.5) node {6};
\draw (2.5,1.5) node {7};
\draw (3.5,1.5) node {11};
\draw (4.5,1.5) node {26};
\draw (5.5,1.5) node {27};
\draw (6.5,1.5) node {41};
\draw (7.5,1.5) node {68};
\draw (0.5,2.5) node {5};
\draw (1.5,2.5) node {9};
\draw (2.5,2.5) node {14};
\draw (3.5,2.5) node {22};
\draw (4.5,2.5) node {30};
\draw (5.5,2.5) node {48};
\draw (6.5,2.5) node {58};
\draw (7.5,2.5) node {80};
\draw (0.5,3.5) node {10};
\draw (1.5,3.5) node {15};
\draw (2.5,3.5) node {17};
\draw (3.5,3.5) node {25};
\draw (4.5,3.5) node {42};
\draw (5.5,3.5) node {61};
\draw (6.5,3.5) node {119};
\draw (7.5,3.5) node {166};
\draw (0.5,4.5) node {12};
\draw (1.5,4.5) node {19};
\draw (2.5,4.5) node {20};
\draw (3.5,4.5) node {47};
\draw (4.5,4.5) node {57};
\draw (0.5,5.5) node {13};
\draw (1.5,5.5) node {24};
\draw (2.5,5.5) node {35};
\draw (3.5,5.5) node {56};
\draw (4.5,5.5) node {83};
\draw (0.5,6.5) node {18};
\draw (1.5,6.5) node {29};
\draw (2.5,6.5) node {43};
\draw (3.5,6.5) node {66};
\draw (4.5,6.5) node {85};
\draw (0.5,7.5) node {23};
\draw (1.5,7.5) node {34};
\draw (2.5,7.5) node {52};
\draw (3.5,7.5) node {72};
\draw (4.5,7.5) node {110};
\end{tikzpicture}

%% file: programy-jeu-de-taquin/grafika-jdt-boxes2.tex
\begin{tikzpicture}[scale=0.7]
\small
 \clip (-0.2,-0.2) rectangle (6.9,6.9);
\fill[blue!20](5,7) -- (5,6) -- (5,5) -- (5,4) -- (5,3) -- (4,3) -- (4,2) -- (4,1) -- (3,1) -- (3,0) -- (2,0) -- (1,0) -- (0,0) -- (0,1) -- (1,1) -- (2,1) -- (2,2) -- (3,2) -- (3,3) -- (3,4) -- (4,4) -- (4,5) -- (4,6) -- (4,7) -- (4,8) -- cycle;
\begin{scope}[]
\clip[](0,0) -- (49,0) -- (49,1) -- (33,1) -- (33,2) -- (25,2) -- (25,3) -- (18,3) -- (18,4) -- (5,4) -- (5,22) -- (4,22) -- (4,30) -- (3,30) -- (3,32) -- (2,32) -- (2,38) -- (1,38) -- (1,73) -- (0,73) -- cycle;
\draw[very thin,black] (0,0) grid (7,7);
\end{scope}
\draw[very thick](0,0) -- (49,0) -- (49,1) -- (33,1) -- (33,2) -- (25,2) -- (25,3) -- (18,3) -- (18,4) -- (5,4) -- (5,22) -- (4,22) -- (4,30) -- (3,30) -- (3,32) -- (2,32) -- (2,38) -- (1,38) -- (1,73) -- (0,73) -- cycle;
% \draw[draw=blue,ultra thick](5,7) -- (5,6) -- (5,5) -- (5,4) -- (5,3) -- (4,3) -- (4,2) -- (4,1) -- (3,1) -- (3,0) -- (2,0) -- (1,0) -- (0,0) -- (0,1) -- (1,1) -- (2,1) -- (2,2) -- (3,2) -- (3,3) -- (3,4) -- (4,4) -- (4,5) -- (4,6) -- (4,7) -- (4,8) -- cycle;
\draw (0.5,0.5) node {2};
\draw (1.5,0.5) node {3};
\draw (2.5,0.5) node {7};
\draw (3.5,0.5) node {8};
\draw (4.5,0.5) node {16};
\draw (5.5,0.5) node {21};
\draw (6.5,0.5) node {28};
\draw (7.5,0.5) node {31};
\draw (0.5,1.5) node {4};
\draw (1.5,1.5) node {6};
\draw (2.5,1.5) node {11};
\draw (3.5,1.5) node {22};
\draw (4.5,1.5) node {26};
\draw (5.5,1.5) node {27};
\draw (6.5,1.5) node {41};
\draw (7.5,1.5) node {68};
\draw (0.5,2.5) node {5};
\draw (1.5,2.5) node {9};
\draw (2.5,2.5) node {14};
\draw (3.5,2.5) node {25};
\draw (4.5,2.5) node {30};
\draw (5.5,2.5) node {48};
\draw (6.5,2.5) node {58};
\draw (7.5,2.5) node {80};
\draw (0.5,3.5) node {10};
\draw (1.5,3.5) node {15};
\draw (2.5,3.5) node {17};
\draw (3.5,3.5) node {42};
\draw (4.5,3.5) node {57};
\draw (5.5,3.5) node {61};
\draw (6.5,3.5) node {119};
\draw (7.5,3.5) node {166};
\draw (0.5,4.5) node {12};
\draw (1.5,4.5) node {19};
\draw (2.5,4.5) node {20};
\draw (3.5,4.5) node {47};
\draw (4.5,4.5) node {83};
\draw (0.5,5.5) node {13};
\draw (1.5,5.5) node {24};
\draw (2.5,5.5) node {35};
\draw (3.5,5.5) node {56};
\draw (4.5,5.5) node {85};
\draw (0.5,6.5) node {18};
\draw (1.5,6.5) node {29};
\draw (2.5,6.5) node {43};
\draw (3.5,6.5) node {66};
\draw (4.5,6.5) node {110};
\draw (0.5,7.5) node {23};
\draw (1.5,7.5) node {34};
\draw (2.5,7.5) node {52};
\draw (3.5,7.5) node {72};
\draw (4.5,7.5) node {145};
\end{tikzpicture} 

%% file: programy-jeu-de-taquin/grafika-jdt-boxes-deadend1.tex
\begin{tikzpicture}[scale=0.9]
\small
\clip (-0.2,-0.2) rectangle (6.8,6.8);
\fill[blue!20](4,4) -- (4,3) -- (4,2) -- (3,2) -- (3,1) -- (3,0) -- (2,0) -- (1,0) -- (0,0) -- (0,1) -- (1,1) -- (2,1) -- (2,2) -- (2,3) -- (3,3) -- (3,4) -- cycle;
\begin{scope}[]
\clip[](0,0) -- (41,0) -- (41,1) -- (36,1) -- (36,2) -- (28,2) -- (28,3) -- (4,3) -- (4,4) -- (2,4) -- (2,76) -- (1,76) -- (1,125) -- (0,125) -- cycle;
\draw[very thin,black] (0,0) grid (7,7);
\end{scope}
\draw[very thick](0,0) -- (41,0) -- (41,1) -- (36,1) -- (36,2) -- (28,2) -- (28,3) -- (4,3) -- (4,4) -- (2,4) -- (2,76) -- (1,76) -- (1,125) -- (0,125) -- cycle;
% \draw[draw=blue,ultra thick](4,4) -- (4,3) -- (4,2) -- (3,2) -- (3,1) -- (3,0) -- (2,0) -- (1,0) -- (0,0) -- (0,1) -- (1,1) -- (2,1) -- (2,2) -- (2,3) -- (3,3) -- (3,4) -- cycle;
\draw (0.5,0.5) node {1};
\draw (1.5,0.5) node {2};
\draw (2.5,0.5) node {3};
\draw (3.5,0.5) node {16};
\draw (4.5,0.5) node {19};
\draw (5.5,0.5) node {23};
\draw (6.5,0.5) node {27};
\draw (7.5,0.5) node {28};
\draw (0.5,1.5) node {4};
\draw (1.5,1.5) node {7};
\draw (2.5,1.5) node {8};
\draw (3.5,1.5) node {17};
\draw (4.5,1.5) node {20};
\draw (5.5,1.5) node {24};
\draw (6.5,1.5) node {30};
\draw (7.5,1.5) node {45};
\draw (0.5,2.5) node {5};
\draw (1.5,2.5) node {9};
\draw (2.5,2.5) node {13};
\draw (3.5,2.5) node {21};
\draw (4.5,2.5) node {60};
\draw (5.5,2.5) node {61};
\draw (6.5,2.5) node {89};
\draw (7.5,2.5) node {91};
\draw (0.5,3.5) node {6};
\draw (1.5,3.5) node {32};
\draw (2.5,3.5) node {41};
\draw (3.5,3.5) node {50};
\draw (0.5,4.5) node {10};
\draw (1.5,4.5) node {42};
\draw (0.5,5.5) node {11};
\draw (1.5,5.5) node {47};
\draw (0.5,6.5) node {12};
\draw (1.5,6.5) node {53};
\draw (0.5,7.5) node {14};
\draw (1.5,7.5) node {65};
 \end{tikzpicture}

%% file: programy-jeu-de-taquin/grafika-jdt.tex
\begin{tikzpicture}[scale=0.2]
\clip (-1,-1) rectangle (49.8,49.8); 
 \draw[black!10] (0,0) grid (50,50); 
\draw[thick](0,0) -- (50,0);
\draw[thick](0,0) -- (0,50);
\definecolor{currentcolor}{cmyk}{0.0,0.0,0.0,0.6}
\begin{scope}[currentcolor,fill=currentcolor,fill opacity=0.2]
\filldraw[] (0,0) rectangle +(1,1); 
\filldraw[] (0,1) rectangle +(1,1); 
\filldraw[] (1,1) rectangle +(1,1); 
\filldraw[] (2,1) rectangle +(1,1); 
\filldraw[] (3,1) rectangle +(1,1); 
\filldraw[] (4,1) rectangle +(1,1); 
\filldraw[] (5,1) rectangle +(1,1); 
\filldraw[] (6,1) rectangle +(1,1); 
\filldraw[] (7,1) rectangle +(1,1); 
\filldraw[] (8,1) rectangle +(1,1); 
\filldraw[] (9,1) rectangle +(1,1); 
\filldraw[] (10,1) rectangle +(1,1); 
\filldraw[] (11,1) rectangle +(1,1); 
\filldraw[] (12,1) rectangle +(1,1); 
\filldraw[] (13,1) rectangle +(1,1); 
\filldraw[] (14,1) rectangle +(1,1); 
\filldraw[] (15,1) rectangle +(1,1); 
\filldraw[] (16,1) rectangle +(1,1); 
\filldraw[] (17,1) rectangle +(1,1); 
\filldraw[] (18,1) rectangle +(1,1); 
\filldraw[] (19,1) rectangle +(1,1); 
\filldraw[] (20,1) rectangle +(1,1); 
\filldraw[] (21,1) rectangle +(1,1); 
\filldraw[] (22,1) rectangle +(1,1); 
\filldraw[] (23,1) rectangle +(1,1); 
\filldraw[] (24,1) rectangle +(1,1); 
\filldraw[] (24,2) rectangle +(1,1); 
\filldraw[] (25,2) rectangle +(1,1); 
\filldraw[] (26,2) rectangle +(1,1); 
\filldraw[] (27,2) rectangle +(1,1); 
\filldraw[] (28,2) rectangle +(1,1); 
\filldraw[] (29,2) rectangle +(1,1); 
\filldraw[] (30,2) rectangle +(1,1); 
\filldraw[] (31,2) rectangle +(1,1); 
\filldraw[] (32,2) rectangle +(1,1); 
\filldraw[] (33,2) rectangle +(1,1); 
\filldraw[] (34,2) rectangle +(1,1); 
\filldraw[] (35,2) rectangle +(1,1); 
\filldraw[] (36,2) rectangle +(1,1); 
\filldraw[] (37,2) rectangle +(1,1); 
\filldraw[] (38,2) rectangle +(1,1); 
\filldraw[] (39,2) rectangle +(1,1); 
\filldraw[] (40,2) rectangle +(1,1); 
\filldraw[] (41,2) rectangle +(1,1); 
\filldraw[] (42,2) rectangle +(1,1); 
\filldraw[] (43,2) rectangle +(1,1); 
\filldraw[] (44,2) rectangle +(1,1); 
\filldraw[] (45,2) rectangle +(1,1); 
\filldraw[] (46,2) rectangle +(1,1); 
\filldraw[] (47,2) rectangle +(1,1); 
\filldraw[] (48,2) rectangle +(1,1); 
\filldraw[] (49,2) rectangle +(1,1); 
\filldraw[] (50,2) rectangle +(1,1); 
\end{scope}
\draw[currentcolor,thick,dashed](0,2.5) -- (50,2.5);
\definecolor{currentcolor}{cmyk}{0.5,0.0,0.5,0.0}
\begin{scope}[currentcolor,fill=currentcolor,fill opacity=0.2]
\filldraw[] (0,0) rectangle +(1,1); 
\filldraw[] (1,0) rectangle +(1,1); 
\filldraw[] (2,0) rectangle +(1,1); 
\filldraw[] (3,0) rectangle +(1,1); 
\filldraw[] (4,0) rectangle +(1,1); 
\filldraw[] (5,0) rectangle +(1,1); 
\filldraw[] (6,0) rectangle +(1,1); 
\filldraw[] (7,0) rectangle +(1,1); 
\filldraw[] (8,0) rectangle +(1,1); 
\filldraw[] (9,0) rectangle +(1,1); 
\filldraw[] (10,0) rectangle +(1,1); 
\filldraw[] (11,0) rectangle +(1,1); 
\filldraw[] (12,0) rectangle +(1,1); 
\filldraw[] (13,0) rectangle +(1,1); 
\filldraw[] (14,0) rectangle +(1,1); 
\filldraw[] (15,0) rectangle +(1,1); 
\filldraw[] (16,0) rectangle +(1,1); 
\filldraw[] (17,0) rectangle +(1,1); 
\filldraw[] (18,0) rectangle +(1,1); 
\filldraw[] (19,0) rectangle +(1,1); 
\filldraw[] (20,0) rectangle +(1,1); 
\filldraw[] (21,0) rectangle +(1,1); 
\filldraw[] (22,0) rectangle +(1,1); 
\filldraw[] (23,0) rectangle +(1,1); 
\filldraw[] (24,0) rectangle +(1,1); 
\filldraw[] (25,0) rectangle +(1,1); 
\filldraw[] (26,0) rectangle +(1,1); 
\filldraw[] (27,0) rectangle +(1,1); 
\filldraw[] (28,0) rectangle +(1,1); 
\filldraw[] (29,0) rectangle +(1,1); 
\filldraw[] (30,0) rectangle +(1,1); 
\filldraw[] (31,0) rectangle +(1,1); 
\filldraw[] (32,0) rectangle +(1,1); 
\filldraw[] (33,0) rectangle +(1,1); 
\filldraw[] (34,0) rectangle +(1,1); 
\filldraw[] (35,0) rectangle +(1,1); 
\filldraw[] (36,0) rectangle +(1,1); 
\filldraw[] (37,0) rectangle +(1,1); 
\filldraw[] (38,0) rectangle +(1,1); 
\filldraw[] (39,0) rectangle +(1,1); 
\filldraw[] (40,0) rectangle +(1,1); 
\filldraw[] (41,0) rectangle +(1,1); 
\filldraw[] (42,0) rectangle +(1,1); 
\filldraw[] (43,0) rectangle +(1,1); 
\filldraw[] (44,0) rectangle +(1,1); 
\filldraw[] (45,0) rectangle +(1,1); 
\filldraw[] (46,0) rectangle +(1,1); 
\filldraw[] (47,0) rectangle +(1,1); 
\filldraw[] (48,0) rectangle +(1,1); 
\filldraw[] (49,0) rectangle +(1,1); 
\filldraw[] (50,0) rectangle +(1,1); 
\end{scope}
\draw[currentcolor,thick,dashed](0,0.5) -- (50,0.5);
\definecolor{currentcolor}{cmyk}{0.0,1.0,0.0,0.0}
\begin{scope}[currentcolor,fill=currentcolor,fill opacity=0.2]
\filldraw[] (0,0) rectangle +(1,1); 
\filldraw[] (0,1) rectangle +(1,1); 
\filldraw[] (0,2) rectangle +(1,1); 
\filldraw[] (0,3) rectangle +(1,1); 
\filldraw[] (0,4) rectangle +(1,1); 
\filldraw[] (0,5) rectangle +(1,1); 
\filldraw[] (0,6) rectangle +(1,1); 
\filldraw[] (0,7) rectangle +(1,1); 
\filldraw[] (0,8) rectangle +(1,1); 
\filldraw[] (0,9) rectangle +(1,1); 
\filldraw[] (0,10) rectangle +(1,1); 
\filldraw[] (0,11) rectangle +(1,1); 
\filldraw[] (0,12) rectangle +(1,1); 
\filldraw[] (0,13) rectangle +(1,1); 
\filldraw[] (1,13) rectangle +(1,1); 
\filldraw[] (1,14) rectangle +(1,1); 
\filldraw[] (1,15) rectangle +(1,1); 
\filldraw[] (1,16) rectangle +(1,1); 
\filldraw[] (1,17) rectangle +(1,1); 
\filldraw[] (1,18) rectangle +(1,1); 
\filldraw[] (1,19) rectangle +(1,1); 
\filldraw[] (1,20) rectangle +(1,1); 
\filldraw[] (1,21) rectangle +(1,1); 
\filldraw[] (1,22) rectangle +(1,1); 
\filldraw[] (1,23) rectangle +(1,1); 
\filldraw[] (1,24) rectangle +(1,1); 
\filldraw[] (1,25) rectangle +(1,1); 
\filldraw[] (1,26) rectangle +(1,1); 
\filldraw[] (1,27) rectangle +(1,1); 
\filldraw[] (1,28) rectangle +(1,1); 
\filldraw[] (1,29) rectangle +(1,1); 
\filldraw[] (1,30) rectangle +(1,1); 
\filldraw[] (1,31) rectangle +(1,1); 
\filldraw[] (1,32) rectangle +(1,1); 
\filldraw[] (1,33) rectangle +(1,1); 
\filldraw[] (1,34) rectangle +(1,1); 
\filldraw[] (1,35) rectangle +(1,1); 
\filldraw[] (1,36) rectangle +(1,1); 
\filldraw[] (1,37) rectangle +(1,1); 
\filldraw[] (1,38) rectangle +(1,1); 
\filldraw[] (1,39) rectangle +(1,1); 
\filldraw[] (1,40) rectangle +(1,1); 
\filldraw[] (1,41) rectangle +(1,1); 
\filldraw[] (1,42) rectangle +(1,1); 
\filldraw[] (1,43) rectangle +(1,1); 
\filldraw[] (1,44) rectangle +(1,1); 
\filldraw[] (1,45) rectangle +(1,1); 
\filldraw[] (1,46) rectangle +(1,1); 
\filldraw[] (1,47) rectangle +(1,1); 
\filldraw[] (1,48) rectangle +(1,1); 
\filldraw[] (1,49) rectangle +(1,1); 
\filldraw[] (1,50) rectangle +(1,1); 
\end{scope}
\draw[currentcolor,thick,dashed](1.5,0) -- (1.5,50);
\definecolor{currentcolor}{cmyk}{0.25,0.25,0.0,0.0}
\begin{scope}[currentcolor,fill=currentcolor,fill opacity=0.2]
\filldraw[] (0,0) rectangle +(1,1); 
\filldraw[] (1,0) rectangle +(1,1); 
\filldraw[] (1,1) rectangle +(1,1); 
\filldraw[] (1,2) rectangle +(1,1); 
\filldraw[] (2,2) rectangle +(1,1); 
\filldraw[] (3,2) rectangle +(1,1); 
\filldraw[] (3,3) rectangle +(1,1); 
\filldraw[] (4,3) rectangle +(1,1); 
\filldraw[] (4,4) rectangle +(1,1); 
\filldraw[] (5,4) rectangle +(1,1); 
\filldraw[] (5,5) rectangle +(1,1); 
\filldraw[] (5,6) rectangle +(1,1); 
\filldraw[] (5,7) rectangle +(1,1); 
\filldraw[] (5,8) rectangle +(1,1); 
\filldraw[] (5,9) rectangle +(1,1); 
\filldraw[] (5,10) rectangle +(1,1); 
\filldraw[] (5,11) rectangle +(1,1); 
\filldraw[] (6,11) rectangle +(1,1); 
\filldraw[] (6,12) rectangle +(1,1); 
\filldraw[] (6,13) rectangle +(1,1); 
\filldraw[] (7,13) rectangle +(1,1); 
\filldraw[] (8,13) rectangle +(1,1); 
\filldraw[] (8,14) rectangle +(1,1); 
\filldraw[] (8,15) rectangle +(1,1); 
\filldraw[] (8,16) rectangle +(1,1); 
\filldraw[] (8,17) rectangle +(1,1); 
\filldraw[] (8,18) rectangle +(1,1); 
\filldraw[] (8,19) rectangle +(1,1); 
\filldraw[] (8,20) rectangle +(1,1); 
\filldraw[] (9,20) rectangle +(1,1); 
\filldraw[] (9,21) rectangle +(1,1); 
\filldraw[] (10,21) rectangle +(1,1); 
\filldraw[] (10,22) rectangle +(1,1); 
\filldraw[] (10,23) rectangle +(1,1); 
\filldraw[] (10,24) rectangle +(1,1); 
\filldraw[] (10,25) rectangle +(1,1); 
\filldraw[] (11,25) rectangle +(1,1); 
\filldraw[] (12,25) rectangle +(1,1); 
\filldraw[] (13,25) rectangle +(1,1); 
\filldraw[] (14,25) rectangle +(1,1); 
\filldraw[] (15,25) rectangle +(1,1); 
\filldraw[] (15,26) rectangle +(1,1); 
\filldraw[] (15,27) rectangle +(1,1); 
\filldraw[] (15,28) rectangle +(1,1); 
\filldraw[] (16,28) rectangle +(1,1); 
\filldraw[] (16,29) rectangle +(1,1); 
\filldraw[] (16,30) rectangle +(1,1); 
\filldraw[] (16,31) rectangle +(1,1); 
\filldraw[] (17,31) rectangle +(1,1); 
\filldraw[] (18,31) rectangle +(1,1); 
\filldraw[] (19,31) rectangle +(1,1); 
\filldraw[] (19,32) rectangle +(1,1); 
\filldraw[] (20,32) rectangle +(1,1); 
\filldraw[] (20,33) rectangle +(1,1); 
\filldraw[] (20,34) rectangle +(1,1); 
\filldraw[] (21,34) rectangle +(1,1); 
\filldraw[] (21,35) rectangle +(1,1); 
\filldraw[] (22,35) rectangle +(1,1); 
\filldraw[] (23,35) rectangle +(1,1); 
\filldraw[] (23,36) rectangle +(1,1); 
\filldraw[] (23,37) rectangle +(1,1); 
\filldraw[] (23,38) rectangle +(1,1); 
\filldraw[] (24,38) rectangle +(1,1); 
\filldraw[] (25,38) rectangle +(1,1); 
\filldraw[] (25,39) rectangle +(1,1); 
\filldraw[] (25,40) rectangle +(1,1); 
\filldraw[] (25,41) rectangle +(1,1); 
\filldraw[] (25,42) rectangle +(1,1); 
\filldraw[] (25,43) rectangle +(1,1); 
\filldraw[] (25,44) rectangle +(1,1); 
\filldraw[] (25,45) rectangle +(1,1); 
\filldraw[] (25,46) rectangle +(1,1); 
\filldraw[] (25,47) rectangle +(1,1); 
\filldraw[] (25,48) rectangle +(1,1); 
\filldraw[] (25,49) rectangle +(1,1); 
\filldraw[] (25,50) rectangle +(1,1); 
\end{scope}
\draw[currentcolor,thick,dashed](0,0) -- (53.9805531805272,73.9805531805272);
\definecolor{currentcolor}{cmyk}{0.8,0.4,0.0,0.0}
\begin{scope}[currentcolor,fill=currentcolor,fill opacity=0.2]
\filldraw[] (0,0) rectangle +(1,1); 
\filldraw[] (0,1) rectangle +(1,1); 
\filldraw[] (0,2) rectangle +(1,1); 
\filldraw[] (0,3) rectangle +(1,1); 
\filldraw[] (0,4) rectangle +(1,1); 
\filldraw[] (0,5) rectangle +(1,1); 
\filldraw[] (1,5) rectangle +(1,1); 
\filldraw[] (2,5) rectangle +(1,1); 
\filldraw[] (3,5) rectangle +(1,1); 
\filldraw[] (4,5) rectangle +(1,1); 
\filldraw[] (4,6) rectangle +(1,1); 
\filldraw[] (4,7) rectangle +(1,1); 
\filldraw[] (4,8) rectangle +(1,1); 
\filldraw[] (4,9) rectangle +(1,1); 
\filldraw[] (4,10) rectangle +(1,1); 
\filldraw[] (4,11) rectangle +(1,1); 
\filldraw[] (4,12) rectangle +(1,1); 
\filldraw[] (4,13) rectangle +(1,1); 
\filldraw[] (4,14) rectangle +(1,1); 
\filldraw[] (4,15) rectangle +(1,1); 
\filldraw[] (4,16) rectangle +(1,1); 
\filldraw[] (5,16) rectangle +(1,1); 
\filldraw[] (5,17) rectangle +(1,1); 
\filldraw[] (5,18) rectangle +(1,1); 
\filldraw[] (5,19) rectangle +(1,1); 
\filldraw[] (5,20) rectangle +(1,1); 
\filldraw[] (5,21) rectangle +(1,1); 
\filldraw[] (6,21) rectangle +(1,1); 
\filldraw[] (6,22) rectangle +(1,1); 
\filldraw[] (6,23) rectangle +(1,1); 
\filldraw[] (6,24) rectangle +(1,1); 
\filldraw[] (6,25) rectangle +(1,1); 
\filldraw[] (6,26) rectangle +(1,1); 
\filldraw[] (6,27) rectangle +(1,1); 
\filldraw[] (6,28) rectangle +(1,1); 
\filldraw[] (7,28) rectangle +(1,1); 
\filldraw[] (7,29) rectangle +(1,1); 
\filldraw[] (7,30) rectangle +(1,1); 
\filldraw[] (8,30) rectangle +(1,1); 
\filldraw[] (9,30) rectangle +(1,1); 
\filldraw[] (10,30) rectangle +(1,1); 
\filldraw[] (10,31) rectangle +(1,1); 
\filldraw[] (11,31) rectangle +(1,1); 
\filldraw[] (11,32) rectangle +(1,1); 
\filldraw[] (11,33) rectangle +(1,1); 
\filldraw[] (11,34) rectangle +(1,1); 
\filldraw[] (11,35) rectangle +(1,1); 
\filldraw[] (11,36) rectangle +(1,1); 
\filldraw[] (11,37) rectangle +(1,1); 
\filldraw[] (11,38) rectangle +(1,1); 
\filldraw[] (11,39) rectangle +(1,1); 
\filldraw[] (12,39) rectangle +(1,1); 
\filldraw[] (13,39) rectangle +(1,1); 
\filldraw[] (14,39) rectangle +(1,1); 
\filldraw[] (14,40) rectangle +(1,1); 
\filldraw[] (15,40) rectangle +(1,1); 
\filldraw[] (15,41) rectangle +(1,1); 
\filldraw[] (16,41) rectangle +(1,1); 
\filldraw[] (16,42) rectangle +(1,1); 
\filldraw[] (16,43) rectangle +(1,1); 
\filldraw[] (16,44) rectangle +(1,1); 
\filldraw[] (16,45) rectangle +(1,1); 
\filldraw[] (16,46) rectangle +(1,1); 
\filldraw[] (17,46) rectangle +(1,1); 
\filldraw[] (17,47) rectangle +(1,1); 
\filldraw[] (18,47) rectangle +(1,1); 
\filldraw[] (18,48) rectangle +(1,1); 
\filldraw[] (19,48) rectangle +(1,1); 
\filldraw[] (19,49) rectangle +(1,1); 
\filldraw[] (19,50) rectangle +(1,1); 
\filldraw[] (20,50) rectangle +(1,1); 
\end{scope}
\draw[currentcolor,thick,dashed](0,0) -- (28.8263563503563,108.826356350356);
\definecolor{currentcolor}{cmyk}{0.0,0.0,0.0,1.0}
\begin{scope}[currentcolor,fill=currentcolor,fill opacity=0.2]
\filldraw[] (0,0) rectangle +(1,1); 
\filldraw[] (0,1) rectangle +(1,1); 
\filldraw[] (1,1) rectangle +(1,1); 
\filldraw[] (2,1) rectangle +(1,1); 
\filldraw[] (3,1) rectangle +(1,1); 
\filldraw[] (4,1) rectangle +(1,1); 
\filldraw[] (5,1) rectangle +(1,1); 
\filldraw[] (6,1) rectangle +(1,1); 
\filldraw[] (7,1) rectangle +(1,1); 
\filldraw[] (8,1) rectangle +(1,1); 
\filldraw[] (8,2) rectangle +(1,1); 
\filldraw[] (9,2) rectangle +(1,1); 
\filldraw[] (10,2) rectangle +(1,1); 
\filldraw[] (10,3) rectangle +(1,1); 
\filldraw[] (11,3) rectangle +(1,1); 
\filldraw[] (12,3) rectangle +(1,1); 
\filldraw[] (12,4) rectangle +(1,1); 
\filldraw[] (12,5) rectangle +(1,1); 
\filldraw[] (12,6) rectangle +(1,1); 
\filldraw[] (13,6) rectangle +(1,1); 
\filldraw[] (14,6) rectangle +(1,1); 
\filldraw[] (14,7) rectangle +(1,1); 
\filldraw[] (15,7) rectangle +(1,1); 
\filldraw[] (16,7) rectangle +(1,1); 
\filldraw[] (16,8) rectangle +(1,1); 
\filldraw[] (16,9) rectangle +(1,1); 
\filldraw[] (16,10) rectangle +(1,1); 
\filldraw[] (16,11) rectangle +(1,1); 
\filldraw[] (17,11) rectangle +(1,1); 
\filldraw[] (17,12) rectangle +(1,1); 
\filldraw[] (18,12) rectangle +(1,1); 
\filldraw[] (18,13) rectangle +(1,1); 
\filldraw[] (19,13) rectangle +(1,1); 
\filldraw[] (20,13) rectangle +(1,1); 
\filldraw[] (20,14) rectangle +(1,1); 
\filldraw[] (20,15) rectangle +(1,1); 
\filldraw[] (21,15) rectangle +(1,1); 
\filldraw[] (22,15) rectangle +(1,1); 
\filldraw[] (23,15) rectangle +(1,1); 
\filldraw[] (24,15) rectangle +(1,1); 
\filldraw[] (25,15) rectangle +(1,1); 
\filldraw[] (25,16) rectangle +(1,1); 
\filldraw[] (25,17) rectangle +(1,1); 
\filldraw[] (25,18) rectangle +(1,1); 
\filldraw[] (25,19) rectangle +(1,1); 
\filldraw[] (26,19) rectangle +(1,1); 
\filldraw[] (27,19) rectangle +(1,1); 
\filldraw[] (28,19) rectangle +(1,1); 
\filldraw[] (29,19) rectangle +(1,1); 
\filldraw[] (30,19) rectangle +(1,1); 
\filldraw[] (31,19) rectangle +(1,1); 
\filldraw[] (32,19) rectangle +(1,1); 
\filldraw[] (33,19) rectangle +(1,1); 
\filldraw[] (34,19) rectangle +(1,1); 
\filldraw[] (35,19) rectangle +(1,1); 
\filldraw[] (36,19) rectangle +(1,1); 
\filldraw[] (37,19) rectangle +(1,1); 
\filldraw[] (38,19) rectangle +(1,1); 
\filldraw[] (39,19) rectangle +(1,1); 
\filldraw[] (40,19) rectangle +(1,1); 
\filldraw[] (41,19) rectangle +(1,1); 
\filldraw[] (41,20) rectangle +(1,1); 
\filldraw[] (42,20) rectangle +(1,1); 
\filldraw[] (43,20) rectangle +(1,1); 
\filldraw[] (44,20) rectangle +(1,1); 
\filldraw[] (45,20) rectangle +(1,1); 
\filldraw[] (46,20) rectangle +(1,1); 
\filldraw[] (47,20) rectangle +(1,1); 
\filldraw[] (47,21) rectangle +(1,1); 
\filldraw[] (47,22) rectangle +(1,1); 
\filldraw[] (48,22) rectangle +(1,1); 
\filldraw[] (49,22) rectangle +(1,1); 
\filldraw[] (50,22) rectangle +(1,1); 
\end{scope}
\draw[currentcolor,thick,dashed](0,0) -- (79.3795265070547,49.3795265070547);
\definecolor{currentcolor}{cmyk}{0.25,0.6,0.9,0.0}
\begin{scope}[currentcolor,fill=currentcolor,fill opacity=0.2]
\filldraw[] (0,0) rectangle +(1,1); 
\filldraw[] (1,0) rectangle +(1,1); 
\filldraw[] (2,0) rectangle +(1,1); 
\filldraw[] (3,0) rectangle +(1,1); 
\filldraw[] (4,0) rectangle +(1,1); 
\filldraw[] (5,0) rectangle +(1,1); 
\filldraw[] (6,0) rectangle +(1,1); 
\filldraw[] (7,0) rectangle +(1,1); 
\filldraw[] (8,0) rectangle +(1,1); 
\filldraw[] (9,0) rectangle +(1,1); 
\filldraw[] (10,0) rectangle +(1,1); 
\filldraw[] (11,0) rectangle +(1,1); 
\filldraw[] (12,0) rectangle +(1,1); 
\filldraw[] (13,0) rectangle +(1,1); 
\filldraw[] (14,0) rectangle +(1,1); 
\filldraw[] (14,1) rectangle +(1,1); 
\filldraw[] (15,1) rectangle +(1,1); 
\filldraw[] (15,2) rectangle +(1,1); 
\filldraw[] (15,3) rectangle +(1,1); 
\filldraw[] (15,4) rectangle +(1,1); 
\filldraw[] (15,5) rectangle +(1,1); 
\filldraw[] (15,6) rectangle +(1,1); 
\filldraw[] (16,6) rectangle +(1,1); 
\filldraw[] (17,6) rectangle +(1,1); 
\filldraw[] (18,6) rectangle +(1,1); 
\filldraw[] (19,6) rectangle +(1,1); 
\filldraw[] (20,6) rectangle +(1,1); 
\filldraw[] (21,6) rectangle +(1,1); 
\filldraw[] (22,6) rectangle +(1,1); 
\filldraw[] (23,6) rectangle +(1,1); 
\filldraw[] (24,6) rectangle +(1,1); 
\filldraw[] (25,6) rectangle +(1,1); 
\filldraw[] (25,7) rectangle +(1,1); 
\filldraw[] (26,7) rectangle +(1,1); 
\filldraw[] (27,7) rectangle +(1,1); 
\filldraw[] (27,8) rectangle +(1,1); 
\filldraw[] (28,8) rectangle +(1,1); 
\filldraw[] (28,9) rectangle +(1,1); 
\filldraw[] (28,10) rectangle +(1,1); 
\filldraw[] (29,10) rectangle +(1,1); 
\filldraw[] (30,10) rectangle +(1,1); 
\filldraw[] (31,10) rectangle +(1,1); 
\filldraw[] (32,10) rectangle +(1,1); 
\filldraw[] (33,10) rectangle +(1,1); 
\filldraw[] (34,10) rectangle +(1,1); 
\filldraw[] (35,10) rectangle +(1,1); 
\filldraw[] (36,10) rectangle +(1,1); 
\filldraw[] (37,10) rectangle +(1,1); 
\filldraw[] (38,10) rectangle +(1,1); 
\filldraw[] (39,10) rectangle +(1,1); 
\filldraw[] (40,10) rectangle +(1,1); 
\filldraw[] (41,10) rectangle +(1,1); 
\filldraw[] (42,10) rectangle +(1,1); 
\filldraw[] (43,10) rectangle +(1,1); 
\filldraw[] (44,10) rectangle +(1,1); 
\filldraw[] (45,10) rectangle +(1,1); 
\filldraw[] (46,10) rectangle +(1,1); 
\filldraw[] (47,10) rectangle +(1,1); 
\filldraw[] (48,10) rectangle +(1,1); 
\filldraw[] (49,10) rectangle +(1,1); 
\filldraw[] (50,10) rectangle +(1,1); 
\end{scope}
\draw[currentcolor,thick,dashed](0,0) -- (121.799556208846,21.7995562088459);
\draw[thick] (50,0) -- (0,0) -- (0,50);
 
\end{tikzpicture}

%% file: programy-jeu-de-taquin/grafika-insertion-tableau2a.tex
\begin{tikzpicture}[scale=0.9]
\begin{scope} \clip (5,0) rectangle +(1,1); \draw[line width=6pt,green] (5,0) rectangle +(1,1); \end{scope}
\begin{scope} \clip (3,1) rectangle +(1,1); \draw[line width=6pt,green] (3,1) rectangle +(1,1); \end{scope}
\begin{scope} \clip (2,2) rectangle +(1,1); \draw[line width=6pt,green] (2,2) rectangle +(1,1); \end{scope}
\begin{scope} \clip (2,3) rectangle +(1,1); \draw[line width=6pt,green] (2,3) rectangle +(1,1); \end{scope}
\begin{scope} \clip (1,4) rectangle +(1,1); \draw[line width=6pt,green] (1,4) rectangle +(1,1); \end{scope}
\begin{scope} \clip (1,5) rectangle +(1,1); \draw[line width=6pt,green] (1,5) rectangle +(1,1); \end{scope}
\fill[pattern=horizontal lines A,pattern color=blue] (0,0) rectangle +(1,1); 
\draw (0.5,0.5) node[circle,inner sep=0pt,fill=white] {1};
\fill[pattern=horizontal lines A,pattern color=blue] (1,0) rectangle +(1,1); 
\draw (1.5,0.5) node[circle,inner sep=0pt,fill=white] {1};
\fill[pattern=horizontal lines A,pattern color=blue] (2,0) rectangle +(1,1); 
\draw (2.5,0.5) node[circle,inner sep=0pt,fill=white] {1};
\fill[pattern=horizontal lines B,pattern color=blue] (3,0) rectangle +(1,1); 
\draw (3.5,0.5) node[circle,inner sep=0pt,fill=white] {2};
\fill[pattern=horizontal lines B,pattern color=blue] (4,0) rectangle +(1,1); 
\draw (4.5,0.5) node[circle,inner sep=0pt,fill=white] {2};
\fill[pattern=horizontal lines C,pattern color=blue] (5,0) rectangle +(1,1); 
\draw (5.5,0.5) node[circle,inner sep=0pt,fill=white] {3};
\draw (6.5,0.5) node[circle,inner sep=0pt,fill=white] {0.1};
\draw (7.5,0.5) node[circle,inner sep=0pt,fill=white] {0.8};
\fill[pattern=horizontal lines B,pattern color=blue] (0,1) rectangle +(1,1); 
\draw (0.5,1.5) node[circle,inner sep=0pt,fill=white] {2};
\fill[pattern=horizontal lines C,pattern color=blue] (1,1) rectangle +(1,1); 
\draw (1.5,1.5) node[circle,inner sep=0pt,fill=white] {3};
\fill[pattern=horizontal lines C,pattern color=blue] (2,1) rectangle +(1,1); 
\draw (2.5,1.5) node[circle,inner sep=0pt,fill=white] {3};
\draw (3.5,1.5) node[circle,inner sep=0pt,fill=white] {0.5};
\draw (4.5,1.5) node[circle,inner sep=0pt,fill=white] {0.6};
\fill[pattern=vertical lines C,pattern color=red] (5,1) rectangle +(1,1); 
\draw (5.5,1.5) node[circle,inner sep=0pt,fill=white] {-3};
\fill[pattern=vertical lines B,pattern color=red] (6,1) rectangle +(1,1); 
\draw (6.5,1.5) node[circle,inner sep=0pt,fill=white] {-2};
\fill[pattern=vertical lines A,pattern color=red] (7,1) rectangle +(1,1); 
\draw (7.5,1.5) node[circle,inner sep=0pt,fill=white] {-1};
\fill[pattern=horizontal lines C,pattern color=blue] (0,2) rectangle +(1,1); 
\draw (0.5,2.5) node[circle,inner sep=0pt,fill=white] {3};
\draw (1.5,2.5) node[circle,inner sep=0pt,fill=white] {0.3};
\draw (2.5,2.5) node[circle,inner sep=0pt,fill=white] {0.7};
\draw (3.5,2.5) node[circle,inner sep=0pt,fill=white] {0.9};
\fill[pattern=vertical lines C,pattern color=red] (4,2) rectangle +(1,1); 
\draw (4.5,2.5) node[circle,inner sep=0pt,fill=white] {-3};
\fill[pattern=vertical lines B,pattern color=red] (5,2) rectangle +(1,1); 
\draw (5.5,2.5) node[circle,inner sep=0pt,fill=white] {-2};
\fill[pattern=vertical lines A,pattern color=red] (6,2) rectangle +(1,1); 
\draw (6.5,2.5) node[circle,inner sep=0pt,fill=white] {-1};
\draw (0.5,3.5) node[circle,inner sep=0pt,fill=white] {0.2};
\draw (1.5,3.5) node[circle,inner sep=0pt,fill=white] {0.4};
\fill[pattern=vertical lines B,pattern color=red] (2,3) rectangle +(1,1); 
\draw (2.5,3.5) node[circle,inner sep=0pt,fill=white] {-2};
\fill[pattern=vertical lines C,pattern color=red] (0,4) rectangle +(1,1); 
\draw (0.5,4.5) node[circle,inner sep=0pt,fill=white] {-3};
\fill[pattern=vertical lines B,pattern color=red] (1,4) rectangle +(1,1); 
\draw (1.5,4.5) node[circle,inner sep=0pt,fill=white] {-2};
\fill[pattern=vertical lines C,pattern color=red] (0,5) rectangle +(1,1); 
\draw (0.5,5.5) node[circle,inner sep=0pt,fill=white] {-3};
\fill[pattern=vertical lines A,pattern color=red] (1,5) rectangle +(1,1); 
\draw (1.5,5.5) node[circle,inner sep=0pt,fill=white] {-1};
\fill[pattern=vertical lines C,pattern color=red] (0,6) rectangle +(1,1); 
\draw (0.5,6.5) node[circle,inner sep=0pt,fill=white] {-3};
\draw[very thick](0,0) -- (8,0) -- (8,2) -- (7,2) -- (7,3) -- (3,3) -- (3,4) -- (2,4) -- (2,6) -- (1,6) -- (1,7) -- (0,7) -- cycle;
\begin{scope}[]
\clip[](0,0) -- (8,0) -- (8,2) -- (7,2) -- (7,3) -- (3,3) -- (3,4) -- (2,4) -- (2,6) -- (1,6) -- (1,7) -- (0,7) -- cycle;
\draw (0,0) grid (50,50);
\end{scope}

\fill[pattern=horizontal lines B,pattern color=blue] (7,-2) rectangle +(1,1); 
\draw (7,-2) rectangle +(1,1);
\draw (7.5,-1.5) node[circle,inner sep=0pt,fill=white] {2};
\draw[->,ultra thick] (7,-1.5) .. controls (6,-1.5) and (5.5,-1.5) .. (5.5,-0.5);
 
\end{tikzpicture}

%% file: programy-jeu-de-taquin/grafika-insertion-tableau2b.tex
\begin{tikzpicture}[scale=0.9]
\begin{scope} \clip (5,0) rectangle +(1,1); \draw[line width=6pt,green] (5,0) rectangle +(1,1); \end{scope}
\begin{scope} \clip (3,1) rectangle +(1,1); \draw[line width=6pt,green] (3,1) rectangle +(1,1); \end{scope}
\begin{scope} \clip (2,2) rectangle +(1,1); \draw[line width=6pt,green] (2,2) rectangle +(1,1); \end{scope}
\begin{scope} \clip (2,3) rectangle +(1,1); \draw[line width=6pt,green] (2,3) rectangle +(1,1); \end{scope}
\begin{scope} \clip (1,4) rectangle +(1,1); \draw[line width=6pt,green] (1,4) rectangle +(1,1); \end{scope}
\begin{scope} \clip (1,5) rectangle +(1,1); \draw[line width=6pt,green] (1,5) rectangle +(1,1); \end{scope}
\begin{scope} \clip (1,6) rectangle +(1,1); \draw[line width=6pt,green] (1,6) rectangle +(1,1); \end{scope}
\fill[pattern=horizontal lines A,pattern color=blue] (0,0) rectangle +(1,1); 
\draw (0.5,0.5) node[circle,inner sep=0pt,fill=white] {1};
\fill[pattern=horizontal lines A,pattern color=blue] (1,0) rectangle +(1,1); 
\draw (1.5,0.5) node[circle,inner sep=0pt,fill=white] {1};
\fill[pattern=horizontal lines A,pattern color=blue] (2,0) rectangle +(1,1); 
\draw (2.5,0.5) node[circle,inner sep=0pt,fill=white] {1};
\fill[pattern=horizontal lines B,pattern color=blue] (3,0) rectangle +(1,1); 
\draw (3.5,0.5) node[circle,inner sep=0pt,fill=white] {2};
\fill[pattern=horizontal lines B,pattern color=blue] (4,0) rectangle +(1,1); 
\draw (4.5,0.5) node[circle,inner sep=0pt,fill=white] {2};
\fill[pattern=horizontal lines B,pattern color=blue] (5,0) rectangle +(1,1); 
\draw (5.5,0.5) node[circle,inner sep=0pt,fill=white] {2};
\draw (6.5,0.5) node[circle,inner sep=0pt,fill=white] {0.1};
\draw (7.5,0.5) node[circle,inner sep=0pt,fill=white] {0.8};
\fill[pattern=horizontal lines B,pattern color=blue] (0,1) rectangle +(1,1); 
\draw (0.5,1.5) node[circle,inner sep=0pt,fill=white] {2};
\fill[pattern=horizontal lines C,pattern color=blue] (1,1) rectangle +(1,1); 
\draw (1.5,1.5) node[circle,inner sep=0pt,fill=white] {3};
\fill[pattern=horizontal lines C,pattern color=blue] (2,1) rectangle +(1,1); 
\draw (2.5,1.5) node[circle,inner sep=0pt,fill=white] {3};
\fill[pattern=horizontal lines C,pattern color=blue] (3,1) rectangle +(1,1); 
\draw (3.5,1.5) node[circle,inner sep=0pt,fill=white] {3};
\draw (4.5,1.5) node[circle,inner sep=0pt,fill=white] {0.6};
\fill[pattern=vertical lines C,pattern color=red] (5,1) rectangle +(1,1); 
\draw (5.5,1.5) node[circle,inner sep=0pt,fill=white] {-3};
\fill[pattern=vertical lines B,pattern color=red] (6,1) rectangle +(1,1); 
\draw (6.5,1.5) node[circle,inner sep=0pt,fill=white] {-2};
\fill[pattern=vertical lines A,pattern color=red] (7,1) rectangle +(1,1); 
\draw (7.5,1.5) node[circle,inner sep=0pt,fill=white] {-1};
\fill[pattern=horizontal lines C,pattern color=blue] (0,2) rectangle +(1,1); 
\draw (0.5,2.5) node[circle,inner sep=0pt,fill=white] {3};
\draw (1.5,2.5) node[circle,inner sep=0pt,fill=white] {0.3};
\draw (2.5,2.5) node[circle,inner sep=0pt,fill=white] {0.5};
\draw (3.5,2.5) node[circle,inner sep=0pt,fill=white] {0.9};
\fill[pattern=vertical lines C,pattern color=red] (4,2) rectangle +(1,1); 
\draw (4.5,2.5) node[circle,inner sep=0pt,fill=white] {-3};
\fill[pattern=vertical lines B,pattern color=red] (5,2) rectangle +(1,1); 
\draw (5.5,2.5) node[circle,inner sep=0pt,fill=white] {-2};
\fill[pattern=vertical lines A,pattern color=red] (6,2) rectangle +(1,1); 
\draw (6.5,2.5) node[circle,inner sep=0pt,fill=white] {-1};
\draw (0.5,3.5) node[circle,inner sep=0pt,fill=white] {0.2};
\draw (1.5,3.5) node[circle,inner sep=0pt,fill=white] {0.4};
\draw (2.5,3.5) node[circle,inner sep=0pt,fill=white] {0.7};
\fill[pattern=vertical lines C,pattern color=red] (0,4) rectangle +(1,1); 
\draw (0.5,4.5) node[circle,inner sep=0pt,fill=white] {-3};
\fill[pattern=vertical lines B,pattern color=red] (1,4) rectangle +(1,1); 
\draw (1.5,4.5) node[circle,inner sep=0pt,fill=white] {-2};
\fill[pattern=vertical lines C,pattern color=red] (0,5) rectangle +(1,1); 
\draw (0.5,5.5) node[circle,inner sep=0pt,fill=white] {-3};
\fill[pattern=vertical lines B,pattern color=red] (1,5) rectangle +(1,1); 
\draw (1.5,5.5) node[circle,inner sep=0pt,fill=white] {-2};
\fill[pattern=vertical lines C,pattern color=red] (0,6) rectangle +(1,1); 
\draw (0.5,6.5) node[circle,inner sep=0pt,fill=white] {-3};
\fill[pattern=vertical lines A,pattern color=red] (1,6) rectangle +(1,1); 
\draw (1.5,6.5) node[circle,inner sep=0pt,fill=white] {-1};
\draw[very thick](0,0) -- (8,0) -- (8,2) -- (7,2) -- (7,3) -- (3,3) -- (3,4) -- (2,4) -- (2,7) -- (0,7) -- cycle;
\begin{scope}[]
\clip[](0,0) -- (8,0) -- (8,2) -- (7,2) -- (7,3) -- (3,3) -- (3,4) -- (2,4) -- (2,7) -- (0,7) -- cycle;
\draw (0,0) grid (50,50);
\end{scope}
 
\end{tikzpicture}

%% file: rsk-thoma.bbl
\begin{thebibliography}{Rom15}

\bibitem[BR85]{BereleRemmel1985}
A.~Berele and J.~B. Remmel.
\newblock Hook flag characters and their combinatorics.
\newblock {\em J. Pure Appl. Algebra}, 35(3):225--245, 1985.

\bibitem[BR87]{RereleRegev1987}
A.~Berele and A.~Regev.
\newblock Hook {Y}oung diagrams with applications to combinatorics and to
  representations of {L}ie superalgebras.
\newblock {\em Adv. in Math.}, 64(2):118--175, 1987.

\bibitem[Ful97]{Fulton1997}
William Fulton.
\newblock {\em Young tableaux}, volume~35 of {\em London Mathematical Society
  Student Texts}.
\newblock Cambridge University Press, Cambridge, 1997.
\newblock With applications to representation theory and geometry.

\bibitem[Ful02]{Fulman2002}
Jason Fulman.
\newblock Applications of symmetric functions to cycle and increasing
  subsequence structure after shuffles.
\newblock {\em J. Algebraic Combin.}, 16(2):165--194, 2002.

\bibitem[Knu73]{KnuthVol3}
Donald~E. Knuth.
\newblock {\em The art of computer programming. {V}olume 3}.
\newblock Addison-Wesley Publishing Co., Reading, Mass.-London-Don Mills, Ont.,
  1973.
\newblock Sorting and searching, Addison-Wesley Series in Computer Science and
  Information Processing.

\bibitem[KV86]{KerovVershik1986}
Sergei~V. Kerov and Anatol~M. Vershik.
\newblock The characters of the infinite symmetric group and probability
  properties of the {R}obinson-{S}chensted-{K}nuth algorithm.
\newblock {\em SIAM J. Algebraic Discrete Methods}, 7(1):116--124, 1986.

\bibitem[LS77]{loganshepp}
B.~F. Logan and L.~A. Shepp.
\newblock A variational problem for random {Y}oung tableaux.
\newblock {\em Advances in Math.}, 26(2):206--222, 1977.

\bibitem[O'C03]{OConnell2003}
Neil O'Connell.
\newblock A path-transformation for random walks and the {R}obinson-{S}chensted
  correspondence.
\newblock {\em Trans. Amer. Math. Soc.}, 355(9):3669--3697 (electronic), 2003.

\bibitem[OY02]{OConnellYor2002}
Neil O'Connell and Marc Yor.
\newblock A representation for non-colliding random walks.
\newblock {\em Electron. Comm. Probab.}, 7:1--12 (electronic), 2002.

\bibitem[Rom15]{Romik2013}
Dan Romik.
\newblock {\em The surprising mathematics of longest increasing subsequences}.
\newblock Institute of Mathematical Statistics Textbooks. Cambridge University
  Press, New York, 2015.

\bibitem[R{\'S}15]{RomikSniady2011}
Dan Romik and Piotr {\'S}niady.
\newblock Jeu de taquin dynamics on infinite {Y}oung tableaux and second class
  particles.
\newblock {\em Ann. Probab.}, 43(2):682--737, 2015.

\bibitem[Sag01]{Sagan2001}
Bruce~E. Sagan.
\newblock {\em The symmetric group}, volume 203 of {\em Graduate Texts in
  Mathematics}.
\newblock Springer-Verlag, New York, second edition, 2001.
\newblock Representations, combinatorial algorithms, and symmetric functions.

\bibitem[Sch63]{Schutzenberger1963}
M.~P. Sch{\"u}tzenberger.
\newblock Quelques remarques sur une construction de {S}chensted.
\newblock {\em Math. Scand.}, 12:117--128, 1963.

\bibitem[Sch77]{schutzenberger}
M.-P. Sch{\"u}tzenberger.
\newblock La correspondance de {R}obinson.
\newblock In {\em Combinatoire et repr\'esentation du groupe sym\'etrique
  ({A}ctes {T}able {R}onde {CNRS}, {U}niv. {L}ouis-{P}asteur {S}trasbourg,
  {S}trasbourg, 1976)}, pages 59--113. Lecture Notes in Math., Vol. 579.
  Springer, Berlin, 1977.

\bibitem[Sil08]{Silva2008}
C.~E. Silva.
\newblock {\em Invitation to ergodic theory}, volume~42 of {\em Student
  Mathematical Library}.
\newblock American Mathematical Society, Providence, RI, 2008.

\bibitem[Sta99]{StanleyVol2}
Richard~P. Stanley.
\newblock {\em Enumerative combinatorics. {V}ol. 2}, volume~62 of {\em
  Cambridge Studies in Advanced Mathematics}.
\newblock Cambridge University Press, Cambridge, 1999.
\newblock With a foreword by Gian-Carlo Rota and appendix 1 by Sergey Fomin.

\bibitem[Tho64]{Thoma1964}
Elmar Thoma.
\newblock Die unzerlegbaren, positiv-definiten {K}lassenfunktionen der
  abz\"ahlbar unendlichen, symmetrischen {G}ruppe.
\newblock {\em Math. Z.}, 85:40--61, 1964.

\bibitem[VK77]{vershikkerov1}
A.~M. Vershik and S.~V. Kerov.
\newblock Asymptotics of the {P}lancherel measure of the symmetric group and
  the limit form of {Y}oung tableaux.
\newblock {\em Soviet Math. Dokl.}, 18:527--531, 1977.

\bibitem[VK81]{VershikKerov1981}
A.~M. Vershik and S.~V. Kerov.
\newblock Asymptotic theory of the characters of a symmetric group.
\newblock {\em Funktsional. Anal. i Prilozhen.}, 15(4):15--27, 96, 1981.

\bibitem[VK85]{vershikkerov2}
A.~M. Vershik and S.~V. Kerov.
\newblock Asymptotic of the largest and the typical dimensions of irreducible
  representations of a symmetric group.
\newblock {\em Functional Anal. Appl.}, 19(1):21--31, 1985.

\end{thebibliography}
